\documentclass[11pt]{amsart}
\usepackage{graphicx}
\usepackage{amssymb}
\usepackage{amsmath}
\usepackage{amsthm}
\usepackage{amsfonts}
\usepackage{bbm}
\usepackage{amsrefs}
\usepackage{tikz}
\usepackage{tikz-cd}
\usepackage{setspace,kantlipsum}
\usepackage[toc,page]{appendix}
\usepackage[utf8]{inputenc}
\usepackage{amssymb,amsmath,amsthm}
\usepackage{graphicx}
\usepackage{amsmath,amsfonts,amssymb}
\usepackage{fullpage}
\usepackage{float}
\usepackage{yhmath}

\newtheorem{thm}{Theorem}[section]
\newtheorem{lemma}{Lemma}[section]
\newtheorem{col}{Corollary}[section]
\newtheorem{prop}{Proposition}[section]

\theoremstyle{definition}
\newtheorem{defn}{Definition}[section]
\newtheorem{ex}{Example}[section]
\newtheorem{rmk}{Remark}[section]
\newtheorem{prob}{Problem}[section]

\begin{document}

\author{Weichen Gao\\\\With an Appendix by Marius Junge, David Gao}
\title{Relative Embeddability of von Neumann Algebras and Amalgamated Free Products}
\date{}
\maketitle

\leftskip 0.5in \rightskip 0.5in
\noindent\textsc{Abstract.} In this paper we study conditions under which, for an inclusion of finite von Neumann algebras $N \subseteq M$, we have the reduced amalgamated free product $\ast_N M$ is embeddable into $(R \bar{\otimes} N_1)^\omega$ for some other finite von Neumann algebra $N_1$, where $R$ is the hyperfinite $\textrm{II}_1$ factor.\par
\leftskip 0in \rightskip 0in

\medskip

\section{\sc Introduction}

\medskip

Connes, in a paper from 1976 [Con76], formulated the now-famous embedding problem: Whether, given any finite von Neumann algebra $N$ with a separable predual, it can be embedded into the ultrapower $R^\omega$ of the hyperfinite $\textrm{II}_1$ factor $R$. This problem was recently answered negatively, in a paper using tools from quantum computational complexity theory [JNV+20]. Therefore, it is now an important task to identify conditions implying that a finite von Neumann algebra $N$ is embeddable in $R^\omega$.

In this paper we study embeddability of amalgamated free products. For example, it seems to be well-known that given $N \subseteq M \hookrightarrow R^\omega$ and $N$ is hyperfinite, then $\ast_N M$ remains embeddable. (See, for example, [Jun05, Theorem 7.15], which shows that this holds when $N = \mathbb{C}$.) Since it is hard to trace in the literature, we include a proof of this fact as Lemma 2.3. However, the problem of given $M$ and $N$ embeddable, whether $\ast_N M$ is embeddable as well remains open, and we, in this paper, investigate some new conditions under which this would hold.

In this investigation, it turns out that the ``position" of $N$ inside $M$ is important. As an example, suppose $N \subseteq M$ is an inclusion of finite von Neumann algebras and we also have an embedding $\pi: M \hookrightarrow \mathbb{M}_n \otimes N$. (This would happen, for example, when $M$ is a factor and $N$ is a subfactor of finite index. See [Tak03, Chapter XIX, Section 2].) If we also have $\pi\restriction_N: N \rightarrow \mathbb{M}_n \otimes N$ is the canonical inclusion, then it would follow that $\ast_N M \hookrightarrow (R \bar{\otimes} N)^\omega$ by using the following argument:
\begin{equation*}
    \ast_N M \hookrightarrow \ast_N (\mathbb{M}_n \otimes N) \hookrightarrow (\ast_{\mathbb{C}} \mathbb{M}_n) \bar{\otimes} N \hookrightarrow R^\omega \bar{\otimes} N \hookrightarrow (R \bar{\otimes} N)^\omega
\end{equation*}
where the third embedding uses $\ast_{\mathbb{C}} \mathbb{M}_n \hookrightarrow R^\omega$, which follows from [Jun05, Theorem 7.15]. It follows that if $N$ is embeddable, then $\ast_N M$ is embeddable as well. However, in general, it is not necessarily true that $\pi\restriction_N$ is the canonical inclusion, in which case the second algebra in the chain of inclusions above should be $\ast_{\pi(N)} \mathbb{M}_n \otimes N$ instead of $\ast_N (\mathbb{M}_n \otimes N)$, so the argument above will not apply.

Let us indicate a natural scenario where $\pi$ restricted to $N$ is not the canonical inclusion. Consider discrete groups $H \triangleleft G$ s.t. $|G:H| = n < \infty$. Then,
\begin{equation*}
l_2(G) = l_2(G/H) \otimes l_2(H) = l_2^n \otimes l_2(H)    
\end{equation*}

So $\mathbb{B}(l_2(G)) = \mathbb{M}_n \otimes \mathbb{B}(l_2(H))$. Under this identification, it is not hard to verify that $L(G)$ is sent into $\mathbb{M}_n \otimes L(H)$. (When $G$ is an i.c.c. group, this is an example of a subfactor of finite index. See [Tak03, Chapter XIX, Example 2.4].) However, the composite inclusion $L(H) \hookrightarrow L(G) \hookrightarrow \mathbb{M}_n \otimes L(H)$ is not necessarily the canonical inclusion. Indeed, any given $h \in L(H)$ is sent into $l_\infty^n \otimes L(H)$, with entries on the diagonal of the form $\widetilde{g}^{-1}h\widetilde{g}$, where $\widetilde{g}$ are representatives of cosets in $G/H$. $\widetilde{g}^{-1}h\widetilde{g}$ is not necessarily $h$ itself, so the inclusion is not canonical.

However, in this case we actually would have $L(G)$ embeds into $(R \bar{\otimes} L(H))^\omega$, but for a different reason. This is due to the following diagram being a commuting square,

\begin{center}
\begin{tikzcd}[node distance = 1.8cm]
    {\mathbb{M}_n \otimes L(H)} \arrow[hookleftarrow]{r}\arrow[hookleftarrow]{d}
        & {\pi(L(G))} \arrow[hookleftarrow]{d} \\
   {l_\infty^n \otimes L(H)}  \arrow[hookleftarrow]{r}
        & {\pi(L(H))}
\end{tikzcd}
\end{center}

So,
\begin{equation*}
    \ast_{L(H)} L(G) \hookrightarrow \ast_{l_\infty^n \otimes L(H)} \mathbb{M}_n \otimes L(H) \hookrightarrow (\ast_{l_\infty^n} \mathbb{M}_n) \bar{\otimes} L(H) \hookrightarrow R^\omega \bar{\otimes} L(H) \hookrightarrow (R \bar{\otimes} L(H))^\omega
\end{equation*}

(For the first inclusion, see Lemma 2.1.) These demonstrate the importance of the ``position" of $N$ in $M$, in the sense of some commuting square condition like the above diagram. This leads to the following definition:

\begin{defn}
Let $M$ be a finite von Neumann algebra, $N \subseteq M$ its weakly closed subalgebra, we say $N \subseteq M$ is \textit{relatively embeddable} with respect to $N_1$ (written as $N \subseteq M$ is RE/$N_1$) if there exists an ultrafilter $\omega$ and a tracial embedding $\pi: M \rightarrow \prod_{\omega} A_k \bar{\otimes} N_1$ s.t. the following diagram is a commuting square,

\begin{center}
\begin{tikzcd}[node distance = 1.8cm]
    {\prod_{\omega} A_k \bar{\otimes} N_1} \arrow[hookleftarrow]{r}\arrow[hookleftarrow]{d}
        & \pi(M) \arrow[hookleftarrow]{d} \\
   {\prod_{\omega} B_k \bar{\otimes} N_1}  \arrow[hookleftarrow]{r}
        & \pi(N)
\end{tikzcd}
\end{center}
where $A_k$ are QWEP finite von Neumann algebras and $B_k \subseteq A_k$ are their hyperfinite subalgebras.
\end{defn}

Our main result is as follows:

\begin{thm}
If $N \subseteq M$ is RE/$N_1$, then $\ast_N M$ tracially embeds into an ultrapower of $R \bar{\otimes} N_1$.
\end{thm}

In the Appendix, we shall show that with some further assumptions or if $\ast_N M$ is slightly enlarged, then a converse to this theorem is true.

We shall also identify some concrete examples of the RE condition related to discrete groups. In particular, we are interested in inclusions of the form $N \rtimes_{vN} H \subseteq N \rtimes_{vN} G$ where $N$ is a finite von Neumann algebra, $G$ is a discrete group acting through trace-preserving automorphisms on $N$, and $H < G$. We will prove:

\begin{thm}
Let $N$ be a finite von Neumann algebra, $G$ be a group acting through trace-preserving automorphisms on $N$, and $H \triangleleft G$.

\medskip

1. If $H$ is co-amenable in $G$, then $N \rtimes_{vN} H \subseteq N \rtimes_{vN} G$ is RE/$N \rtimes_{vN} H$;

\medskip

2. If $G/H$ is hyperlinear, then $N \rtimes_{vN} H \subseteq N \rtimes_{vN} G$ is RE/$N \rtimes_{vN} G$.
\end{thm}

Using the concept of subgroup separability applied to free groups, it will also be shown that,

\begin{prop}
$\ast_H G$ is a hyperlinear group whenever $G$ is a free group and $H < G$.
\end{prop}

In the paper, we shall prove a more general result of which the above two are special cases.

In the second item of Theorem 1.2 above, we could only obtain RE over the larger algebra $N \rtimes_{vN} G$ instead of over the smaller algebra $N \rtimes_{vN} H$, as opposed to the situation with amenable quotients in the first item of Theorem 1.2. This is due to the fact that $G$ acts non-trivially on $H$ and $N$ and these actions are not easily approximated through matrices without the aid of Følner sequences. More generally, when we consider the question of embeddability, in the trivial case where $\mathbb{C} = N \subseteq M$, $N$ commutes with everything in $M$, and so the remainder of $M$ does not ``act" on $N$ and the embedding is, in some sense, only an issue of using finite-dimensional algebras to approximate the ``quotient structure" of $M$ over $N$. And indeed in this case we have $\ast_{\mathbb{C}} M$ is QWEP whenever $M$ is. But in the nontrivial case, there is no reason for a general subalgebra $N$ to commute with everything else in $M$, so the embedding has to ``encode", approximately by finite-dimensional algebras, the ``action" of $M$ on $N$ in addition to the ``quotient structure". This presents surprising complexities. For a specific example that shows how a naïve approach can fail, see Example 5.6. In this paper, we shall introduce in the specific case of $N \rtimes_{vN} H \subseteq N \rtimes_{vN} G$, several concepts characterizing the actions of $G$ on $H$ and $N$ so that, with these extra conditions on the action, we could obtain RE over the smaller algebra $N \rtimes_{vN} H$. We will apply these results to obtain, among other things:

\begin{prop}
$L^\infty(X, \mu) \subseteq L^\infty(X, \mu) \rtimes_{vN} G$ is RE/$\mathbb{C}$ whenever $(X, \mu)$ is a standard probability space, $G$ is a hyperlinear group, and the p.m.p. action $G \curvearrowright X$ is profinite.
\end{prop}

\begin{prop}
$N^{\bar{\otimes} X} \subseteq N^{\bar{\otimes} X} \rtimes_{vN} G$ is RE/$N^{\bar{\otimes}\infty}$ whenever $X$ is a countable discrete set, $G$ is a free group, and the action of $G$ on $N^{\bar{\otimes} X}$ is induced by an action $G \curvearrowright X$.
\end{prop}

This discussion of the actions of $G$ on $N$ raises several interesting questions that are still open. In particular, while we obtained a characterization of when $R \rtimes_{vN} G$ is QWEP given a hyperlinear group $G$ acting on $R$, it is not immediate whether certain natural actions satisfy this condition. Specifically, apart from the case where $G$ is either amenable or a free group, it is still open whether $R^{\bar{\otimes} G} \rtimes_{vN} G$ is QWEP or not. (Here, the action of $G$ on $R^{\bar{\otimes} G}$ is induced by the left multiplication action of $G$ on itself.) More research is needed to answer these questions, which might shine more light on the nature of the QWEP property.

\medskip

\section{\sc Definition of Relative Embeddability and Preliminaries}

\medskip

We shall now define the notion of relative embeddability mentioned in the introduction and establish some relevant preliminary results.

\begin{rmk}
In the following we shall adopt the following notations and conventions. We shall use $\tau$ to denote the tracial state on a finite von Neumann algebra. When we have an inclusion of finite von Neumann algebras, $N \subseteq M$, $E_{M, N}$ shall be the trace-preserving conditional expectation of $M$ onto $N$. Unless otherwise indicated, all groups $G$ mentioned in this paper shall be assumed discrete and countable. Unless given some other name and when it does not cause confusion, the action of a group $G$ on a von Neumann algebra $N$ shall be referred to by $\alpha$. When $G$ acts on $N$, $N \rtimes_{vN} G$ shall mean the von Neumann algebra crossed product. When $N$ is finite, we shall use $Aut_\tau(N)$ to denote the group of trace-preserving automorphisms of $N$. The tracial amalgamated free product of finite von Neumann algebras $A_i$ over their shared subalgebra $B$ shall be denoted by $\ast_B^{i \in I} A_i$. When all $A_i$ equal $A$ and when it does not cause confusion, we shall use $\ast_B A$ as a shorthand for $\ast_B^{i \in I} A_i$. When we use this shorthand, unless otherwise indicated, it shall be assumed that the index set $I$ can be any countable set.
\end{rmk}

\begin{defn}
Let $\hat{M}$ be a finite von Neumann algebra. Let $N$, $\hat{N}$, and $M$ be its weakly closed subalgebras. A \textit{commuting square} is a commutative diagram

\begin{center}
\begin{tikzcd}[node distance = 1.8cm]
    {\hat{M}} \arrow[hookleftarrow]{r}\arrow[hookleftarrow]{d}
        & M \arrow[hookleftarrow]{d} \\
    {\hat{N}} \arrow[hookleftarrow]{r}
        & N
\end{tikzcd}
\end{center}
where all the arrows are the inclusion maps. Furthermore, we require $E_{\hat{M}, N} = E_{\hat{M}, \hat{N}} E_{\hat{M}, M}$; or, equivalently, $E_{\hat{M}, N} = E_{\hat{M}, M} E_{\hat{M}, \hat{N}}$.
\end{defn}

\begin{defn}
Let $M$ be a finite von Neumann algebra, $N \subseteq M$ its weakly closed subalgebra, we say $N \subseteq M$ is \textit{relatively embeddable} with respect to $N_1$ (written as $N \subseteq M$ is RE/$N_1$) if there exists an ultrafilter $\omega$ and a tracial embedding $\pi: M \rightarrow \prod_{\omega} A_k \bar{\otimes} N_1$ s.t. the following diagram is a commuting square,

\begin{center}
\begin{tikzcd}[node distance = 1.8cm]
    {\prod_{\omega} A_k \bar{\otimes} N_1} \arrow[hookleftarrow]{r}\arrow[hookleftarrow]{d}
        & \pi(M) \arrow[hookleftarrow]{d} \\
   {\prod_{\omega} B_k \bar{\otimes} N_1}  \arrow[hookleftarrow]{r}
        & \pi(N)
\end{tikzcd}
\end{center}
where $A_k$ are QWEP finite von Neumann algebras and $B_k \subseteq A_k$ are their hyperfinite subalgebras.
\end{defn}

We first note the following easy fact:

\begin{prop}
If $N \subseteq M$ is RE/$N_1$, then $M$ tracially embeds into $(R \bar{\otimes} N_1)^\omega$, where $\omega$ is an ultrafilter and $R$ is the hyperfinite $\textrm{II}_1$ factor.
\end{prop}

\begin{proof} By definition of RE/$N_1$, we have a tracial embedding $M$ into $\prod_{\omega'} A_k \bar{\otimes} N_1$ where $\omega'$ is some ultrafilter and $A_k$ are QWEP. Hence, $A_k$ tracially embeds into $R^{\omega'}$. We thus have an embedding,
\begin{equation*}
    \prod_{k \rightarrow \omega'} A_k \bar{\otimes} N_1 \hookrightarrow \prod_{k \rightarrow \omega'} R^{\omega'} \bar{\otimes} N_1
\end{equation*}

Note that $R^{\omega'} \bar{\otimes} N_1$ naturally embeds into $(R \bar{\otimes} N_1)^{\omega'}$ by sending $(r_k)^\circ \otimes x$ to $(r_k \otimes x)^\circ \in (R \bar{\otimes} N_1)^{\omega'}$. As such, we have an embedding of $M$ into $((R \bar{\otimes} N_1)^{\omega'})^{\omega'}$. To conclude the proof, we observe that
\begin{equation*}
    ((R \bar{\otimes} N_1)^{\omega'})^{\omega'} = (R \bar{\otimes} N)^\omega
\end{equation*}
using the ultrafilter $\omega$ s.t. $A \subseteq \mathbb{N}^2$ satisfies $A \in \omega = (\omega')^2$ iff $\{j \in \mathbb{N}: \{k \in \mathbb{N}: (j, k) \in A\} \in \omega'\} \in \omega'$.
\end{proof}

In the following we shall keep the notation $(\omega')^2$ for the iterated ultraproduct defined above. The definition of the RE condition is inspired by the following fact:

\begin{thm}
If $N \subseteq M$ is RE/$N_1$, then $N \subseteq \ast_N M$ is RE/$N_1$.
\end{thm}

Combining this theorem with Proposition 2.1, we immediately obtain,

\begin{col}
If $N \subseteq M$ is RE/$N_1$, then $\ast_N M$ tracially embeds into $(R \bar{\otimes} N_1)^\omega$, where $\omega$ is an ultrafilter.
\end{col}

To prove Theorem 2.1, we shall need the following lemmas:

\begin{lemma}
Let $\hat{M}$ be a finite von Neumann algebra. Let $N$, $\hat{N}$, and $M$ be its weakly closed subalgebras. If the following diagram is a commuting square,

\begin{center}
\begin{tikzcd}[node distance = 1.8cm]
    {\hat{M}} \arrow[hookleftarrow]{r}\arrow[hookleftarrow]{d}
        & M \arrow[hookleftarrow]{d} \\
    {\hat{N}} \arrow[hookleftarrow]{r}
        & N
\end{tikzcd}
\end{center}

Then there exists a unique inclusion map $\iota: \ast_{N} M \rightarrow \ast_{\hat{N}} \hat{M}$ where the index sets for the amalgamated free products shall be the same countable set $I$, s.t. $\iota(m_1^{i_1} m_2^{i_2} \cdots m_n^{i_n}) = m_1^{i_1} m_2^{i_2} \cdots m_n^{i_n}$, where $i_k \in I$, $m_k^i \in M_i = M \subseteq \ast_{N} M$, $E_{M, N}(m_k^{i_k}) = 0$, and $i_1 \neq i_2 \neq \cdots \neq i_n$, and where we interpret $m_k^i$ on the right-hand side as in $\hat{M}_i = \hat{M} \subseteq \ast_{\hat{N}} \hat{M}$.
\end{lemma}

\begin{proof} Uniqueness is clear. Now, recall that $\ast_{N} M$ is the tracial amalgamated free product, and we shall use $\check{\ast}_{N} M$ to denote the universal amalgamated free product of countable copies of $M$ over $N$.

Let $\varphi_i: M_i \rightarrow \ast_{\hat{N}} \hat{M}$ be the composition of the inclusion maps $M_i \subseteq \hat{M}_i$ and $\hat{M}_i \subseteq \ast_{\hat{N}} \hat{M}$. By definition of the universal amalgamated free product, we may define the *-homomorphism $\varphi = \ast_N^{i \in I} \varphi_i: \check{\ast}_{N} M \rightarrow \ast_{\hat{N}} \hat{M}$ s.t. $\varphi(m_1^{i_1} m_2^{i_2} \cdots m_n^{i_n}) = \varphi_{i_1}(m_1^{i_1}) \varphi_{i_1}(m_2^{i_2}) \cdots \varphi_{i_1}(m_n^{i_n})$, where $i_k \in I$, $m_k^{i_k} \in M_{i_k}$, $E_{M, N}(m_k^{i_k}) = 0$, and $i_1 \neq i_2 \neq \cdots \neq i_n$.

We now demonstrate that $\varphi$ is trace-preserving. To be precise, the trace on $\check{\ast}_{N} M$, $\widetilde{\tau_1}$, is the composition of the canonical quotient map $\pi: \check{\ast}_{N} M \rightarrow \ast_{N} M$ and the canonical trace $\tau_1$ on $\ast_{N} M$. The trace on $\ast_{\hat{N}} \hat{M}$ is the canonical trace $\tau_2$. For any $n \in N \subseteq \check{\ast}_{N} M$, by definition $\varphi(n) = n \in N \subseteq \hat{N} \subseteq \ast_{\hat{N}} \hat{M}$, so,
\begin{equation*}
    \tau_2(\varphi(n)) = \tau_{\hat{N}}(n) = \tau_N(n) = \widetilde{\tau_1}(n)
\end{equation*}

Now, the commuting square condition implies that $E_{\hat{M}, \hat{N}} E_{\hat{M}, M} = E_{\hat{M}, N}$. Note that for any $m \in M_i = M$, $i \in I$, $E_{M, N}(m) = 0$, we have,
\begin{equation*}
    E_{\hat{M}, \hat{N}}(m) = E_{\hat{M}, \hat{N}}(E_{\hat{M}, M}(m)) = E_{\hat{M}, N}(m) = E_{M, N}(m) = 0
\end{equation*}

This demonstrates that, given $i_k \in I$, $m_k^{i_k} \in M_{i_k}$, $E_{M, N}(m_k^{i_k}) = 0$, and $i_1 \neq i_2 \neq \cdots \neq i_n$, for $\varphi(m_1^{i_1} m_2^{i_2} \cdots m_n^{i_n}) = \varphi_{i_1}(m_1^{i_1}) \varphi_{i_2}(m_2^{i_2}) \cdots \varphi_{i_n}(m_n^{i_n})$, each term $\varphi_{i_k}(m_k^{i_k})$ satisfies $E_{\hat{M}, \hat{N}}(\varphi_{i_k}(m_k^{i_k})) = E_{\hat{M}, \hat{N}}(a^k_{i_k}) = 0$, so by definition,
\begin{equation*}
    \tau_2(\varphi(m_1^{i_1} m_2^{i_2} \cdots m_n^{i_n})) = 0 = \widetilde{\tau_1}(m_1^{i_1} m_2^{i_2} \cdots m_n^{i_n})
\end{equation*}

Taking linear combinations and norm limits now conclude the proof that $\varphi$ is trace-preserving.

Finally, as we have noted $\tau_1 = \tau_N \circ E$ is a faithful tracial state on $\ast_{N} M$. Hence, the kernel of the canonical quotient map $\pi: \check{\ast}_{N} M \rightarrow \ast_{N} M$ is simply $\{\xi \in \check{\ast}_{N} M: \tau_1(\pi(\xi)^*\pi(\xi)) = 0\} = \{\xi \in \check{\ast}_{N} M: \widetilde{\tau_1}(\xi^*\xi) = 0\}$. But now for any $\xi \in \check{\ast}_{N} M$ such that $\widetilde{\tau_1}(\xi^*\xi) = 0$, we have,
\begin{equation*}
\begin{split}
    \tau_2(\varphi(\xi)^*\varphi(\xi)) & = \tau_2(\varphi(\xi^*\xi))\\
    & = \widetilde{\tau_1}(\xi^*\xi)\\
    & = 0
\end{split}
\end{equation*}
where we have used the fact that $\varphi$ is trace-preserving. This shows that $\varphi(\textrm{ker}(\pi)) = 0$, so $\varphi$ factors through $\ast_{N} M$ and we may define $\iota: \ast_{N} M \rightarrow \ast_{\hat{N}} \hat{M}$ s.t. $\iota \circ \pi = \varphi$. It is then easy to verify that $\iota$, being trace-preserving, is the desired inclusion map.
\end{proof}

\begin{rmk}
The argument in this proof is presented with an abundance of details here so as to make it as clear as possible. Similar arguments will be used repeatedly throughout this paper and we shall omit most details when doing so.
\end{rmk}

\begin{lemma}
For any finite von Neumann algebra $N$, $N^{\bar{\otimes} \infty} \rtimes_{vN} \Sigma_\infty$ is a factor. Here, $\Sigma_\infty$ is the group of all finitely supported permutations of $\mathbb{N}$, which acts naturally on $N^{\bar{\otimes} \infty}$ via permutation.
\end{lemma}

\begin{proof} Let $x$ be any element of $\mathcal{Z}(N^{\bar{\otimes} \infty} \rtimes_{vN} \Sigma_\infty)$, $x$ has finite $L^2$ norm and hence belongs to $L^2(N^{\bar{\otimes} \infty} \rtimes_{vN} \Sigma_\infty)$, which is canonically isomorphic to $L^2(N^{\bar{\otimes} \infty}) \otimes l_2(\Sigma_\infty)$. The Fourier coefficient of $x$ associated with any non-unity $g \in \Sigma_\infty$ must be zero. This follows from the following consideration. We may write $x = \sum_{g \in \Sigma_\infty} x_g g$ where $x_g \in N^{\bar{\otimes} \infty}$ and the convergence is in $L^2$. Then, for any $g' \in \Sigma_\infty$,
\begin{equation*}
\begin{split}
    g'xg'^{-1} &= \sum_{g \in \Sigma_\infty} g' x_g g g'^{-1}\\
    &= \sum_{g \in \Sigma_\infty} \alpha_{g'}(x_g) g' g g'^{-1}\\
    &= \sum_{g \in \Sigma_\infty} \alpha_{g'}(x_{g'^{-1} g g'}) g
\end{split}
\end{equation*}

But $x \in \mathcal{Z}(N^{\bar{\otimes} \infty} \rtimes_{vN} \Sigma_\infty)$, so $g'xg'^{-1} = x$. In particular, $\alpha_{g'}(x_{g'^{-1} g g'}) = x_g$. But as the action is trace-preserving, we have,
\begin{equation*}
    \|x_{g'^{-1} g g'}\|_2 = \|x_g\|_2
\end{equation*}

Thus, if there is any non-unity $g \in \Sigma_\infty$ for which $x_g \neq 0$, $x_{g'^{-1} g g'}$ would have the same $L^2$ norm $\|x_g\|_2 > 0$ for all $g' \in \Sigma_\infty$. But $\Sigma_\infty$ is an i.c.c. group, so infinitely many Fourier coefficients of $x$ would have $L^2$ norm $\|x_g\|_2$, contradicting the finiteness of $\|x\|_2$.

Now, we know $x$ belongs to $N^{\bar{\otimes} \infty}$. Consider $\hat{N} = \langle\Sigma_\infty\rangle' \cap N^{\bar{\otimes} \infty}$, which we may now see contains $\mathcal{Z}(N^{\bar{\otimes} \infty} \rtimes_{vN} \Sigma_\infty)$. Note that the finite symmetric groups $\Sigma_n$ form an increasing sequence of subgroups of $\Sigma_\infty$ whose union is the whole group. Then we have, for all $y \in N^{\bar{\otimes} \infty}$.
\begin{equation*}
    E_{\hat{N}}(y) = \lim_{n \rightarrow \infty} \frac{1}{n!} \sum_{g \in \Sigma_n} gyg^{-1}
\end{equation*}
where the convergence is in $L^2$. In order to justify this, we first show the limit exists. To do so, we let $\{e_i\}_{i \in \mathbb{N}}$ be an orthonormal basis of $L^2(N)$ and we may assume $e_1 = \hat{1}$. Then finite tensors of them form an orthonormal basis $\mathcal{S}$ of $L^2(N^{\bar{\otimes} \infty})$. The conjugation action by any $g \in \Sigma_\infty$ simply permutes this basis. Let $y = e_{i_1} \otimes e_{i_2} \otimes \cdots \otimes e_{i_m} \otimes 1 \otimes 1 \otimes \cdots$ be an element of $\mathcal{S}$ and we define $\mathrm{supp}(y)$ be the number of $1 \leq k \leq m$ s.t. $e_{i_k} \neq \hat{1}$. Note that this is invariant under the action of $\Sigma_\infty$. Now, assume $y \neq \hat{1}$, and we may assume WLOG that $e_{i_m} \neq \hat{1}$. For any $n \geq m$, $g \in \Sigma_n$, we consider when $gyg^{-1} = y'$, where $y'$ is any element of $\mathcal{S}$. First, we have $\mathrm{supp}(y) = \mathrm{supp}(y')$, and we may disregard any $y'$ with a different $\mathrm{supp}$. This means there are only $\mathrm{supp}(y)$ many tensor components of $y'$ that are non-unity. If $gyg^{-1} = y'$, then as $e_{i_m} \neq \hat{1}$, $g$ must send $m$ to one of these $\mathrm{supp}(y)$ many components. The proportion of elements of $\Sigma_n$ that do this is $\frac{\mathrm{supp}(y)}{n}$, which is less than or equal to $\frac{m}{n}$. Regarding $\frac{1}{n!} \sum_{g \in \Sigma_n} gyg^{-1}$ as a function on $\mathcal{S}$ (since it is a finite linear combination of elements of $\mathcal{S}$), we have,
\begin{equation*}
\begin{split}
    \|\frac{1}{n!} \sum_{g \in \Sigma_n} gyg^{-1}\|_{l_\infty(\mathcal{S})} &= \frac{1}{n!} \max \{\#\textrm{ of }g \in \Sigma_n\textrm{ s.t. }gyg^{-1} = y': y' \in \mathcal{S}\}\\
    &= \max \{\textrm{the proportion of }g \in \Sigma_n\textrm{ s.t. }gyg^{-1} = y': y' \in \mathcal{S}\}\\
    &\leq \frac{m}{n}
\end{split}
\end{equation*}
while it is clear that $\|\frac{1}{n!} \sum_{g \in \Sigma_n} gyg^{-1}\|_{l_1(\mathcal{S})} = 1$. Therefore,
\begin{equation*}
\begin{split}
    \|\frac{1}{n!} \sum_{g \in \Sigma_n} gyg^{-1}\|^2_2 &= \|\frac{1}{n!} \sum_{g \in \Sigma_n} gyg^{-1}\|^2_{l_2(\mathcal{S})}\\
    &\leq \|\frac{1}{n!} \sum_{g \in \Sigma_n} gyg^{-1}\|_{l_1(\mathcal{S})}\|\frac{1}{n!} \sum_{g \in \Sigma_n} gyg^{-1}\|_{l_\infty(\mathcal{S})}\\
    &\leq 1 \cdot \frac{m}{n}\\
    &= \frac{m}{n}
\end{split}
\end{equation*}

Hence, $\lim_{n \rightarrow \infty} \sum_{g \in \Sigma_n} \frac{1}{n!} gyg^{-1} = 0$. If $y = \hat{1}$, then clearly $\lim_{n \rightarrow \infty} \sum_{g \in \Sigma_n} \frac{1}{n!} gyg^{-1} = \hat{1}$. Thus, the limit exists for all finite sums of members of $\mathcal{S}$. One may easily see that $\|\sum_{g \in \Sigma_n} \frac{1}{n!} gyg^{-1}\|_2 \leq \|y\|_2$, so an easy approximation argument shows that the limit exists for all $y \in L^2(N^{\bar{\otimes} \infty})$. When $y \in N^{\bar{\otimes} \infty}$, as $\|\sum_{g \in \Sigma_n} \frac{1}{n!} gyg^{-1}\|_\infty \leq \|y\|_\infty$, a subnet of $\sum_{g \in \Sigma_n} \frac{1}{n!} gyg^{-1}$ necessarily converges $\sigma$-weakly to an element within the algebra. But $\sum_{g \in \Sigma_n} \frac{1}{n!} gyg^{-1}$ converges in $L^2$, so the two limits must coincide and therefore the $L^2$ limit is in fact contained in the algebra whenever $y \in N^{\bar{\otimes} \infty}$. The limit is clearly invariant under conjugation by elements of $\Sigma_\infty$, so it is in $\hat{N} = \langle\Sigma_\infty\rangle' \cap N^{\bar{\otimes} \infty}$. One may then easily verify that the limit must indeed be $E_{\hat{N}}(y)$.

We observe that in the above argument we have also shown for $y \in \mathcal{S}$, $\lim_{n \rightarrow \infty} \sum_{g \in \Sigma_n} \frac{1}{n!} gyg^{-1} = \tau(y)$. Taking linear combinations and limits demonstrates that this holds for all $y \in N^{\bar{\otimes} \infty}$. In particular, $\hat{N} = \mathbb{C}$. Since $\mathcal{Z}(N^{\bar{\otimes} \infty} \rtimes_{vN} \Sigma_\infty)$ is contained in $\hat{N}$, this concludes the proof of the lemma.
\end{proof}

\begin{lemma}
For any QWEP finite von Neumann algebra $A$ and its hyperfinite subalgebra $B$, $\ast_B A$ is QWEP.
\end{lemma}

\begin{proof} We first prove the lemma for the case where $B$ is a matrix algebra $\mathbb{M}_n$. Write the matrix units as $\{e_{ij}\}_{1 \leq i,j \leq n}$. As $\mathbb{M}_n = B$ is a subalgebra of $A$, we may consider the algebra $\hat{A} = e_{11}Ae_{11}$. Note that we then have an isomorphism $\mathbb{M}_n \otimes \hat{A} \simeq A$ defined by sending $(a_{ij})_{1 \leq i,j \leq n}$ to $\sum_{1 \leq i,j \leq n} e_{i1}a_{ij}e_{1j}$. Thus, $\ast_B A = \ast_{\mathbb{M}_n} \mathbb{M}_n \otimes \hat{A}$, where the embedding $\mathbb{M}_n \hookrightarrow \mathbb{M}_n \otimes \hat{A}$ is canonical. Now there is an embedding $\ast_{\mathbb{M}_n} \mathbb{M}_n \otimes \hat{A} \hookrightarrow \mathbb{M}_n \otimes \ast_{\mathbb{C}} \hat{A}$. Indeed, there is a natural embedding $\iota: \mathbb{M}_n \otimes \hat{A} \rightarrow \mathbb{M}_n \otimes \ast_{\mathbb{C}} \hat{A}$ by simply tensoring the canonical embedding $\hat{A} \hookrightarrow \ast_{\mathbb{C}} \hat{A}$ with $Id_{\mathbb{M}_n}$. Taking the free product of all such $\iota$ yields the desired embedding. We note that as $\hat{A}$ embeds into $\mathbb{M}_n \otimes \hat{A}$ which is isomorphic to $A$ and $A$ is a QWEP finite von Neumann algebra, $\hat{A}$ is QWEP and thus $\ast_{\mathbb{C}} \hat{A}$ is QWEP by [Jun05, Theorem 7.15]. Hence, $\mathbb{M}_n \otimes \ast_{\mathbb{C}} \hat{A}$ is QWEP and so is $\ast_B A = \ast_{\mathbb{M}_n} \mathbb{M}_n \otimes \hat{A}$.

We now consider the case where $B = R$ is the hyperfinite $\textrm{II}_1$ factor. Then $B$ is the weak closure of the union of an increasing sequence of matrix algebras $\mathbb{M}_{2^n}$. We then have an embedding $\ast_R A \hookrightarrow \prod_{\omega} \ast_{\mathbb{M}_{2^n}} A$ by taking the free product of canonical inclusions,
\begin{equation*}
    A \hookrightarrow \prod_{n \rightarrow \omega} A_{i_0} = \prod_{n \rightarrow \omega} \ast_{\mathbb{M}_{2^n}}^{i \in I} A_i = \prod_{n \rightarrow \omega} \ast_{\mathbb{M}_{2^n}} A
\end{equation*}

Now $\ast_{\mathbb{M}_{2^n}} A$ are QWEP for all $n$, so $\prod_{\omega} \ast_{\mathbb{M}_{2^n}} A$ embeds into $(R^\omega)^\omega = R^{\omega^2}$. This establishes that $\ast_B A$ is QWEP.

Finally, we extend this result to the case where $B$ is not necessarily a factor. The argument we shall sketch here is due to Haagerup. We apply Lemma 2.2 to $A$ and $B$. For the latter, as $B$ is hyperfinite, $B^{\bar{\otimes} \infty}$ is so as well. Since $\Sigma_\infty$ is amenable, $B^{\bar{\otimes} \infty} \rtimes_{vN} \Sigma_\infty$, being hyperfinite, a factor, and infinite-dimensional but having a faithful tracial state, is therefore isomorphic to $R$. Since $A$ is embeddable into $R^\omega$, $A^{\bar{\otimes} \infty}$ is easily seen to be so as well. Using the fact that $\Sigma_\infty$ is amenable, we may embed $A^{\bar{\otimes} \infty} \rtimes_{vN} \Sigma_\infty$ into $(R \bar{\otimes} A^{\bar{\otimes} \infty})^\omega$. (The argument is quite easy and uses the existence of a Følner sequence. For a detailed argument see the proof of Theorem 3.1.) Hence $A^{\bar{\otimes} \infty} \rtimes_{vN} \Sigma_\infty$ is QWEP. One may also verify that we have the following commuting square,

\begin{center}
\begin{tikzcd}[node distance = 1.8cm]
    {A^{\bar{\otimes} \infty} \rtimes_{vN} \Sigma_\infty} \arrow[hookleftarrow]{r}\arrow[hookleftarrow]{d}
        & A \arrow[hookleftarrow]{d} \\
    {B^{\bar{\otimes} \infty} \rtimes_{vN} \Sigma_\infty} \arrow[hookleftarrow]{r}
        & B
\end{tikzcd}
\end{center}

By Lemma 2.1, $\ast_B A$ embeds into $\ast_{B^{\bar{\otimes} \infty} \rtimes_{vN} \Sigma_\infty} A^{\bar{\otimes} \infty} \rtimes_{vN} \Sigma_\infty$. The latter is QWEP by our previous argument, so $\ast_B A$ is QWEP as well. This concludes the proof.
\end{proof}

\begin{rmk}
While the lemma only demonstrates the QWEP property for algebras of the form $\ast_B A$, a slight alteration to the argument shows a more general result: Algebras of the form $\ast_B^{i \in I} A_i$ are QWEP whenever all algebras $A_i$ are QWEP finite von Neumann algebras and $B$ is their shared hyperfinite subalgebra, even without the assumption that all $A_i$ are the same algebra.
\end{rmk}

\begin{proof}[Proof of Theorem 2.1] Let $\hat{M} = \prod_{\omega} A_k \bar{\otimes} N_1$, $\hat{N} = \prod_{\omega} B_k \bar{\otimes} N_1$, where $A_k$ are QWEP and $B_k$ are hyperfinite. By Lemma 2.1, there exists a tracial embedding $\iota: \ast_N M \rightarrow \ast_{\hat{N}} \hat{M}$. We now consider the algebra $\ast_{\hat{N}} \hat{M}$.

We note that, given $(a_k)^\circ \in \prod_{\omega} A_k \bar{\otimes} N_1 = \hat{M}$, $E_{\hat{M},\hat{N}}((a_k)^\circ) = (E_{A_k \bar{\otimes} N_1, B_k \bar{\otimes} N_1}(a_k))^\circ$. Therefore, for any $(a_k)^\circ \in \hat{M}$ s.t. $E_{\hat{M},\hat{N}}((a_k)^\circ) = 0$, we may assume $E_{A_k \bar{\otimes} N_1, B_k \bar{\otimes} N_1}(a_k) = 0$ for all $k$ through replacing $a_k$ by $a_k - E_{A_k \bar{\otimes} N_1, B_k \bar{\otimes} N_1}(a_k)$. This allows us to define a map $\ast_{\hat{N}} \hat{M} = \ast_{\prod_{\omega} B_k \bar{\otimes} N_1} \prod_{\omega} A_k \bar{\otimes} N_1 \rightarrow \prod_{\omega} \ast_{B_k \bar{\otimes} N_1} (A_k \bar{\otimes} N_1)$ by sending $(a^{i_1}_k)^\circ (a^{i_2}_k)^\circ \cdots (a^{i_n}_k)^\circ$ to $(a^{i_1}_k a^{i_2}_k \cdots a^{i_n}_k)^\circ$, where $i_l \in I$, the index set for the amalgamated free product, $(a^{i_l}_k)^\circ \in \hat{M}_{i_l} = \hat{M}$, $E_{A_k \bar{\otimes} N_1, B_k \bar{\otimes} N_1}(a_k) = 0$ for all $l$ and $k$, and $a^{i_l}_k$ are regarded as in $(A_k \bar{\otimes} N_1)_{i_l}$. It is easy to see that this is a tracial embedding.

Now, we consider the algebras $\ast_{B_k \bar{\otimes} N_1} (A_k \bar{\otimes} N_1)$. We may define a map $\ast_{B_k \bar{\otimes} N_1} (A_k \bar{\otimes} N_1) \rightarrow (\ast_{B_k} A_k) \bar{\otimes} N_1$ by taking free products of embedding maps $B_k \bar{\otimes} N_1 \hookrightarrow (\ast_{B_k} A_k) \bar{\otimes} N_1$ obtained by the natural inclusions $B_k \hookrightarrow (\ast_{B_k} A_k)$ tensoring with $Id_{N_1}$. Again, it is easy to see that this is a tracial embedding. Composing the embeddings we defined above, we obtain an embedding $\ast_N M \hookrightarrow \prod_{\omega} (\ast_{B_k} A_k) \bar{\otimes} N_1$. We note here that when restricted to any copy of $M$, the embedding simply reduces to the map $M \hookrightarrow \hat{M}$ in the definition of RE/$N_1$. In particular, when restricted to $N$, the embedding reduces to the given embedding $N \hookrightarrow \hat{N} = \prod_{\omega} B_k \bar{\otimes} N_1$. By Lemma 2.3, $\ast_{B_k} A_k$ are QWEP. As such, it suffices to prove the following commutative diagram is a commuting square,

\begin{center}
\begin{tikzcd}[node distance = 1.8cm]
    {\prod_{\omega} (\ast_{B_k} A_k) \bar{\otimes} N_1} \arrow[hookleftarrow]{r}\arrow[hookleftarrow]{d}
        & {\ast_N M} \arrow[hookleftarrow]{d} \\
    {\prod_{\omega} B_k \bar{\otimes} N_1} \arrow[hookleftarrow]{r}
        & N
\end{tikzcd}
\end{center}

To do so, note that the embedding $\ast_N M \hookrightarrow \prod_{\omega} (\ast_{B_k} A_k) \bar{\otimes} N_1$ is defined by first embedding $\ast_N M$ into $\ast_{\hat{N}} \hat{M}$ and then embedding $\ast_{\hat{N}} \hat{M}$ into $\prod_{\omega} (\ast_{B_k} A_k) \bar{\otimes} N_1$. Therefore, it suffices to prove the following is a commuting square,

\begin{center}
\begin{tikzcd}[node distance = 1.8cm]
    {\ast_{\hat{N}} \hat{M}} \arrow[hookleftarrow]{r}\arrow[hookleftarrow]{d}
        & {\ast_N M} \arrow[hookleftarrow]{d} \\
    {\hat{N}} \arrow[hookleftarrow]{r}
        & N
\end{tikzcd}
\end{center}

This is immediate from the commuting square condition in the definition of RE/$N_1$ and the definition of the embedding $\ast_N M \hookrightarrow \ast_{\hat{N}} \hat{M}$.
\end{proof}

\begin{rmk}
Corollary 2.1 demonstrates that the RE condition leads to embeddability of the amalgamated free product. If we require some further assumptions or if the amalgamated free product is enlarged to a slightly larger algebra, then the converse is also true. See Appendix.
\end{rmk}

Theorem 2.1 is the first stability property of RE/$N_1$ we shall prove here. We also observe the following easy facts:

\begin{prop}
Let $N \subseteq M$ be RE/$N_1$. If $N_1$ tracially embeds into $(R \bar{\otimes} N_2)^\omega$ for some ultrafilter $\omega$ and finite von Neumann algebra $N_2$, then $N \subseteq M$ is RE/$N_2$. In particular, if $N_1 \subseteq N_2$, then $N \subseteq M$ is RE/$N_2$. Furthermore, if $N_1$ is QWEP, then $N \subseteq M$ is RE/$\mathbb{C}$.
\end{prop}

\begin{rmk}
The above result clearly also holds if $N_1$ tracially embeds into $(A \bar{\otimes} N_2)^\omega$ where $A$ is a QWEP finite von Neumann algebra.
\end{rmk}

We also have the following:

\begin{prop} Let $I$ be a countable index set. Suppose $N_i \subseteq M_i$ is RE/$\hat{N}_i$ for all $i$, then,

\medskip

1. $\oplus_{i \in I} N_i \subseteq \oplus_{i \in I} M_i$ is RE/$\oplus_{i \in I} \hat{N}_i$. Here, the trace on $\oplus_{i \in I} N_i$ is of the form $\sum_{i \in I} a_i \tau_{N_i}$ where $a_i > 0$ and $\sum_{i \in I} a_i = 1$. The same holds for $\oplus_{i \in I} M_i$ and $\oplus_{i \in I} \hat{N}_i$ and the same $a_i$ shall be used for all three algebras;

\medskip

2. $\bar{\otimes}_{i \in I} N_i \subseteq \bar{\otimes}_{i \in I} M_i$ is RE/$\bar{\otimes}_{i \in I} \hat{N}_i$.
\end{prop}

\begin{proof} Let $\widetilde{\pi_i}: M_i \rightarrow \oplus_{k=1}^\infty A_k \bar{\otimes} \hat{N}_i$ be linear maps s.t. when composed with the canonical quotient maps $\oplus_{k=1}^\infty A_k \bar{\otimes} \hat{N}_i \rightarrow \prod_{\omega} A_k \bar{\otimes} \hat{N}_i$ become tracial embeddings as in the definition of RE. For the first part of the proposition, in order for trace to preserved, we consider $\widetilde{\pi}: \oplus_{i \in I} M_i \rightarrow \oplus_{k=1}^\infty (\bar{\otimes}_{k=1}^\infty A_k) \bar{\otimes} (\oplus_{i \in I} \hat{N}_i)$ defined by,
\begin{equation*}
    \widetilde{\pi}(m_1 \oplus m_2 \oplus \cdots) = \widetilde{\pi_1}(m_1) \otimes 1_{\bar{\otimes}_{k \neq 1} A_k} + \widetilde{\pi_2}(m_2) \otimes 1_{\bar{\otimes}_{k \neq 2} A_k} + \cdots
\end{equation*}

Composing this map with the canonical quotient map $\oplus_{k=1}^\infty (\bar{\otimes}_{k=1}^\infty A_k) \bar{\otimes} (\oplus_{i \in I} \hat{N}_i) \rightarrow \prod_\omega (\bar{\otimes}_{k=1}^\infty A_k) \bar{\otimes} (\oplus_{i \in I} \hat{N}_i)$ now gives an tracial embedding. We note that $\bar{\otimes}_{k=1}^\infty A_k$ is QWEP. On the other hand, $\oplus_{i \in I} N_i$ would be sent to $\prod_\omega (\bar{\otimes}_{k=1}^\infty B_k) \bar{\otimes} (\oplus_{i \in I} \hat{N}_i)$ and $\bar{\otimes}_{k=1}^\infty B_k$ is hyperfinite. We now have a commutative diagram,

\begin{center}
\begin{tikzcd}[node distance = 1.8cm]
    {\prod_\omega (\bar{\otimes}_{k=1}^\infty A_k) \bar{\otimes} (\oplus_{i \in I} \hat{N}_i)} \arrow[hookleftarrow]{r}\arrow[hookleftarrow]{d}
        & {\oplus_{i \in I} M_i} \arrow[hookleftarrow]{d} \\
    {\prod_\omega (\bar{\otimes}_{k=1}^\infty A_k) \bar{\otimes} (\oplus_{i \in I} \hat{N}_i)} \arrow[hookleftarrow]{r}
        & {\oplus_{i \in I} N_i}
\end{tikzcd}
\end{center}

The commuting square condition can be verified easily.

For the second part of the proposition, we consider $\widetilde{\pi}: \odot_{i \in I} M_i \rightarrow \oplus_{k=1}^\infty (\bar{\otimes}_{k=1}^\infty A_k) \bar{\otimes} (\bar{\otimes}_{i \in I} \hat{N}_i)$ defined by,

\begin{equation*}
    \widetilde{\pi}(m_1 \otimes m_2 \otimes \cdots m_n \otimes 1 \otimes 1 \otimes \cdots) = \widetilde{\pi_1}(m_1) \otimes \widetilde{\pi_2}(m_2) \otimes \cdots \otimes \widetilde{\pi_n}(m_n) \otimes 1 \otimes 1 \otimes \cdots
\end{equation*}
where $\odot_{i \in I} M_i$ indicates the algebraic tensor product. We then compose it with the canonical quotient map from $\oplus_{k=1}^\infty (\bar{\otimes}_{k=1}^\infty A_k) \bar{\otimes} (\bar{\otimes}_{i \in I} \hat{N}_i)$ to $\prod_\omega (\bar{\otimes}_{k=1}^\infty A_k) \bar{\otimes} (\bar{\otimes}_{i \in I} \hat{N}_i)$. It is now easy to see that the map is trace-preserving, allowing us to extend it to $\bar{\otimes}_{i \in I} M_i$. The commuting square condition can then be verified easily.
\end{proof}

The following are two limit approximation theorems, where we shall demonstrate that the RE condition is stable under taking the union (resp. intersection) of an upward-directed (resp. downward-directed) sequence of algebras. The proof follows essentially the same line as the part of the proof of Lemma 2.3 where we approximate algebras of the form $\ast_R A$ by $\ast_{\mathbb{M}_{2^n}} A$. While the proof is relatively easy, the results obtained will be quite useful later.

\begin{thm}[Upward Limit Approximation Theorem]
Let $M$ be a finite von Neumann algebra, $M_i$, $N_i$ be two increasing sequences of subalgebras s.t. $N_i \subseteq M_i$ is RE/$\hat{N}$ for all $i$. Suppose further that $\cup_i M_i$ is weakly dense in $M$ and the weak closure of $\cup_i N_i$ is $N$, then $N \subseteq M$ is RE/$\hat{N}$.
\end{thm}

\begin{proof} Note that we have a tracial embedding of $M$ into $\prod_\omega M_i$ by sending $m \in M$ to $(E_{M, M_i}(m))^\circ$. By definition of RE, each $M_i$ embeds into an algebra of the form $\prod_{k \rightarrow \omega} A_{ik} \bar{\otimes} \hat{N}$ where $A_{ik}$ are QWEP. So we obtain an embedding,
\begin{equation*}
    M \hookrightarrow \prod_{i \rightarrow \omega} \prod_{k \rightarrow \omega} A_{ik} \bar{\otimes} \hat{N}
\end{equation*}

We observe here that if $m$ belongs to $N$, then $(E_{M, M_i}(m))^\circ$ and $(E_{N, N_i}(m))^\circ$ are actually the same element in $\prod_\omega M_i$ as both sequences converge to $m$ in $L^2(M)$. Hence, the embedding $M \hookrightarrow \prod_\omega M_i$ sends $N$ into $\prod_\omega N_i$. Therefore, the composition embedding above sends $N$ into $\prod_{i \rightarrow \omega} \prod_{k \rightarrow \omega} B_{ik} \bar{\otimes} \hat{N}$. We thus have a commutative diagram,

\begin{center}
\begin{tikzcd}[node distance = 1.8cm]
    {\prod_{i \rightarrow \omega} \prod_{k \rightarrow \omega} A_{ik} \bar{\otimes} \hat{N}} \arrow[hookleftarrow]{r}\arrow[hookleftarrow]{d}
        & M \arrow[hookleftarrow]{d} \\
    {\prod_{i \rightarrow \omega} \prod_{k \rightarrow \omega} B_{ik} \bar{\otimes} \hat{N}} \arrow[hookleftarrow]{r}
        & N
\end{tikzcd}
\end{center}

To verify this is a commuting square, we note that this commutative diagram is in fact the composition of two commutative diagrams,

\begin{center}
\begin{tikzcd}[node distance = 1.8cm]
    {\prod_{i \rightarrow \omega} \prod_{k \rightarrow \omega} A_{ik} \bar{\otimes} \hat{N}} \arrow[hookleftarrow]{r}\arrow[hookleftarrow]{d} & {\prod_{i \rightarrow \omega} M_i} \arrow[hookleftarrow]{r}\arrow[hookleftarrow]{d}
        & M \arrow[hookleftarrow]{d} \\
    {\prod_{i \rightarrow \omega} \prod_{k \rightarrow \omega} B_{ik} \bar{\otimes} \hat{N}} \arrow[hookleftarrow]{r} & {\prod_{i \rightarrow \omega} N_i} \arrow[hookleftarrow]{r}
        & N
\end{tikzcd}
\end{center}

It is sufficient to show both small squares satisfy the commuting square condition. The commuting square condition on the left one follows from $N_i \subseteq M_i$ being RE/$\hat{N}$. The commuting square condition for the other small square follows from the following consideration. For any $m \in M$, the image of $E_{M, N}(m)$ under the embedding $N \hookrightarrow \prod_\omega N_i$ is $(E_{N, N_i}(E_{M, N}(m)))^\circ = (E_{M, N_i}(m))^\circ$. Since the image of $m$ under the embedding $M \hookrightarrow \prod_\omega M_i$ is $(E_{M, M_i}(m))^\circ$, we have,
\begin{equation*}
\begin{split}
    E_{\prod_\omega M_i, \prod_\omega N_i}(m) &= E_{\prod_\omega M_i, \prod_\omega N_i}((E_{M, M_i}(m))^\circ)\\
    &= (E_{M_i, N_i}(E_{M, M_i}(m)))^\circ\\
    &= (E_{M, N_i}(m))^\circ\\
    &= E_{M, N}(m)
\end{split}
\end{equation*}
\end{proof}

\begin{thm}[Downward Limit Approximation Theorem]
Let $M_i$, $N_i$ be two decreasing sequences of finite von Neumann algebras s.t. $N_i \subseteq M_i$ is RE/$\hat{N}$ for all $i$. Let $\cap_i M_i = M$ and $\cap_i N_i = N$, then $N \subseteq M$ is RE/$\hat{N}$.
\end{thm}

\begin{proof} The argument is the same as that of the Upward Limit Approximation Theorem.
\end{proof}

\medskip

\section{\sc Facts on Group Algebras, Crossed Products, and Co-amenability}

\medskip

We shall now proceed to prove several facts regarding group algebras and crossed products which will be useful when demonstrating examples of the RE condition. We shall also provide some examples in this section involving co-amenability.

\begin{lemma}
Let $N$ be a finite von Neumann algebra, $G$ be a group acting through trace-preserving automorphisms on $N$, $H \triangleleft G$, and $|G:H| = n < \infty$, then $N \rtimes_{vN} H \subseteq N \rtimes_{vN} G$ is RE/$N \rtimes_{vN} H$.
\end{lemma}

\begin{proof} Fix some representatives $\{\widetilde{g_1}, \widetilde{g_2}, \cdots, \widetilde{g_n}\}$ of cosets in $G/H$, then any element $g \in G$ can be uniquely written as $g = \widetilde{g_i}h$ for some $i$ and $h \in H$. In this way, $L^2(N \rtimes_{vN} G)$ can be identified with,
\begin{equation*}
\begin{split}
    L^2(N \rtimes_{vN} G) &\simeq L^2(N) \otimes l_2(G)\\
    &\simeq L^2(N) \otimes l_2(G/H) \otimes l_2(H)\\
    &\simeq L^2(N) \otimes l_2^n \otimes l_2(H)\\
    &\simeq l_2^n \otimes L^2(N \rtimes_{vN} H)
\end{split}
\end{equation*}

Now, for any $n, n' \in N$, $h \in H$, $g \in G$, and any $\widetilde{g_i}$, there exists unique $j$ and $h' \in H$ s.t. $g\widetilde{g_i} = \widetilde{g_j}h'$, so,
\begin{equation*}
\begin{split}
    gn\widetilde{g_i}hn' &= g\widetilde{g_i}\alpha_{\widetilde{g_i}^{-1}}(n)hn'\\
    &= \widetilde{g_j}h'\alpha_{\widetilde{g_i}^{-1}}(n)hn'
\end{split}
\end{equation*}
where $h'$ and $\alpha_{\widetilde{g_i}^{-1}}(n)$ does not depend on $hn'$. As elements of the form $gn$ span $N \rtimes_{vN} G$ and elements of the form $hn'$ span $N \rtimes_{vN} H$, this shows that when elements of $N \rtimes_{vN} G$ are written as a matrix in the algebra $\mathbb{B}(l_2^n \otimes L^2(N \rtimes_{vN} H)) \simeq \mathbb{M}_n \otimes \mathbb{B}(L^2(N \rtimes_{vN} H))$, every entry thereof is in $N \rtimes_{vN} H$. Thus, $N \rtimes_{vN} G$ embeds into $\mathbb{M}_n \otimes N \rtimes_{vN} H$ and we now need to show that this embedding is tracial. To do so, let $g \in G$, $n \in N$. If $g = e$, then it is easy to verify that $gn = n$ is sent to a diagonal matrix with entries $\alpha_{\widetilde{g_i}^{-1}}(n)$, which all have the same trace as $n$. This matrix therefore has the same trace as $n$. If $g \neq e$ but $g \in H$, then $gn\widetilde{g_i} = \widetilde{g_i}h'\alpha_{\widetilde{g_i}^{-1}}(n)$ where $h' = \widetilde{g_i}^{-1}g\widetilde{g_i} \in H$ by normality of $H$ in $G$. However, as $g \neq e$, $h' \neq e$ and thus $h'\alpha_{\widetilde{g_i}^{-1}}(n)$ has trace 0. The matrix of $gn$ therefore has trace 0 as well. Finally, when $g \notin H$, we have $\widetilde{g_i}^{-1}g\widetilde{g_i} \notin H$ for all $i$. As such, all entries of $gn$ as a matrix in $\mathbb{M}_n \otimes N \rtimes_{vN} H$ are off-diagonal and it hence has trace 0. This proves the embedding is tracial.

The above argument also shows the embedding sends $N \rtimes_{vN} H$ to $l_\infty^n \otimes N \rtimes_{vN} H$, so we have a commutative diagram,

\begin{center}
\begin{tikzcd}[node distance = 1.8cm]
    {\mathbb{M}_n \otimes N \rtimes_{vN} H} \arrow[hookleftarrow]{r}\arrow[hookleftarrow]{d}
        & {N \rtimes_{vN} G} \arrow[hookleftarrow]{d} \\
    {l_\infty^n \otimes N \rtimes_{vN} H} \arrow[hookleftarrow]{r}
        & {N \rtimes_{vN} H}
\end{tikzcd}
\end{center}

To prove this is a commuting square, we note that it is sufficient to verify that $\ring{(N \rtimes_{vN} G)} = \{x \in N \rtimes_{vN} G: E_{N \rtimes_{vN} G, N \rtimes_{vN} H}(x) = 0\}$ is sent by $E_{\mathbb{M}_n \otimes N \rtimes_{vN} H, l_\infty^n \otimes N \rtimes_{vN} H}$ to 0. But $\ring{(N \rtimes_{vN} G)}$ is the weak closure of the span of elements of the form $gn$, where $g \notin H$, so it suffices to verify $E_{\mathbb{M}_n \otimes N \rtimes_{vN} H, l_\infty^n \otimes N \rtimes_{vN} H}(gn) = 0$ for all $g \in G \backslash H$. This is now obvious as we have already seen in the above argument that such $gn$ is sent to a matrix with only off-diagonal entries through the embedding $N \rtimes_{vN} G \hookrightarrow \mathbb{M}_n \otimes N \rtimes_{vN} H$.

Finally, both $\mathbb{M}_n \otimes N \rtimes_{vN} H$ and $l_\infty^n \otimes N \rtimes_{vN} H$ embed canonically into their own ultrapowers. This demonstrates the RE condition.
\end{proof}

\begin{rmk}
A slight alteration of the first part of the above proof shows that $N \rtimes_{vN} G$ still tracially embeds into $\mathbb{M}_n \otimes N \rtimes_{vN} H$ without the normality assumption, i.e., as long as $H < G$ and $|G:H| = n < \infty$, $N \rtimes_{vN} G$ tracially embeds into $\mathbb{M}_n \otimes N \rtimes_{vN} H$. However, the above argument cannot be used to show the RE condition without the normality assumption, but we will use a different method to prove the general result later. We still present this argument here, nonetheless, as similar arguments and constructions will be used repeatedly throughout this paper, and isolating this argument in this relatively easy lemma makes it much clearer. We shall omit most details when using similar arguments later.
\end{rmk}

It is important to note here that while in Lemma 3.1, $N \rtimes_{vN} H \hookrightarrow N \rtimes_{vN} G \hookrightarrow \mathbb{M}_{|G:H|} \otimes N \rtimes_{vN} H$, the composition embedding $N \rtimes_{vN} H \hookrightarrow \mathbb{M}_{|G:H|} \otimes N \rtimes_{vN} H$ is not the canonical embedding $x \mapsto 1_{\mathbb{M}_{|G:H|}} \otimes x$. Indeed, one can easily see from the proof that given $h \in H$, $n \in N$, the matrix $hn \in \mathbb{M}_n \otimes N \rtimes_{vN} H$ is diagonal but the entries on the diagonal are $\widetilde{g_i}^{-1}h\widetilde{g_i}\alpha_{\widetilde{g_i}^{-1}}(n)$ instead of simply $hn$. We note here that the map $hn \mapsto \widetilde{g_i}^{-1}h\widetilde{g_i}\alpha_{\widetilde{g_i}^{-1}}(n)$ extends to an automorphism on $N \rtimes_{vN} H$. We shall slightly abuse the notation and write this automorphism as $N \rtimes_{vN} H \ni x \mapsto \widetilde{g_i}^{-1}x\widetilde{g_i} \in N \rtimes_{vN} H$. It is easy to see that this automorphism preserves the canonical trace on $N \rtimes_{vN} H$. Following this example, we shall generalize this situation and introduce the following concept:

\begin{defn}
Given a von Neumann algebra $N$, a collection of automorphisms $\{\alpha_1, \alpha_2, \cdots, \alpha_n\}$ on $N$, a \textit{twisted inclusion} of $N$ into $\mathbb{M}_n \otimes N$ is an injective *-homomorphism $\pi: N \rightarrow \mathbb{M}_n \otimes N$ s.t.
\begin{equation*}
    \pi(x) = 
\begin{pmatrix}
    \alpha_1(x) & & & \\
     & \alpha_2(x) & & \\
     & & \ddots & \\
     & & & \alpha_n(x)
\end{pmatrix}
\end{equation*}

More generally, if we are given a collection of countably infinitely many automorphisms $\{\alpha_1, \alpha_2, \cdots\}$, then a \textit{twisted inclusion} of $N$ into $\mathbb{B}(l_2(\mathbb{N})) \bar{\otimes} N$ is an injective *-homomorphism $\pi: N \rightarrow \mathbb{B}(l_2(\mathbb{N})) \bar{\otimes} N$ s.t.
\begin{equation*}
    \pi(x) = 
\begin{pmatrix}
    \alpha_1(x) & & \\
     & \alpha_2(x) & \\
     & & \ddots
\end{pmatrix}
\end{equation*}

Note that as an automorphism on a von Neumann algebra is automatically normal, any twisted inclusion is normal as well.
\end{defn}

From the definition, it is easy to verify the following two facts:

\begin{prop} $ $

\medskip

1. Suppose $\pi: N \rightarrow \mathbb{B}(l_2(\mathbb{N})) \bar{\otimes} N$ is a twisted inclusion, and $P \in \mathbb{B}(l_2(\mathbb{N}))$ is a projection whose range is of the form $\textrm{span}\{e_{i_1}, e_{i_2}, \cdots, e_{i_n}\}$ where $e_{i_j}$ are among the canonical basis vectors of $l_2(\mathbb{N})$, then $(P \otimes 1_N) \pi (P \otimes 1_N)$ is a twisted inclusion of $N$ into $\mathbb{M}_n \otimes N$;

\medskip

2. Suppose $\pi: N \rightarrow \mathbb{M}_n \otimes N$ is a twisted inclusion and the automorphisms $\alpha_1, \alpha_2, \cdots, \alpha_n$ preserves the trace, then the conditional expectation $E_{\mathbb{M}_n \otimes L(H), \pi(N)}$ is given by, 
\begin{equation*}
    E_{\mathbb{M}_n \otimes L(H), \pi(N)}((x_{ij})_{1 \leq i, j \leq n}) = \frac{1}{n} \sum_{i = 1}^n \alpha_i^{-1}(x_{ii})
\end{equation*}
\end{prop}

Observe that, while Lemma 3.1 only concerns the case where $|G:H| < \infty$, when instead $|G:H| = \infty$, we can still fix $\{\widetilde{g_1}, \widetilde{g_2}, \cdots\}$, representatives of cosets in $G/H$, and similarly construct an embedding $\iota: N \rtimes_{vN} G \hookrightarrow \mathbb{B}(l_2(G/H)) \bar{\otimes} N \rtimes_{vN} H \simeq \mathbb{B}(l_2(\mathbb{N})) \bar{\otimes} N \rtimes_{vN} H$. In this case we still have the composition of $N \rtimes_{vN} H \hookrightarrow N \rtimes_{vN} G$ and $\iota$ is not the canonical embedding $x \mapsto 1_{\mathbb{B}(l_2(\mathbb{N}))} \otimes x$, but a twisted inclusion instead. We observe here that while we cannot say these embeddings are tracial, as no canonical trace exists on $\mathbb{B}(l_2(\mathbb{N})) \bar{\otimes} N \rtimes_{vN} H$, it is possible to consider its finite-dimensional contractions. More precisely, let $F$ be a finite subset of $G/H$, $P_F$ be the orthogonal projection from $l_2(G/H)$ onto $\textrm{span}(F)$. Then we can consider the ucp map $\varphi_F: N \rtimes_{vN} G \rightarrow \mathbb{M}_{|F|} \otimes N \rtimes_{vN} H$ defined by $\varphi_F(x) = (P_F \otimes 1_{N \rtimes_{vN} H}) \iota(x) (P_F \otimes 1_{N \rtimes_{vN} H})$, and we have the following:

\begin{lemma}
For simplicity, write $A = N \rtimes_{vN} G$ and $B = N \rtimes_{vN} H$, then $\varphi_F$ as defined above is a ucp $B-B$ bimodule map which preserves the conditional expectation onto $B$ and the trace. We note that $\varphi_F$ restricted to $B \subseteq A$ is a twisted inclusion of $B$ into $\mathbb{M}_{|F|} \otimes B$, and here $B-B$ bimodule map is understood so that the copy of $B$ in $\mathbb{M}_{|F|} \otimes B$ is the twisted copy $\varphi_F(B)$, i.e., $\varphi_F(xy) = \varphi_F(x)\varphi_F(y)$ whenever $x \in B$ and $y \in A$, or $x \in A$ and $y \in B$. Similarly, preserving the conditional expectation onto $B$ should be understood to mean $E_{\mathbb{M}_{|F|} \otimes B, \varphi_F(B)} \circ \varphi_F = E_{A, B}$.
\end{lemma}

\begin{proof} That $\varphi_F$ restricted to $B \subseteq A$ is a twisted inclusion of $B$ into $\mathbb{M}_{|F|} \otimes B$ follows from part 1 of Proposition 3.1, and we also get $\varphi_F(B)$ is indeed a weakly closed subalgebra of $\mathbb{M}_{|F|} \otimes B$.

Now, given $x \in B$, $y \in A$, we write the matrix form of $y$ in $\mathbb{B}(l_2(G/H)) \bar{\otimes} B$ as $(y_{ij})_{i, j \in G/H}$. We also note the matrix form of $x$ in $\mathbb{B}(l_2(G/H)) \bar{\otimes} B$ would be $(x_{ij})_{i, j \in G/H}$ where $x_{ij} = \delta_{ij}\widetilde{g_i}^{-1}x\widetilde{g_i}$. Thus,
\begin{equation*}
\begin{split}
    \varphi_F(xy) &= \varphi_F((\widetilde{g_i}^{-1}x\widetilde{g_i} y_{ij})_{i, j \in G/H})\\
    &=(\widetilde{g_i}^{-1}x\widetilde{g_i} y_{ij})_{i, j \in F}\\
    &=(\delta_{ij}\widetilde{g_i}^{-1}x\widetilde{g_i})_{i, j \in F}(y_{ij})_{i, j \in F}\\
    &=\varphi_F(x)\varphi_F(y)
\end{split}
\end{equation*}

This holds similarly for $\varphi_F(yx)$, which concludes the proof that $\varphi_F$ is an $B-B$ bimodule map. Being trace-preserving is an immediate consequence of preserving the conditional expectation, and to prove the latter, we note that, for any $g \in G \backslash H$, $n \in N$, all entries of $\varphi_F(gn) \in \mathbb{M}_{|F|} \otimes B$ are off-diagonal and it hence has conditional expectation 0, as $\varphi(B) \subseteq l_\infty^{|F|} \otimes B$. Now, for $h \in H$, we directly calculate,
\begin{equation*}
    E_{\mathbb{M}_{|F|} \otimes B, \varphi_F(B)}(\varphi_F(hn)) = \frac{1}{|F|} \sum_{i \in F} \widetilde{g_i}\widetilde{g_i}^{-1}hn\widetilde{g_i}\widetilde{g_i}^{-1} = hn = E_{M, N}(hn)
\end{equation*}
\end{proof}

Equipped with the above results, we shall now generalize the arguments in the proof of Lemma 3.1 to the case of amenable groups using the existence of Følner sequences.

\begin{thm}
Let $N$ be a finite von Neumann algebra, $G$ be a group acting through trace-preserving automorphisms on $N$, $H \triangleleft G$, and $H$ is co-amenable in $G$, then $N \rtimes_{vN} H \subseteq N \rtimes_{vN} G$ is RE/$N \rtimes_{vN} H$.
\end{thm}

\begin{proof} Since the quotient group $G/H$ is amenable, we may choose a Følner sequence of finite subsets $F_k$ of $G/H$. Let $n(k) = |F_k|$. By definition of a Følner sequence, we have that for any $g \in G$, $gH \in F_k$ when $k$ is large enough and furthermore,
\begin{equation*}
    \lim_{k \rightarrow \infty} \frac{|F_k \bigtriangleup (gH)F_k|}{n(k)} = 0
\end{equation*}

Let $\varphi_{F_k}: N \rtimes_{vN} G \rightarrow \mathbb{M}_{n(k)} \otimes N \rtimes_{vN} H$ be the ucp map defined in Lemma 3.2. We can then consider the ucp map $\pi: N \rtimes_{vN} G \rightarrow \prod_\omega \mathbb{M}_{n(k)} \otimes N \rtimes_{vN} H$ given by $\pi(x) = (\varphi_{F_k}(x))^\circ$. By Lemma 3.2, $\varphi_{F_k}$ preserves the trace, so $\pi$ is tracial as well. To prove that it is a tracial embedding, it suffices to prove it is multiplicative. To do so, we only need to prove that both elements of the group $G$ and of $N$ are in its multiplicative domain. The latter is clear as $\varphi_{F_k}$ restricts to a twisted inclusion on $N \rtimes_{vN} H$, so $\pi$ is multiplicative on $N \rtimes_{vN} H$. For the former, we need to verify $\|\varphi_{F_k}(g_1g_2)-\varphi_{F_k}(g_1)\varphi_{F_k}(g_2)\|_2 \rightarrow 0$ for all $g_1, g_2 \in G$.

Fix $\{\widetilde{g_1}, \cdots, \widetilde{g_{n(k)}}\}$, representatives of cosets in $F_k$. Then, the matrix representation of $\varphi_{F_k}(g)$ for $g \in G$ is,
\begin{equation*}
    (g_{ij})_{1 \leq i, j \leq n(k)} = (1_{\widetilde{g_i}^{-1}g\widetilde{g_j} \in H} \widetilde{g_i}^{-1}g\widetilde{g_j})_{1 \leq i, j \leq n(k)}
\end{equation*}

Thus,
\begin{equation*}
\begin{split}
    (\varphi_{F_k}(g_1g_2)-\varphi_{F_k}(g_1)\varphi_{F_k}(g_2))_{ij} &= 1_{\widetilde{g_i}^{-1}g_1g_2\widetilde{g_j} \in H} \widetilde{g_i}^{-1}g_1g_2\widetilde{g_j} - \sum_{w=1}^{n(k)} 1_{\widetilde{g_i}^{-1}g_1\widetilde{g_w} \in H}1_{\widetilde{g_w}^{-1}g_2\widetilde{g_j} \in H} \widetilde{g_i}^{-1}g_1\widetilde{g_w}\widetilde{g_w}^{-1}g_2\widetilde{g_j}\\
    &= (1_{\widetilde{g_i}^{-1}g_1g_2\widetilde{g_j} \in H} - \sum_{w=1}^{n(k)} 1_{\widetilde{g_i}^{-1}g_1\widetilde{g_w} \in H}1_{\widetilde{g_w}^{-1}g_2\widetilde{g_j} \in H}) \widetilde{g_i}^{-1}g_1g_2\widetilde{g_j}
\end{split}
\end{equation*}

Note that, for fixed $i$, $j$, there is at most one $1 \leq w \leq n(k)$ s.t. $1_{\widetilde{g_i}^{-1}g_1\widetilde{g_w} \in H}1_{\widetilde{g_w}^{-1}g_2\widetilde{g_j} \in H} = 1$ and it would be zero otherwise. Furthermore, whenever such a $w$ exists, we must have $1_{\widetilde{g_i}^{-1}g_1g_2\widetilde{g_j} \in H} = 1$. Hence, the term $1_{\widetilde{g_i}^{-1}g_1g_2\widetilde{g_j} \in H} - \sum_{w=1}^{n(k)} 1_{\widetilde{g_i}^{-1}g_1\widetilde{g_w} \in H}1_{\widetilde{g_w}^{-1}g_2\widetilde{g_j} \in H}$ is zero when such $w$ exists, or when $1_{\widetilde{g_i}^{-1}g_1g_2\widetilde{g_j} \in H} = 0$, and one otherwise. Now, as all non-zero entries of $\varphi_{F_k}(g_1g_2)-\varphi_{F_k}(g_1)\varphi_{F_k}(g_2)$ are group elements, we have,
\begin{equation*}
    \|\varphi_{F_k}(g_1g_2)-\varphi_{F_k}(g_1)\varphi_{F_k}(g_2)\|_2 = \frac{\#\textrm{ of nonzero entries in }\varphi_{F_k}(g_1g_2)-\varphi_{F_k}(g_1)\varphi_{F_k}(g_2)}{n(k)}
\end{equation*}

As have been noted, for an entry of $\varphi_{F_k}(g_1g_2)-\varphi_{F_k}(g_1)\varphi_{F_k}(g_2)$ to be nonzero, the number $w$ as mentioned in the last paragraph must not exist. We must also have $1_{\widetilde{g_i}^{-1}g_1g_2\widetilde{g_j} \in H} = 1$. Note that for a fixed $1 \leq j \leq n(k)$, there is at most one such $i$. Also, if $1_{\widetilde{g_i}^{-1}g_1g_2\widetilde{g_j} \in H} = 1$, then for such a $w$ to not exist we must have $\widetilde{g_w}^{-1}g_2\widetilde{g_j} \notin H$ for all $1 \leq w \leq n(k)$, i.e., we must have $(g_2H)(\widetilde{g_j}H) \notin F_k$. Hence, $\widetilde{g_j}H \in F_k \backslash (g_2^{-1}H)F_k \subseteq F_k \bigtriangleup (g_2^{-1}H)F_k$. Therefore, the number of nonzero entries in $\varphi_{F_k}(g_1g_2)-\varphi_{F_k}(g_1)\varphi_{F_k}(g_2)$ is less than or equal to $|F_k \bigtriangleup (g_2^{-1}H)F_k|$, so,
\begin{equation*}
\begin{split}
    \|\varphi_{F_k}(g_1g_2)-\varphi_{F_k}(g_1)\varphi_{F_k}(g_2)\|_2 &= \frac{\#\textrm{ of nonzero entries in }\varphi_{F_k}(g_1g_2)-\varphi_{F_k}(g_1)\varphi_{F_k}(g_2)}{n(k)}\\
    &\leq \frac{|F_k \bigtriangleup (g_2^{-1}H)F_k|}{n(k)}\\
    &\rightarrow 0
\end{split}
\end{equation*}

This establishes that $\pi$ is a tracial embedding. Thanks to Lemma 3.2, $\varphi_{F_k}$ restricted to $N \rtimes_{vN} H$ is a twisted inclusion of $N \rtimes_{vN} H$ into $\mathbb{M}_{n(k)} \otimes N \rtimes_{vN} H$. This implies $\pi(N \rtimes_{vN} H) \subseteq \prod_{k \rightarrow \omega} l_\infty^{n(k)} \otimes N \rtimes_{vN} H$. To prove the commuting square condition, it suffices to prove that $E_{\prod_\omega \mathbb{M}_{n(k)} \otimes N \rtimes_{vN} H, \prod_\omega l_\infty^{n(k)} \otimes N \rtimes_{vN} H}(\pi(gn)) = 0$, where $g \in G \backslash H$ and $n \in N$. We observe that, as in the proof of Lemma 3.1, when $g \in G \backslash H$, all entries of $\varphi_{F_k}(gn)$ are off-diagonal, and this implies the desired condition.
\end{proof}

\begin{rmk}
For readers who are familiar with the concept of co-amenability in the context of von Neumann algebras, it is known that for a pair of groups $H < G$, $H$ is co-amenable in $G$ iff $L(H)$ is co-amenable in $L(G)$ [MP03, Corollary 7]. There is a Følner sequence characterization of co-amenability without assuming $H$ normal in $G$. Indeed, $H$ is co-amenable in $G$ iff the natural action of $G$ on the set of left cosets $G/H$ is amenable, and a Følner sequence characterization exists for amenable actions. See [KT08, Theorem 1.1]. A slight alteration of the above proof thus shows that $N \rtimes_{vN} G$ tracially embeds into $\prod_\omega \mathbb{M}_{n(k)} \otimes N \rtimes_{vN} H$ as long as $H < G$ and $H$ is co-amenable in $G$. Again, this cannot be used to show the RE condition without the normality assumption. It is open whether the RE condition holds in this general context.
\end{rmk}

\medskip

\section{\sc Comultiplication and Co-hyperlinearity}

\medskip

In the last section, we have used Følner sequences to demonstrate examples of the RE condition. In essence, Følner sequences are used to directly construct explicit matrix models, corresponding to these finite sets in the Følner sequences. While it is possible to extend this kind of arguments to other cases, for example to cases where residually finite groups are involved (in which case one may employ methods similar to those in the proof of Theorem 3.1 to show that given $H \triangleleft G$ and $G/H$ is residually finite, we have $N \rtimes_{vN} H \subseteq N \rtimes_{vN} G$ is RE/$N \rtimes_{vN} G$), this straightforward approach is hardly applicable to situations where no such explicit matrix models exist. In this section, we shall introduce a new and substantially different approach to proving the RE condition.

Given a group $G$, there is an operation known as comultiplication on it, which is defined by,
\begin{equation*}
    \Delta: G \rightarrow G \times G, \Delta(g) = (g, g)
\end{equation*}

It is easy to verify this map extends to a trace-preserving *-homomorphism on the group ring $\mathbb{C}[G]$ and therefore extends to a tracial embedding $L(G) \hookrightarrow L(G) \bar{\otimes} L(G)$. In a similar vein, there is a tracial embedding $N \rtimes_{vN} G \hookrightarrow N \rtimes_{vN} G \bar{\otimes} L(G)$ by sending $ng$ to $ng \otimes g$. By applying these maps repeatedly, we may also increase the number of copies of $L(G)$ on the RHS. While this seems like a trivial observation, it is in fact quite useful. We shall first demonstrate this with a relatively easy theorem:

\begin{thm}
Let $G$ be a group, $H_1, H_2, \cdots, H_n$ be its subgroups. If $L(H_1) \bar{\otimes} \cdots \bar{\otimes} L(H_n) \subseteq L(G) \bar{\otimes} \cdots \bar{\otimes} L(G)$ is RE/$N$ for some finite von Neumann algebra $N$, then $L(H_1 \cap \cdots \cap H_n) \subseteq L(G)$ is RE/$N$.
\end{thm}

\begin{proof} Consider the comultiplication map $L(G) \hookrightarrow L(G) \bar{\otimes} \cdots \bar{\otimes} L(G)$. $L(H_1 \cap \cdots \cap H_n)$ is clearly sent into $L(H_1) \bar{\otimes} \cdots \bar{\otimes} L(H_n)$, i.e., we have a commutative diagram,

\begin{center}
\begin{tikzcd}[node distance = 1.8cm]
    {L(G) \bar{\otimes} \cdots \bar{\otimes} L(G)} \arrow[hookleftarrow]{r}\arrow[hookleftarrow]{d}
        & {L(G)} \arrow[hookleftarrow]{d} \\
    {L(H_1) \bar{\otimes} \cdots \bar{\otimes} L(H_n)} \arrow[hookleftarrow]{r}
        & {L(H_1 \cap \cdots \cap H_n)}
\end{tikzcd}
\end{center}

To show this is a commuting square, it suffices to prove that whenever $g \in G$ is not in $H_1 \cap \cdots \cap H_n$, then $E_{L(G) \bar{\otimes} \cdots \bar{\otimes} L(G), L(H_1) \bar{\otimes} \cdots \bar{\otimes} L(H_n)}(g) = 0$. But the image of $g$ in $L(G) \bar{\otimes} \cdots \bar{\otimes} L(G)$ is $g \otimes \cdots \otimes g$, so,
\begin{equation*}
\begin{split}
    E_{L(G) \bar{\otimes} \cdots \bar{\otimes} L(G), L(H_1) \bar{\otimes} \cdots \bar{\otimes} L(H_n)}(g) &= E_{L(G) \bar{\otimes} \cdots \bar{\otimes} L(G), L(H_1) \bar{\otimes} \cdots \bar{\otimes} L(H_n)}(g \otimes \cdots \otimes g)\\
    &= E_{L(G), L(H_1)}(g) \otimes \cdots \otimes E_{L(G), L(H_n)}(g)
\end{split}
\end{equation*}

As $g \notin H_1 \cap \cdots \cap H_n$, $g$ does not belong to at least one of $H_i$ and therefore at least one of $E_{L(G), L(H_i)}(g)$ is zero. This proves the commutative diagram above is indeed a commuting square. Composing it with the commuting square in the definition of RE/$N$ concludes the proof.
\end{proof}

Combining this with the Downward Limit Approximation Theorem gives the infinitary version of this result:

\begin{col}
Let $G$ be a group, $H_1, H_2, \cdots$ be its subgroups. If $L(H_1) \bar{\otimes} \cdots \bar{\otimes} L(H_n) \subseteq L(G) \bar{\otimes} \cdots \bar{\otimes} L(G)$ is RE/$N$ for some finite von Neumann algebra $N$ and for all $n \in \mathbb{N}$, then $L(\cap_i H_i) \subseteq L(G)$ is RE/$N$.
\end{col}

\begin{rmk}
We observe that the following is a commuting square,

\begin{center}
\begin{tikzcd}[node distance = 1.8cm]
    {L(G)^{\bar{\otimes} \infty}} \arrow[hookleftarrow]{r}\arrow[hookleftarrow]{d}
        & {L(G)^{\bar{\otimes} n}} \arrow[hookleftarrow]{d} \\
    {\bar{\otimes}_{i=1}^\infty L(H_i)} \arrow[hookleftarrow]{r}
        & {\bar{\otimes}_{i=1}^n L(H_i)}
\end{tikzcd}
\end{center}
for all $n$. As such, the assumption in the previous corollary can be changed to $\bar{\otimes}_{i=1}^\infty L(H_i) \subseteq L(G)^{\bar{\otimes} \infty}$ is RE/$N$. We also note that this can be used in conjunction with part 2 of Proposition 2.3. For example, if $L(H_i) \subseteq L(G)$ is RE/$\mathbb{C}$ for all $i$, then by part 2 of Proposition 2.3 we have $\bar{\otimes}_{i=1}^\infty L(H_i) \subseteq L(G)^{\bar{\otimes} \infty}$ is RE/$\mathbb{C}$. In such cases, then, $L(\cap_i H_i) \subseteq L(G)$ is RE/$\mathbb{C}$.
\end{rmk}

To state the main result of this section, we need some definitions first:

\begin{defn}
Let $H < G$ be a pair of groups. We say $H$ is \textit{co-hyperlinear} in $G$ if there exist two decreasing sequences $G_i$ and $H_i$ of subgroups of $G$ s.t. $\cap_i G_i = H$, $H_i < G_i$, $H_i \triangleleft G$, $G/H_i$ is hyperlinear, and $G_i/H_i$ is amenable for all $i$.
\end{defn}

\begin{ex}
A trivial example of co-hyperlinearity is when $H \triangleleft G$ and $G/H$ is hyperlinear. In such cases one may simply choose $G_i = H_i = H$ for all $i$. The definition of co-hyperlinearity should be seen as a generalization of this example. However, the motivating example for this relatively complex notion is the concept of a group being separable over a subgroup, which is defined as follows:
\end{ex}

\begin{defn}[{[LR08]}]
Let $H < G$ be a pair of groups. We say $G$ is \textit{separable} over $H$ if there exists a sequence $G_i$ of subgroups of $G$ s.t. $\cap_i G_i = H$, and furthermore $|G:G_i| < \infty$, $\forall i \in \mathbb{N}$.
\end{defn}

When $H \triangleleft G$, it is clear that $G$ is separable over $H$ iff $G/H$ is residually finite.

\begin{lemma}
If $G$ is separable over $H$, then $H$ is co-hyperlinear in $G$.
\end{lemma}

\begin{proof} The sequence $G_i$ of subgroups of $G$ in the definition of separability directly translates to the corresponding sequence of subgroups in the definition of co-hyperlinearity. We still need to construct the sequence $H_i$ in the definition of co-hyperlinearity. To do so, we note that for any subgroup $K$ of $G$, we may define its normal core $K_G = \cap_{g \in G} gKg^{-1}$. It is easy to see that $K_G < K$ and that $K_G$ is a normal subgroup of $G$. Now suppose $|G:K| = n < \infty$. We may write $G/K = \{\widetilde{g_1}K, \cdots, \widetilde{g_n}K\}$. Then we have $K_G = \cap_{i=1}^n \widetilde{g_i}K\widetilde{g_i}^{-1}$. As $|G: \widetilde{g_i}K\widetilde{g_i}^{-1}| = |G:K| < \infty$ for all $i$, $K_G$ is a finite intersection of finite index subgroups of $G$, and is therefore of finite index in $G$. Applying this to the sequence of subgroups of $G_i$, we define $H_i = (G_i)_G$. It is then easy to see that the sequences $G_i$ and $H_i$ satisfy the definition of co-hyperlinearity.
\end{proof}

\begin{rmk}
Unfortunately, this argument does not generalize to conditions defined using co-amenability instead of finite index subgroups. The problem is that $H < G$, $H$ is co-amenable in $G$ does not imply $H_G$ is co-amenable in $G$ (or in $H$). See [MP03, Theorem 1].
\end{rmk}

\begin{ex}
It is known that when $G = \mathbb{F}_r$, $2 \leq r \leq \infty$, and $H$ is a finitely generated subgroup of $G$, then $G$ is separable over $H$ [Hall49, Theorem 5.1]. However, this does not apply to cases where $H$ is not finitely generated. This includes most cases where $H \triangleleft G$ as it is known that when $H \triangleleft G$ and $|G:H| = \infty$, $H$ cannot be finitely generated [Hat05, Exercise 1.A.7]. In order to solve this issue, we introduce the following definition:
\end{ex}

\begin{defn}
Let $H < G$ be a pair of groups. We say $H$ is \textit{$\sigma$-co-hyperlinear} in $G$ if there exist two increasing sequences $G_i$ and $H_i$ of subgroups of $G$ s.t. $\cup_i G_i = G$, $\cup_i H_i = H$, $H_i < G_i$, and $H_i$ is co-hyperlinear in $G_i$ for all $i$.
\end{defn}

The main theorem of this section is as follows:

\begin{thm}
Let $N$ be a finite von Neumann algebra, $G$ be a group acting through trace-preserving automorphisms on $N$, $H < G$, and $H$ is $\sigma$-co-hyperlinear in $G$, then $N \rtimes_{vN} H \subseteq N \rtimes_{vN} G$ is RE/$N \rtimes_{vN} G$.
\end{thm}

\begin{proof} By the Upward Limit Approximation Theorem, it suffices to consider the case where $H$ is co-hyperlinear in $G$. Let $G_i$ and $H_i$ be the sequences of subgroups of $G$ as in the definition of co-hyperlinearity. We have a comultiplication map $N \rtimes_{vN} G \hookrightarrow \prod_\omega L(G/H_i) \bar{\otimes} N \rtimes_{vN} G$ by sending $ng$ to $(gH_i \otimes ng)^\circ$. Then $N \rtimes_{vN} H$ is sent into $\prod_\omega L(G_i/H_i) \bar{\otimes} N \rtimes_{vN} G$. Thus, we have a commutative diagram:

\begin{center}
\begin{tikzcd}[node distance = 1.8cm]
    {\prod_\omega L(G/H_i) \bar{\otimes} N \rtimes_{vN} G} \arrow[hookleftarrow]{r}\arrow[hookleftarrow]{d}
        & {N \rtimes_{vN} G} \arrow[hookleftarrow]{d} \\
    {\prod_\omega L(G_i/H_i) \bar{\otimes} N \rtimes_{vN} G} \arrow[hookleftarrow]{r}
        & {N \rtimes_{vN} H}
\end{tikzcd}
\end{center}

To verify this is a commuting square, it suffices to show $E_{\prod_\omega L(G/H_i) \bar{\otimes} N \rtimes_{vN} G, \prod_\omega L(G_i/H_i) \bar{\otimes} N \rtimes_{vN} G}(ng) = 0$ when $g \notin H$. As $\cap_i G_i = H$, $g \notin G_i$ for sufficiently large $i$, so $gH_i \notin G_i/H_i$ for large $i$. For large $i$, then, $E_{L(G/H_i) \bar{\otimes} N \rtimes_{vN} G, L(G_i/H_i) \bar{\otimes} N \rtimes_{vN} G}(gH_i \otimes ng) = 0$. But the image of $ng$ in $\prod_\omega L(G/H_i) \bar{\otimes} N \rtimes_{vN} G$ is $(gH_i \otimes ng)^\circ$, so this establishes that we have a commuting square. As $G/H_i$ are hyperlinear and $G_i/H_i$ are amenable, $L(G/H_i)$ are QWEP and $L(G_i/H_i)$ are hyperfinite, so this concludes the proof.
\end{proof}

\begin{rmk}
When applying this to the case of finite index subgroups, we obtain that $N \rtimes_{vN} H \subseteq N \rtimes_{vN} G$ is RE/$N \rtimes_{vN} G$ when $H < G$, $|G:H| < \infty$. By Remark 3.1, we have $N \rtimes_{vN} G$ tracially embeds into $\mathbb{M}_{|G:H|} \otimes N \rtimes_{vN} H$, so $N \rtimes_{vN} H \subseteq N \rtimes_{vN} G$ is RE/$N \rtimes_{vN} H$. This is a generalization of Lemma 3.1. More generally, suppose $H < G$ and $H$ is both $\sigma$-co-hyperlinear and co-amenable in $G$, then combining Theorem 4.2 and Remark 3.2 gives $N \rtimes_{vN} H \subseteq N \rtimes_{vN} G$ is RE/$N \rtimes_{vN} H$. We observe here that when $H \triangleleft G$, co-amenability implies co-hyperlinearity, so this is a generalization of Theorem 3.1. It is clear that co-hyperlinearity does not imply co-amenability, even with the assumption that $H$ is normal in $G$. Conversely, assuming the hyperlinear conjecture is false, then co-amenability does not imply co-hyperlinearity in general either. Indeed, consider the construction in [MP03, Theorem 1]. If we choose $Q$ to be a non-hyperlinear group in the construction there, we have $K$ is co-amenable in $G$. But if $K$ is co-hyperlinear in $G$, then the intersection of the sequence $H_i$ in the definition of co-hyperlinearity would give a normal subgroup of $G$ contained in $K$. The only such group is the trivial group, so $\cap_i H_i =\{e\}$. But $G/H_i$ is hyperlinear for any $i$, so $G = G/\{e\} = G/\cap_i H_i$ is hyperlinear. But $Q < G$ is not hyperlinear, which is a contradiction.
\end{rmk}

\begin{ex}
Thanks to Lemma 4.1 and Example 4.2, we have that $H$ is $\sigma$-co-hyperlinear in $G$ whenever $H < G$ and $G$ is a free group. Thus, $L(H) \subseteq L(G)$ is RE/$L(G)$. Free groups are hyperlinear, so $L(G)$ is QWEP. Therefore, $L(H) \subseteq L(G)$ is RE/$\mathbb{C}$. This implies groups of the form $\ast_H G$ are hyperlinear whenever $G$ is a free group and $H < G$. More generally, we always have $N \rtimes_{vN} H \subseteq N \rtimes_{vN} G$ is RE/$\mathbb{C}$ whenever $N$ is hyperfinite, $G$ is a free group, and $H < G$. This follows from first observing that we have $N \rtimes_{vN} H \subseteq N \rtimes_{vN} G$ is RE/$N \rtimes_{vN} G$. We then note that $N \rtimes_{vN} G = N \rtimes_{vN} \mathbb{F}_r$ is naturally isomorphic to $\ast_N N \rtimes_{vN} \mathbb{Z}$, which is QWEP by Remark 2.3. We therefore get $N \rtimes_{vN} H \subseteq N \rtimes_{vN} G$ is RE/$\mathbb{C}$.
\end{ex}

\medskip

\section{\sc Approximate Retraction and Approximable Actions}

\medskip

While the results in the last section are interesting, one shortcoming is that Theorem 4.2 is of the form $N \subseteq M$ is RE/$M$ as opposed to $N \subseteq M$ is RE/$N$ in Theorem 3.1. If for instance we are interested in proving $\ast_N M$ is QWEP but it is only known that $N$ is QWEP, then Theorem 4.2 would be of little use. In this section, we impose some extra conditions to obtain results that would be useful in the above scenario.

Again, the examples we are concerned with are inclusions of algebras of the form $N \rtimes_{vN} H \subseteq N \rtimes_{vN} G$. Therefore, the conditions we are to define would relate to either the pair of groups $H < G$ or the action of $G$ on $N$.

\begin{defn}
Let $H < G$ be a pair of groups. We say $G$ \textit{approximately retracts} onto $H$ if there exists $G' < G$ s.t. $G'$ is co-amenable in $G$ and $H < G'$, and a group homomorphism $\phi: G' \rightarrow \widetilde{G}$ where $\widetilde{G}$ is a group containing $H$ as a co-amenable subgroup and $\phi$ restricts to the identity on $H$.
\end{defn}

An important example and the inspiration for this definition is the concept of virtual retraction.

\begin{defn}[{[LR08, Definition 1.5]}]
Let $H < G$ be a pair of groups. We say $G$ \textit{virtually retracts} onto $H$ if there exists $G' < G$ s.t. $|G:G'| < \infty$ and $H < G'$, and a group homomorphism $\phi: G' \rightarrow H$ which restricts to the identity on $H$.
\end{defn}

One may easily see that approximate retraction is a generalization of virtual retraction.

\begin{ex}
It is known that for a finitely generated free group $G = \mathbb{F}_r$, $2 \leq r < \infty$, and a finitely generated subgroup $H$ of $G$, $G$ virtually retracts onto $H$. This follows from the proof of [Hall49, Theorem 5.1].
\end{ex}

It is useful to understand the concept of virtual retraction as a condition on the ``action" of $G$ on $H$. Indeed, if $H \triangleleft G$, then the existence of the group homomorphism $\phi$ as in the definition implies the conjugation action of $G'$ on $H$ is inner. This follows from,
\begin{equation*}
    ghg^{-1} = \phi(ghg^{-1}) = \phi(g)\phi(h)\phi(g)^{-1} = \phi(g)h\phi(g)^{-1}
\end{equation*}
and $\phi(g) \in H$. So $G$ virtually retracts onto $H$ can be understood as meaning that the ``action" of $G$ on $H$ is inner up to some finite index subgroup. Analogously, $G$ approximately retracts onto $H$ can be understood as meaning that the ``action" of $G$ on $H$ is inner up to some co-amenability. Conversely, we may obtain examples of approximate retraction by assuming $H \triangleleft G$ and the conjugation action of $G$ on $H$ is nearly inner:

\begin{prop} Let $G$ be a group, $G' < G$ co-amenable, and $H \triangleleft G'$. Suppose further that $Z(H) = \{e\}$. Then if one of the following holds, we would have $G$ approximately retracts onto $H$,

\medskip

1. $\phi(G')/Inn(H)$ is amenable, where $\phi: G' \rightarrow Aut(H)$ is defined by $\phi(g)(h) = ghg^{-1}$; or,

\medskip

2. $Out(H) = Aut(H)/Inn(H)$ is amenable; or,

\medskip

3. $H$ is a complete group, i.e., $Z(H) = \{e\}$ and all automorphisms of $H$ are inner.
\end{prop}

\begin{proof} Certainly $3 \Rightarrow 2 \Rightarrow 1$. To prove 1, note that as $Z(H) = \{e\}$, $\phi$ when restricted to $H$ is an isomorphism between $H$ and its inner automorphism group $Inn(H)$. We may therefore identify $H$ with $Inn(H)$. Now we simply take $G'$ and $\phi$ in the definition of approximate retraction to be the corresponding objects in the statement of the proposition to conclude the proof.
\end{proof}

By analogy with this understanding of approximate retraction by actions, we may define the following condition on the action $G \curvearrowright N$:

\begin{defn}
Let $N$ be a finite von Neumann algebra, $G$ be a group acting through trace-preserving automorphisms on $N$, $H$ be a subgroup of $G$. We say the action $\alpha: G \rightarrow Aut_\tau(N)$ is \textit{amenably $H$-inner} if there exists $G'$, a co-amenable subgroup of $G$, s.t. $G' < U(N \rtimes_{vN} H) \cap \mathcal{N}(N)$ and $\alpha$ restricted to $G'$ is the conjugation action of $U(N \rtimes_{vN} H) \cap \mathcal{N}(N)$ on $N$. Here $\mathcal{N}(N)$ is the normalizer of $N$, i.e., $U(N \rtimes_{vN} H) \cap \mathcal{N}(N) = \{u \in U(N \rtimes_{vN} H): uNu^{-1} = N\}$. In case $H = \{e\}$, we shall simply say the action is \textit{amenably inner}.
\end{defn}

A generalization of this condition is as follows:

\begin{defn}
Let $N$ be a finite von Neumann algebra, $G$ be a group acting through trace-preserving automorphisms on $N$, $H$ be a subgroup of $G$. We say the action $\alpha: G \rightarrow Aut_\tau(N)$ is \textit{amenably $H$-approximable} if there exists a sequence of groups $G_i$ containing $H$ as a subgroup, a sequence of maps (not necessarily group homomorphisms) $\beta_i: G \rightarrow G_i$, and a sequence of amenably $H$-inner actions $\alpha_i: G_i \rightarrow Aut_\tau(N)$, s.t. the maps $\beta_i$ are approximately multiplicative, i.e., for any $g_1, g_2 \in G$, $\beta_i(g_1)\beta_i(g_2) = \beta_i(g_1g_2)$ eventually as $i \rightarrow \omega$; $\alpha_i\restriction_H = \alpha\restriction_H$; and there exists an ultrafilter $\omega$ s.t. for any $g \in G$, $n \in N$, $\alpha_i \circ \beta_i(g)(n)$ converges to $\alpha(g)(n)$ in $L^2(N)$ as $i \rightarrow \omega$. In case $H = \{e\}$, we shall simply say the action is \textit{amenably approximable}.
\end{defn}

\begin{ex}
An example and the motivation for generalizing amenably $H$-inner actions to amenably $H$-approximable actions comes from the concept of profinite actions, defined as follows:

\begin{defn}[{[Ioa11]}]
Given a standard probability space $(X, \mu)$ (i.e., $(X, \mu)$ is isomorphic to $([0, 1], dt)$, a p.m.p. (probability measure preserving) action $\alpha: G \curvearrowright X$ is said to be \textit{profinite} if it is an inverse limit of a sequence of p.m.p. actions $\alpha_i: G \curvearrowright X_i$ with $X_i$ finite probability spaces.
\end{defn}

That is, there are surjective maps $q_{ij}: X_j \rightarrow X_i$ for all $i \leq j$ s.t. $q_{ii} = Id_{X_i}$ for all $i$, $q_{ij} = q_{ik} \circ q_{kj}$ for all $i \leq k \leq j$, and $q_{ij} \circ \alpha_j(g) = \alpha_i(g) \circ q_{ij}$ for all $i \leq j$ and $g \in G$. And $X$ is isomorphic to $\varprojlim X_i$ with the action $\alpha$ being isomorphic to the action $\alpha(g)(x_1, x_2, \cdots) = (\alpha_1(g)(x_1), \alpha_2(g)(x_2), \cdots)$. Note that there are natural surjective maps $q_i: X \rightarrow X_i$ s.t. $q_i = q_{ij} \circ q_j$ and $q_i \circ \alpha(g) = \alpha_i(g) \circ q_i$.

Since the action $\alpha_i: G \curvearrowright X_i$ simply permutes elements of the finite set $X_i$, if we write $\alpha_i: G \rightarrow Sym(X_i)$, then $G_i = \alpha_i(G)$ is a finite group. Now, dualizing $q_i$ gives inclusions of von Neumann algebras $L^\infty(X_i) \hookrightarrow L^\infty(X, \mu)$. We note that $X \simeq \varprojlim X_i$ implies $L^\infty(X, \mu)$ is the weak closure of the union of the increasing sequence of algebras $L^\infty(X_i)$. Also observe that each atom $a_{i, 1}, \cdots, a_{i, |X_i|}$ of $X_i$ corresponds to a measurable subset $A_{i, k} = q_i^{-1}(\{a_{i, k}\})$ of $X$. Since $X$ is a standard probability space, it is without atoms. Thus, $A_{i, k}$ is without atoms as well, so $L^\infty(A_{i, k}, \frac{\mu}{\mu(A_{i, k})})$ is isomorphic to $L^\infty([0, 1])$. Fixing such isomorphisms gives an identification $L^\infty(X, \mu) \simeq L^\infty(X_i) \bar{\otimes} L^\infty([0, 1])$. Thus, the induced actions $G \rightarrow G_i \curvearrowright L^\infty(X_i)$ can be extended to actions $G_i \curvearrowright L^\infty(X_i) \bar{\otimes} L^\infty([0, 1]) \simeq L^\infty(X, \mu)$, which we shall write as $\alpha'_i$. The maps $G \rightarrow G_i$ shall be denoted by $\beta_i$. One may then verify that $\alpha'_i \circ \beta_i(g)(n) \rightarrow \alpha(g)(n)$ in $L^2$ for all $g \in G$ and $n \in L^\infty(X, \mu)$. Here, $\alpha$ is also used to denote the action $G \curvearrowright L^\infty(X, \mu)$ induced by $\alpha$. (This assertion clearly holds for $n \in \cup_i L^\infty(X_i)$. The weak density of $\cup_i L^\infty(X_i)$ in $L^\infty(X, \mu)$ and a standard approximation argument then show it holds in general.) It is then clear that $\alpha'_i$ is amenably inner and hence $\alpha$ is amenably approximable.

To summarize, we have,

\begin{prop}
Given a standard probability space $(X, \mu)$ and a p.m.p. profinite action $G \curvearrowright X$, the induced action $G \curvearrowright L^\infty(X, \mu)$ is amenably approximable.
\end{prop}

Another example of approximable actions follows from the above example. When $G \curvearrowright X$ is profinite, as the induced action $G \curvearrowright L^\infty(X, \mu)$ is approximable, it is easy to see that the action $G \curvearrowright L^\infty(X, \mu) \bar{\otimes} N$ is approximable for any finite von Neumann algebra $N$, where the action of $G$ on $L^\infty(X, \mu) \bar{\otimes} N$ is the action $G \curvearrowright L^\infty(X, \mu)$ tensoring with the trivial action on $N$. While this is not an interesting example in itself, we may consider taking the algebra associated with the noncommutative Poisson random measure on $(L^\infty(X, \mu) \bar{\otimes} N, \tau)$, as defined in [Jun21]. Given a finite von Neumann algebra $N$ with a faithful tracial state $\tau$, this algebra is given by,
\begin{equation*}
    \mathcal{M}(N) = \oplus_{k = 0}^\infty \bar{\otimes}_s^k N
\end{equation*}
where $\bar{\otimes}_s^k N$ is the $k$-th symmetric tensor product of $N$, i.e., the subalgebra of $N^{\bar{\otimes} k}$ invariant under the permutation action of $S_k$. The trace on $\mathcal{M}(N)$ is given by,
\begin{equation*}
    \sigma((x_k)) = \sum_{k = 0}^\infty \frac{e^{-1}}{k!} \tau^{\otimes k}(x_k)
\end{equation*}

One easily sees that given a tracial embedding $\pi: N \hookrightarrow M$, it naturally induces a tracial embedding $\mathcal{M}(\pi): \mathcal{M}(N) \hookrightarrow \mathcal{M}(M)$. Also, an trace-preserving action $\alpha: G \curvearrowright N$ naturally induces an action $\alpha': G \curvearrowright \mathcal{M}(N)$ defined by $\alpha'((x_k)) = (\alpha^{\otimes k}(x_k))$. Furthermore, if we have an increasing sequence of subalgebras $N_i \subseteq M$ s.t. their union is weakly dense in $M$, then $\mathcal{M}(N_i)$ is an increasing sequence of subalgebras of $\mathcal{M}(M)$ and their union is weakly dense in $\mathcal{M}(M)$ as well.

Combining the above observations and essentially repeating the proof of Proposition 5.2 give,

\begin{prop}
Given a standard probability space $(X, \mu)$ and a p.m.p. profinite action $G \curvearrowright X$, the induced action $G \curvearrowright \mathcal{M}(L^\infty(X, \mu) \bar{\otimes} N)$ is amenably approximable for any finite von Neumann algebra $N$.
\end{prop}

We can obtain another interesting example of amenably approximable actions by slightly altering the above situation. Now, instead of understanding $X_i$ as finite probability spaces, we regard them as finite discrete topological spaces. Correspondingly, we understand $X \simeq \varprojlim X_i$ as a compact topological space and the isomorphism between $X$ and $\varprojlim X_i$ as a homeomorphism. Again, we let the action $\alpha: G \curvearrowright X$ (through homeomorphisms) be the inverse limit of a sequence of actions $\alpha_i: G \curvearrowright X_i$. In this case, we shall still call the action $\alpha$ \textit{profinite}.

Given any discrete group $H$, we note that the group $C(X_i, H)$ defined with pointwise operations is naturally isomorphic to $X_i^{|X_i|}$, so $L(C(X_i, H)) \simeq L(H)^{\bar{\otimes}|X_i|}$. We shall write this algebra as $L(H)^{\bar{\otimes}X_i}$. By an abuse of notation, we write $L(C(X, H))$ as $L(H)^{\bar{\otimes}X}$, where we regard $C(X, H)$ as a discrete group. We observe that the projection maps $q_i: X \rightarrow X_i$ dualize to group inclusions $\iota_i: C(X_i, H) \hookrightarrow C(X, H)$. Similarly, $q_{ij}: X_j \rightarrow X_i$ dualizes to a group inclusion $\iota_{ij}: C(X_i, H) \hookrightarrow C(X_j, H)$ for all $i \leq j$. Obviously, there maps are compatible, i.e., $\iota_{ij} = \iota_{kj} \circ \iota_{ik}$ and $\iota_i = \iota_j \circ \iota_{ij}$ for all $i \leq k \leq j$. This means $\iota_i(C(X_i, H))$ forms an increasing sequence of subgroups of $C(X, H)$. If we also use $\iota_i$ and $\iota_{ij}$ to denote the induced tracial inclusions of group von Neumann algebras, then we also have $\iota_i(L(H)^{\bar{\otimes}X_i})$ forms an increasing of subalgebras of $L(H)^{\bar{\otimes}X}$. We have the following proposition:

\begin{lemma}
$\cup_i \iota_i(C(X_i, H)) = C(X, H)$. Consequently, $\cup_i \iota_i(L(H)^{\bar{\otimes}X_i})$ is weakly dense in $L(H)^{\bar{\otimes}X}$.
\end{lemma}

\begin{proof} We need to show that given any continuous map $f: X \rightarrow H$, there exists an $i$ and a continuous map $\widetilde{f}: X_i \rightarrow H$ s.t. $f = \widetilde{f} \circ q_i$. We claim that any clopen subset of $X$ is of the form $q_i^{-1}(A_i)$ where $A_i \subseteq X_i$. Granted the claim, we note that as $X$ is compact and $H$ is discrete, $f(X)$ is finite and thus can be written as $f(X) = \{h_1, \cdots, h_n\}$. Then each $f^{-1}(h_j)$ is clopen in $X$ and thus $f^{-1}(h_j) = q_{i_j}^{-1}(A_{i_j})$ for some $A_{i_j} \subseteq X_{i_j}$. By compatibility between $q_i$'s, if, for instance $i_{j_1} \geq i_{j_2}$, then $q_{i_{j_2}}^{-1}(A_{i_{j_2}}) = q_{i_{j_1}}^{-1}(q_{i_{j_2}i_{j_1}}^{-1}(A_{i_{j_2}}))$. Thus, we may assume all $i_j$ equal some fixed $i$, i.e., $f^{-1}(h_j) = q_i^{-1}(A_j)$ for some $A_j \subseteq X_i$. However, $\cup_j f^{-1}(h_j) = X$, so $\cup_j A_j = X_i$. We also clearly have $A_j$'s are pairwise disjoint. Hence, we may define $\widetilde{f}$ by letting $\widetilde{f}(A_j) = \{h_j\}$. This concludes the proof apart from the claim.

We now prove the claim. Noting that $X \simeq \varprojlim X_i$ is a subspace of the Cartesian product $\prod_i X_i$ and using the compatibility between $q_i$'s, we see that open subsets of $X$ are of the form $\{(x_i) \in X: x_1 \in A_1,\textrm{ or }x_2 \in A_2,\textrm{ or }\cdots\}$ where $A_i \subseteq X_i$. For simplicity, we write this set as $A_1 \vee A_2 \vee \cdots$. Similarly, closed subsets of $X$ are of the form $\{(x_i) \in X: x_1 \in B_1,\textrm{ and }x_2 \in B_2,\textrm{ and }\cdots\}$ where $B_i \subseteq X_i$, and we shall write this set as $B_1 \wedge B_2 \wedge \cdots$. Given $K \subseteq X$ clopen, then $K = A_1 \vee A_2 \vee \cdots = B_1 \wedge B_2 \wedge \cdots$ for some $A_i, B_i \subseteq X_i$. By taking $\widetilde{A_i} = \cup_{j \leq i} q_{ji}^{-1}(A_j)$ if necessary, we may assume $A_i \supseteq q_{ji}^{-1}(A_j)$ for all $i \geq j$. Similarly, by taking $\widetilde{B_i} = \cap_{j \leq i} q_{ji}^{-1}(B_j)$ if necessary, we may assume $B_i \subseteq q_{ji}^{-1}(B_j)$ for all $i \geq j$. Clearly, it suffices to show there exists $N$ s.t. for all $n \geq N$, $A_n = q_{Nn}^{-1}(A_N)$.

To do so, assume to the contrary that for all $N$, there exists $n \geq N$ s.t. $A_n \supsetneqq q_{Nn}^{-1}(A_N)$. Now, $A_1 \vee A_2 \vee \cdots = B_1 \wedge B_2 \wedge \cdots$ implies that $q_{ij}(A_j) \subseteq B_i$ for all $i \leq j$. Define,
\begin{equation*}
    C_n = (\cup_{m > n} q_{nm}(A_m)) \backslash A_n
\end{equation*}

By our assumptions $C_n \neq \varnothing$ and $C_n \subseteq B_n$. We then let $D_n = \cap_{m \geq n} q_{nm}(C_m)$. Now, for any fixed $n$, for any $m_1 \geq m_2 \geq n$, we have,
\begin{equation*}
\begin{split}
    q_{m_2m_1}(C_{m_1}) &= q_{m_2m_1}(\cup_{k > m_1} q_{m_1k}(A_k) \backslash A_{m_1})\\
    &\subseteq q_{m_2m_1}(\cup_{k > m_1} q_{m_1k}(A_k))\\
    &= \cup_{k > m_1} q_{m_2k}(A_k)\\
    &\subseteq \cup_{k > m_2} q_{m_2k}(A_k)
\end{split}
\end{equation*}
where in the third line we have used $q_{m_2m_1} \circ q_{m_1k} = q_{m_2k}$. And,
\begin{equation*}
\begin{split}
    q_{m_2m_1}(C_{m_1}) &= q_{m_2m_1}(\cup_{k > m_1} q_{m_1k}(A_k) \backslash A_{m_1})\\
    &\subseteq q_{m_2m_1}(X_{m_1} \backslash A_{m_1})\\
    &\subseteq q_{m_2m_1}(X_{m_1} \backslash q_{m_2m_1}^{-1}(A_{m_2}))\\
    &\subseteq X_{m_2} \backslash A_{m_2}
\end{split}
\end{equation*}
where in the third line we have used $q_{m_2m_1}^{-1}(A_{m_2}) \subseteq A_{m_1}$. Thus, $q_{m_2m_1}(C_{m_1}) \subseteq \cup_{k > m_2} q_{m_2k}(A_k) \backslash A_{m_2} = C_{m_2}$, so $q_{nm_1}(C_{m_1}) = q_{nm_2} \circ q_{m_2m_1}(C_{m_1}) \subseteq q_{nm_2}(C_{m_2})$, i.e., for any fixed $n$, the sequence $(q_{nm}(C_m))_{m \geq n}$ is decreasing. Note that $q_{nm}(C_m) \neq \varnothing$ for all $m \geq n$ and they are all subsets of the finite set $X_n$. Thus, $D_n = \cap_{m \geq n} q_{nm}(C_m)$ is nonempty. In fact, $D_n = q_{n\alpha(n)}(C_{\alpha(n)})$ for some $\alpha(n) \geq n$. Also observe that $D_n \subseteq C_n \subseteq B_n$.

Now, we have,
\begin{equation*}
    q_{n(n+1)}(D_{n+1}) = q_{n(n+1)} \circ q_{(n+1)\alpha(n+1)}(C_{\alpha(n+1)}) = q_{n\alpha(n+1)}(C_{\alpha(n+1)}) \supseteq D_n
\end{equation*}
for all $n$. We may then inductively define $(x_i) \in X$ by letting $x_1$ be an arbitrary element of $D_1$. Then, for all $i$, let $x_{i+1}$ be an element of $D_{i+1}$ s.t. $q_{i(i+1)}(x_{i+1}) = x_i$. Since $x_i \in D_i \subseteq B_i$, $(x_i) \in B_1 \wedge B_2 \wedge \cdots$. However, $x_i \in D_i \subseteq C_i \subseteq X_i \backslash A_i$, i.e., $x_i \notin A_i$ for all $i$, so $(x_i) \notin A_1 \vee A_2 \vee \cdots$. But $A_1 \vee A_2 \vee \cdots = B_1 \wedge B_2 \wedge \cdots$, a contradiction.
\end{proof}

Again, we write for the action $\alpha_i: G \curvearrowright X_i$, $G_i = \alpha_i(G)$, which is a finite group. We wish to show the induced actions $G \rightarrow G_i \curvearrowright L(H)^{\bar{\otimes}X_i}$ can be extended to actions $G_i \curvearrowright L(H)^{\bar{\otimes}X}$. Let $G_i X_i = \{G_i x: x \in X_i\}$ be the collection of all orbits of the action $G_i \curvearrowright X_i$. Given any $\sigma \in G_i X_i$ and $x, y \in \sigma$, we claim that $q_i^{-1}(x)$ is homeomorphic to $q_i^{-1}(x)$. Indeed, $y = \widetilde{g}x$ for some $\widetilde{g} \in G_i$. We may then lift $\widetilde{g}$ to some $g \in G$, so $y = gx$. But then the homeomorphism may simply be defined by the restriction of $\alpha(g)$ restricted to $q_i^{-1}(x)$. Now, for each $\sigma \in G_i X_i$, fix some $x_\sigma \in \sigma$ and homeomorphisms between $q_i^{-1}(x_\sigma)$ and $q_i^{-1}(x)$ for all other $x \in \sigma$. These homeomorphisms induce isomorphisms between $C(q_i^{-1}(x_\sigma), H)$ and $C(q_i^{-1}(x), H)$. We observe that $(q_i^{-1}(x))_{x \in X_i}$ is a finite collection of disjoint open subsets of $X$ which covers $X$, so we have,
\begin{equation*}
\begin{split}
    C(X, H) &\simeq \oplus_{x \in X_i} C(q_i^{-1}(x), H)\\
    &\simeq \oplus_{\sigma \in G_i X_i} \oplus_{x \in \sigma} C(q_i^{-1}(x), H)\\
    &\simeq \oplus_{\sigma \in G_i X_i} \oplus_{x \in \sigma} C(q_i^{-1}(x_\sigma), H)\\
    &\simeq \oplus_{\sigma \in G_i X_i} C(q_i^{-1}(x_\sigma), H)^{|\sigma|}
\end{split}
\end{equation*}

At the level of group von Neumann algebras, this gives $L(H)^{\bar{\otimes}X} \simeq \bar{\otimes}_{\sigma \in G_i X_i} (L(H)^{\bar{\otimes} q_i^{-1}(x_\sigma)})^{\bar{\otimes} \sigma}$. Since the action $G_i \curvearrowright X_i$ only acts within each $\sigma$, this allows an extension of $G_i \curvearrowright L(H)^{\bar{\otimes}X_i}$ to $\alpha'_i: G_i \curvearrowright L(H)^{\bar{\otimes}X}$. Again, we let the maps $G \rightarrow G_i$ be denoted by $\beta_i$ and then it is easy to verify that $\alpha'_i \circ \beta_i$ approximates $\alpha: G \curvearrowright L(H)^{\bar{\otimes}X}$. Thus, $\alpha$ is amenably approximable.

To summarize, we have,

\begin{prop}
Given a compact topological space $X$ and a profinite action $G \curvearrowright X$, the induced action $G \curvearrowright L(H)^{\bar{\otimes}X}$ is amenably approximable for any discrete group $H$.
\end{prop}
\end{ex}

\begin{ex}
Another set of examples of amenably approximable actions also concerns approximating the action by finite permutation groups. Here, we consider a countable discrete set $X$ and an action $\alpha: G \curvearrowright X$ through permutations. Given any finite von Neumann algebra $N$, let $N^{\bar{\otimes} X}$ be $N^{\bar{\otimes}\infty}$ with tensor components indexed by elements of $X$. We denote the induced action $G \curvearrowright N^{\bar{\otimes} X}$ also by $\alpha$.

We have the following,
\begin{prop}
Suppose there exists an increasing sequence $G_i$ of subgroups of $G$ s.t. $\cup_i G_i = G$ and all orbits of the restricted action $G_i \curvearrowright X$ are finite, then $\alpha$ is amenably approximable.
\end{prop}

\begin{proof} We shall only prove the case where all orbits of the full action $G \curvearrowright X$ are finite. The general case follows by an easy approximation argument. Under the assumption, we may let $X_i$ be an increasing sequence of finite subsets of $X$, each one of which is a union of orbits of $G \curvearrowright X$, and moreover $\cup_i X_i = X$. Let $\beta_i: G \rightarrow Sym(X_i)$ be the group homomorphism associated with the restricted action $G \curvearrowright X_i$. Since $X_i$ is finite, $Sym(X_i)$ is a finite group. It acts on $N^{\bar{\otimes} X_i}$ naturally. Since $N^{\bar{\otimes} X} = N^{\bar{\otimes} X_i} \bar{\otimes} N^{\bar{\otimes} X \backslash X_i}$, the action extends to an action $\alpha_i: Sym(X_i) \curvearrowright N^{\bar{\otimes} X}$ by simply tensoring the natural action on $N^{\bar{\otimes} X_i}$ with the trivial action on $N^{\bar{\otimes} X \backslash X_i}$. One may then easily verify that $\alpha_i \circ \beta_i$ approximates $\alpha$ by first verifying this for finite tensors $n_{x_1} \otimes n_{x_2} \otimes \cdots n_{x_k} \otimes 1 \otimes 1 \otimes \cdots \in N^{\bar{\otimes} X}$ and then applying a standard approximation argument.
\end{proof}

It may be somewhat surprising that the above argument can be applied without any assumption on the orbits if we instead assume $G$ is a free group. Indeed, because finite linear combinations of finite tensors are weakly dense in $N^{\bar{\otimes} X}$, even if we just pick $X_i$ to be an arbitrary increasing sequence of finite subsets of $X$ whose union is $X$ and then for some $g \in G$, $gX_i \neq X_i$, as long as we define $\beta_i$ so that $\beta_i(g)(x) = gx$ whenever both $x \in X_i$ and $gx \in X_i$, then $\alpha_i \circ \beta_i$ would still approximate $\alpha$. The only issue is that $\beta_i$ may not be multiplicative or even approximately so anymore. However, this is not a problem if $G$ is a free group, in which case we only need to define $\beta_i$ according to the standard above for free generators of $G$ and then extend it to a group homomorphism $G \rightarrow Sym(X_i)$. Hence, we have the following result.

\begin{prop}
Suppose $G$ is a free group acting on a countable discrete set $X$, then the induced action $G \curvearrowright N^{\bar{\otimes} X}$ is amenably approximable.
\end{prop}

\begin{proof}
Fix $X_i$, an increasing sequence of finite subsets of $X$ whose union is $X$. Again, we let $\alpha_i: Sym(X_i) \curvearrowright N^{\bar{\otimes} X}$ be the natural action $Sym(X_i) \curvearrowright N^{\bar{\otimes} X_i}$ extended to $N^{\bar{\otimes} X}$ as in the proof of Proposition 5.5. We need to define $\beta_i: G \rightarrow Sym(X_i)$. To do so, for each free generator $g$ of $G$, we let $\beta_i(g)$ be any permutation of $X_i$ which satisfies $\beta_i(g)(x) = gx$ whenever both $x \in X_i$ and $gx \in X_i$. This is always possible as the sets $X_i \backslash \{x \in X_i: gx \in X_i\}$ and $X_i \backslash g\{x \in X_i: gx \in X_i\}$ have the same cardinality, so we may let $\beta_i(g)$ acts on $X_i \backslash \{x \in X_i: gx \in X_i\}$ as any bijective map from $X_i \backslash \{x \in X_i: gx \in X_i\}$ to $X_i \backslash g\{x \in X_i: gx \in X_i\}$. We then extend $\beta_i$ to a group homomorphism from $G$ to $Sym(X_i)$.

Now, we need to show that $\alpha_i \circ \beta_i$ approximates $\alpha$. We only need to check this for finite tensors $n_{x_1} \otimes n_{x_2} \otimes \cdots n_{x_k} \otimes 1 \otimes 1 \otimes \cdots \in N^{\bar{\otimes} X}$. Given any $g \in G$, we write $g = g_1^{\epsilon_1} g_2^{\epsilon_2} \cdots g_l^{\epsilon_l}$ where each $g_m$ is a free generator of $G$ and $\epsilon_m \in \{\pm 1\}$. For sufficiently large $i$, $X_i$ contains all $x_j$ as well as all $g_m^{\epsilon_m} g_2^{\epsilon_2} \cdots g_l^{\epsilon_l}(x_j)$ for all $1 \leq j \leq k$ and $1 \leq m \leq l$. One may then easily verify that $\beta_i(g)(x_j) = gx_j$ for all $j$. Thus, $\alpha_i \circ \beta_i(g)(n_{x_1} \otimes n_{x_2} \otimes \cdots n_{x_k} \otimes 1 \otimes 1 \otimes \cdots) = \alpha(g)(n_{x_1} \otimes n_{x_2} \otimes \cdots n_{x_k} \otimes 1 \otimes 1 \otimes \cdots)$ for large $i$. This concludes the proof.
\end{proof}
\end{ex}

A concept similar to amenably $H$-approximable action can be defined by changing the form of ``approximation" allowed from co-amenability to matrices.

\begin{defn}
Let $N$ be a finite von Neumann algebra, $G$ be a group acting through trace-preserving automorphisms on $N$, $H$ be a subgroup of $G$. We say the action $\alpha: G \rightarrow Aut_\tau(N)$ is \textit{matricially $H$-approximable} if there exists a sequence of maps (not necessarily group homomorphisms) $\beta_i: G \rightarrow U(\mathbb{M}_{n(i)} \otimes N \rtimes_{vN} H)$ and an ultrafilter $\omega$ s.t. the maps $\beta_i$ are approximately multiplicative in $L^2$ and under $\omega$, i.e., for any $g_1, g_2 \in G$, $\|\beta_i(g_1)\beta_i(g_2) - \beta_i(g_1g_2)\|_2 \rightarrow 0$ as $i \rightarrow \omega$; and for any $g \in G$, $n \in N$, $\beta_i(g)n\beta_i(g)^{-1}$ converges to $\alpha(g)(n)$ in $L^2$ as $i \rightarrow \omega$. In case $H = \{e\}$, we shall simply say the action is \textit{matricially approximable}.
\end{defn}

The motivation for this definition will be discussed later.

The main theorem of this section is as follows:

\begin{thm}
Let $N$ be a finite von Neumann algebra, $G$ be a group acting through trace-preserving automorphisms on $N$, $H, H_1 < G$, $H$ is $\sigma$-co-hyperlinear in $G$, $G$ approximately retracts onto $H$, and the action of $G$ on $N$ is either amenably $H_1$-approximable or matricially $H_1$-approximable, then $N \rtimes_{vN} H \subseteq N \rtimes_{vN} G$ is RE/$N \rtimes_{vN} H_1 \bar{\otimes} L(H)$.
\end{thm}

We split some parts of the proof into two lemmas for clarity.

\begin{lemma}
Let $H < G$ be a pair of groups. If $H$ is $\sigma$-co-hyperlinear in $G$ and $G$ approximately retracts onto $H$, then $L(H) \subseteq L(G)$ is RE/$L(H)$.
\end{lemma}

\begin{proof} By Theorem 4.2 and Proposition 2.2, it suffices to show that $L(G)$ tracially embeds into $(R \bar{\otimes} L(H))^\omega$. By definition of $\sigma$-co-hyperlinearity, we may pick increasing sequences $G_i$ and $H_i$ of subgroups of $G$ s.t. $\cup_i G_i = G$, $\cup_i H_i = H$, $H_i < G_i$, and $H_i$ is co-hyperlinear in $G_i$ for all $i$. We then define,
\begin{equation*}
    \pi_1: \mathbb{C}[G] \rightarrow \oplus_i \mathbb{C}[G_i], G \ni g \mapsto (1_{g \in G_i}g)_i \in \oplus_i \mathbb{C}[G_i]
\end{equation*}

Fix an $i$. By definition of co-hyperlinearity, we may pick decreasing sequences $G_{ij}$ and $H_{ij}$ of subgroups of $G_i$ s.t. $\cap_j G_{ij} = H_i$, $H_{ij} < G_{ij}$, $H_{ij} \triangleleft G_i$, $G_i/H_{ij}$ is hyperlinear, and $G_{ij}/H_{ij}$ is amenable for all $j$. Applying the comultiplication map, we first define,
\begin{equation*}
    \pi_{2, i}: \mathbb{C}[G_i] \rightarrow \oplus_j L(G_i/H_{ij}) \odot \mathbb{C}[G], G_i \ni g \mapsto (gH_{ij} \otimes g)_j \in \oplus_j L(G_i/H_{ij}) \odot \mathbb{C}[G]
\end{equation*}

By definition of approximate retraction, there exists $G' < G$ s.t. $G'$ is co-amenable in $G$ and $H < G'$, and a group homomorphism $\phi: G' \rightarrow \widetilde{G}$ where $\widetilde{G}$ is a group containing $H$ as a co-amenable subgroup and $\phi$ restricts to the identity on $H$. Using methods in the proof of Theorem 3.1 gives a sequence of maps $\mathbb{C}[G] \rightarrow \mathbb{M}_{n(k)} \odot \mathbb{C}[G']$ that would be multiplicative after taking the ultraproduct. We write this sequence of maps in the form,
\begin{equation*}
    \pi_3: \mathbb{C}[G] \rightarrow \oplus_k \mathbb{M}_{n(k)} \odot \mathbb{C}[G']
\end{equation*}

Applying $\phi$ then sends $\mathbb{C}[G']$ to $\mathbb{C}[\widetilde{G}]$. Since $\widetilde{G}$ is a group containing $H$ as a co-amenable subgroup, using the above procedures again sends $\mathbb{C}[\widetilde{G}]$ to $\oplus_l \mathbb{M}_{m(l)} \odot \mathbb{C}[H]$, i.e., we obtain the following maps,
\begin{equation*}
    \pi_4: \mathbb{C}[\widetilde{G}] \rightarrow \oplus_l \mathbb{M}_{m(l)} \odot \mathbb{C}[H]
\end{equation*}

We may then embed $\mathbb{C}[H]$ into the corresponding group von Neumann algebra $L(H)$. Now, combining the maps above yields,
\begin{equation*}
    \pi: \mathbb{C}[G] \rightarrow \oplus_i \oplus_j \oplus_k \oplus_l L(G_i/H_{ij}) \odot \mathbb{M}_{n(k)} \odot \mathbb{M}_{m(l)} \odot L(H)
\end{equation*}

Changing direct sums into ultraproducts and algebraic tensor products into von Neumann algebra tensor products gives a map,
\begin{equation*}
    \widetilde{\pi}: \mathbb{C}[G] \rightarrow \prod_{i \rightarrow \omega} \prod_{j \rightarrow \omega} \prod_{k \rightarrow \omega} \prod_{l \rightarrow \omega} L(G_i/H_{ij}) \bar{\otimes} \mathbb{M}_{n(k)} \bar{\otimes} \mathbb{M}_{m(l)} \bar{\otimes} L(H)
\end{equation*}

As $G_i/H_{ij}$ is hyperlinear for all $i$ and $j$, $L(G_i/H_{ij})$ embeds into $R^\omega$. Therefore, the RHS tracially embeds into $(R \bar{\otimes} L(H))^{\omega'}$ for some ultrafilter $\omega'$. As such, it suffices to show that $\widetilde{\pi}$ is tracial and multiplicative, and hence extends to a tracial embedding $L(G) \hookrightarrow \prod_{i \rightarrow \omega} \prod_{j \rightarrow \omega} \prod_{k \rightarrow \omega} \prod_{l \rightarrow \omega} L(G_i/H_{ij}) \bar{\otimes} \mathbb{M}_{n(k)} \bar{\otimes} \mathbb{M}_{m(l)} \bar{\otimes} L(H)$. Multiplicativity is easy to verify as all maps used in defining $\widetilde{\pi}$ are either multiplicative or approximately so in the limit. ($\pi_1$ is easily seen to be approximately multiplicative. $\pi_{2, i}$ and $\phi$ are multiplicative. The arguments in the proof of Theorem 3.1 can be used to show that $\pi_3$ and $\pi_4$ are approximately multiplicative.) It now remains to demonstrate $\widetilde{\pi}$ is trace-preserving. To do so, we only need to calculate the trace of $\widetilde{\pi}(g)$ for $g \in G$.

Recall that $H_G = \cap_{g \in G} gHg^{-1}$ is the normal core of $H$ in $G$. Suppose $g \notin H_G$. Then $g \notin g'Hg'^{-1}$ for some $g' \in G$. As $\cup_i G_i = G$, for sufficiently large $i$, both $g$ and $g'$ belongs to $G_i$. But then as $H_i < H$, we have,
\begin{equation*}
    {H_i}_{G_i} \subseteq g'H_ig'^{-1} \subseteq gHg^{-1}
\end{equation*}

So $g \notin {H_i}_{G_i}$ for sufficiently large $i$. Then, for any fixed such $i$, we recall that $\cap_j G_{ij} = H_i$, $H_{ij} < G_{ij}$, and $H_{ij} \triangleleft G_i$. Thus, $\cap_j H_{ij} \triangleleft G_i$ and furthermore $\cap_j H_{ij} < \cap_j G_{ij} = H_i$. Therefore, by definition of normal core we see that $\cap_j H_{ij} < {H_i}_{G_i}$. As $H_{ij}$ is a decreasing sequence, we have $g \notin H_{ij}$ for sufficiently large $j$. Now recall that $\widetilde{\pi}(g) \in \prod_{i \rightarrow \omega} \prod_{j \rightarrow \omega} \prod_{k \rightarrow \omega} \prod_{l \rightarrow \omega} L(G_i/H_{ij}) \bar{\otimes} \mathbb{M}_{n(k)} \bar{\otimes} \mathbb{M}_{m(l)} \bar{\otimes} L(H)$ is of the form $(gH_{ij} \otimes x_{ijkl})^\circ$ for some $x_{ijkl} \in \mathbb{M}_{n(k)} \bar{\otimes} \mathbb{M}_{m(l)} \bar{\otimes} L(H)$. Since for any fixed sufficiently large $i$, and then for any large enough $j$, $gH_{ij} \neq H_{ij}$, the element above has trace 0, as expected.

Now suppose $g \in H_G$ but $g \neq e$. $\pi_1$ and $\pi_{2, i}$ send $g$ to $(gH_{ij} \otimes g)_{i, j} \in \oplus_i \oplus_j L(G_i/H_{ij}) \odot \mathbb{C}[G]$ for large $i$. We focus on the second part $g \in \mathbb{C}[G]$. We note that it is sent by $\pi_3: \mathbb{C}[G] \rightarrow \oplus_k \mathbb{M}_{n(k)} \odot \mathbb{C}[G']$ to a sequence of diagonal matrices with entries in $H$. Indeed, the matrices in the sequence $\pi_3(g)$ are of the form $(1_{\widetilde{g_s}^{-1}g\widetilde{g_t} \in G'}\widetilde{g_s}^{-1}g\widetilde{g_t})_{1 \leq s, t \leq n(k)}$ where $\{\widetilde{g_1}, \cdots, \widetilde{g_{n(k)}}\}$ are representatives of cosets in some Følner subsets of $G/G'$. Now, for any $1 \leq t \leq n(k)$, as $H_G$ is normal in $G$, $\widetilde{g_t}^{-1}g\widetilde{g_t} \in H_G$. Since $H_G < H < G'$, $\widetilde{g_t}^{-1}g\widetilde{g_t} \in G'$, i.e., the matrices in the sequence $\pi_3(g)$ are diagonal with entries of the form $\widetilde{g_t}^{-1}g\widetilde{g_t}$. All these entries are in $H$ and, as $g \neq e$, none of the entries would be $e$.

Now, $\phi$ sends matrices in $\mathbb{C}[G']$ to matrices in $\mathbb{C}[\widetilde{G}]$. Since $\phi$ restricts to the identity on $H$, after applying $\phi$ we still have a sequence of diagonal matrices with non-identity elements of $H$ as entries. This implies $\pi_4: \mathbb{C}[\widetilde{G}] \rightarrow \oplus_l \mathbb{M}_{m(l)} \odot \mathbb{C}[H]$ would send these matrices to sequences of matrices with zero trace. Indeed, we only need to observe that for the matrices to have nonzero trace, there would need to be group identities on the diagonal, but that would mean that $\phi\circ\pi_3(g)$ has group identities on the diagonal, contradicting what we have already established. This shows that $\widetilde{\pi}(g)$ indeed has zero trace.

Finally, when $g = e$, by similar arguments as in the previous case, we see that $\widetilde{\pi}(g)$ is the identity in $\prod_{i \rightarrow \omega} \prod_{j \rightarrow \omega} \prod_{k \rightarrow \omega} \prod_{l \rightarrow \omega} L(G_i/H_{ij}) \bar{\otimes} \mathbb{M}_{n(k)} \bar{\otimes} \mathbb{M}_{m(l)} \bar{\otimes} L(H)$. This concludes the proof.
\end{proof}

We let $C_c(G, N)$ denote the span of $ng \in N \rtimes_{vN} G$ where $n \in N$ and $g \in G$. Clearly $C_c(G, N)$ is a *-algebra and weakly dense in $N \rtimes_{vN} G$.

\begin{lemma}
Let $N$ be a finite von Neumann algebra, $G$ be a group acting through trace-preserving automorphisms on $N$, $H_1 < G$, and the action of $G$ on $N$ is either amenably $H_1$-approximable or matricially $H_1$-approximable, then there exists a *-homomorphism $C_c(G, N) \rightarrow (R \bar{\otimes} N \rtimes_{vN} H_1)^\omega$ which restricts to a tracial embedding $N \hookrightarrow (R \bar{\otimes} N \rtimes_{vN} H_1)^\omega$.
\end{lemma}

\begin{proof} We will only prove the lemma for the case where the action is amenably $H_1$-approximable. The case where the action is matricially $H_1$-approximable follows similar ideas and is relatively easy. Now, by definition there exists a sequence of maps $\beta_i: G \rightarrow G_i$ which is approximately multiplicative and actions $\alpha_i: G_i \rightarrow Aut_\tau(N)$ s.t. $\alpha_i \circ \beta_i(g)(n) \rightarrow \alpha(g)(n)$ in $L^2(N)$ as $i \rightarrow \omega$. One may then see that the following map is approximately a *-homomorphism as $i \rightarrow \omega$,
\begin{equation*}
    \pi_1: C_c(G, N) \rightarrow \oplus_i C_c(G_i, N), C_c(G, N) \ni ng \mapsto (n\beta_i(g))_i \in \oplus_i C_c(G_i, N)
\end{equation*}

The actions $\alpha_i$ are amenably $H_1$-inner, so there exists co-amenable subgroups $G_i' < G_i$ s.t. $G_i' < U(N \rtimes_{vN} H_1) \cap \mathcal{N}(N)$ and $\alpha_i\restriction_{G_i'}$ is the conjugation action of $U(N \rtimes_{vN} H_1) \cap \mathcal{N}(N)$ on $N$. Note here that $N \rtimes_{vN} H_1$ is a subalgebra of $N \rtimes_{vN} G_i$, with $H_1$ acting on $N$ through $\alpha_i\restriction_{H_1}$, but, as $\alpha_i\restriction_{H_1} = \alpha\restriction_{H_1}$, it is also a subalgebra of $N \rtimes_{vN} G$ with $H_1$ acting on $N$ through $\alpha\restriction_{H_1}$. Thus, it is the correct $N \rtimes_{vN} H_1$ as in the statement of the lemma. Now, using methods in the proof of Theorem 3.1 gives a map,
\begin{equation*}
    \pi_{2, i}: C_c(G_i, N) \rightarrow \oplus_j \mathbb{M}_{n(j)} \odot C_c(G_i', N)
\end{equation*}
which is multiplicative after taking the ultraproduct. Then, there exists a natural map from $C_c(G_i', N)$ to $N \rtimes_{vN} H_1$ defined by,
\begin{equation*}
    \pi_{3, i}: C_c(G_i', N) \rightarrow N \rtimes_{vN} H_1, C_c(G_i', N) \ni ng \mapsto ng \in N \rtimes_{vN} H_1
\end{equation*}

Here, on the RHS, as $g \in G_i'$ and $G_i' < U(N \rtimes_{vN} H_1) \cap \mathcal{N}(N)$, $g$ is interpreted as a unitary in $N \rtimes_{vN} H_1$. Since $\alpha_i\restriction_{G_i'}$ is the conjugation action of $U(N \rtimes_{vN} H_1) \cap \mathcal{N}(N)$ on $N$, this map is clearly a *-homomorphism.

Finally, combining the above maps together and changing direct sums into ultraproducts gives,
\begin{equation*}
    \widetilde{\pi}: C_c(G, N) \rightarrow \prod_{i \rightarrow \omega} \prod_{j \rightarrow \omega} \mathbb{M}_{n(j)} \otimes N \rtimes_{vN} H_1
\end{equation*}

It is easy to verify that this map satisfies the requirements of the lemma.
\end{proof}

\begin{proof}[Proof of Theorem 5.1] By Theorem 4.2 and Proposition 2.2, it suffices to show that $N \rtimes_{vN} G$ tracially embeds into $(R \bar{\otimes} N \rtimes_{vN} H_1 \bar{\otimes} L(H))^\omega$. We begin by applying the comultiplication map,
\begin{equation*}
    \pi: C_c(G, N) \rightarrow C_c(G, N) \odot L(G), C_c(G, N) \ni ng \mapsto ng \otimes g \in C_c(G, N) \odot L(G)
\end{equation*}

By Lemma 5.2, $L(G)$ tracially embeds into $(R \bar{\otimes} L(H))^\omega$. By Lemma 5.3, there is a *-homomorphism sending $C_c(G, N)$ to $(R \bar{\otimes} N \rtimes_{vN} H_1)^\omega$. Combining the maps together gives,
\begin{equation*}
    \widetilde{\pi}: C_c(G, N) \rightarrow (R \bar{\otimes} N \rtimes_{vN} H_1)^\omega \bar{\otimes} (R \bar{\otimes} L(H))^\omega
\end{equation*}

This is a *-homomorphism. To show it is tracial, we note that given $ng \in C_c(G, N)$, if $g \neq e$, $ng$ is sent to $ng \otimes g$ by $\pi$ and $g$, having zero trace in $L(G)$, is then sent to some traceless element of $(R \bar{\otimes} L(H))^\omega$, so $\widetilde{\pi}(ng)$ has trace zero. If $g = e$, then $n$ is sent to $n \otimes 1$ by $\pi$. Since the map defined in Lemma 5.3 preserves the trace on $N$, we then see that $\widetilde{\pi}(n)$ indeed has trace $\tau(n)$. Thus, $\widetilde{\pi}$ is a tracial *-homomorphism and therefore extends to a tracial embedding $N \rtimes_{vN} G \hookrightarrow (R \bar{\otimes} N \rtimes_{vN} H_1)^\omega \bar{\otimes} (R \bar{\otimes} L(H))^\omega$.
\end{proof}

In case where $N = \mathbb{C}$, as all actions on $\mathbb{C}$ is amenably approximable, the theorem reduces to Lemma 5.2. On the other hand, suppose $G$ is hyperlinear and $H = \{e\}$, then $H$ is co-hyperlinear in $G$ and $G$ retracts onto $H$. If we furthermore take $H_1 = \{e\}$, then the theorem reduces to,

\begin{col}
Let $N$ be a finite von Neumann algebra, $G$ be a hyperlinear group acting through trace-preserving automorphisms on $N$, and the action of $G$ on $N$ is either amenably approximable or matricially approximable, then $N \subseteq N \rtimes_{vN} G$ is RE/$N$.
\end{col}

\begin{ex}
By the Corollary 5.1 and Proposition 5.2, we have $L^\infty(X, \mu) \subseteq L^\infty(X, \mu) \rtimes_{vN} G$ is RE/$L^\infty(X, \mu)$ whenever $(X, \mu)$ is a standard probability space, $G$ is a hyperlinear group, and the p.m.p. action $G \curvearrowright X$ is profinite. Since $L^\infty(X, \mu)$ is QWEP, the result can be strengthened to $L^\infty(X, \mu) \subseteq L^\infty(X, \mu) \rtimes_{vN} G$ is RE/$\mathbb{C}$. By Proposition 5.3, we also have $\mathcal{M}(L^\infty(X, \mu) \bar{\otimes} N) \subseteq \mathcal{M}(L^\infty(X, \mu) \bar{\otimes} N) \rtimes_{vN} G$ is RE/$\mathcal{M}(L^\infty(X, \mu) \bar{\otimes} N)$ under the above assumptions. Note that when $N$ is QWEP, it is easy to see that $\mathcal{M}(L^\infty(X, \mu) \bar{\otimes} N)$ is also QWEP and thus $\mathcal{M}(L^\infty(X, \mu) \bar{\otimes} N) \subseteq \mathcal{M}(L^\infty(X, \mu) \bar{\otimes} N) \rtimes_{vN} G$ is RE/$\mathbb{C}$ in this case.

Similarly, using Proposition 5.4, we have $L(H)^{\bar{\otimes} X} \subseteq L(H)^{\bar{\otimes} X} \rtimes_{vN} G$ is RE/$L(H)^{\bar{\otimes} X}$ whenever $X$ is a compact topological space, $G$ is a hyperlinear group, the action $G \curvearrowright X$ is profinite, and $H$ is a discrete group. By Lemma 5.1, $L(H)^{\bar{\otimes} X}$ has an increasing sequence of subalgebras isomorphic to $L(H)^{\bar{\otimes} X_i}$ whose union is weakly dense in $L(H)^{\bar{\otimes} X}$. Each $X_i$ is a finite set, so all such $L(H)^{\bar{\otimes} X_i}$ are embedded in $L(H)^{\bar{\otimes}\infty}$. Therefore, $L(H)^{\bar{\otimes} X}$ tracially embeds into $(L(H)^{\bar{\otimes}\infty})^\omega$, so $L(H)^{\bar{\otimes} X} \subseteq L(H)^{\bar{\otimes} X} \rtimes_{vN} G$ is RE/$L(H)^{\bar{\otimes}\infty}$.

By [Ioa11], an example of profinite actions arises when we consider a residually finite group acting on its profinite completion, a compact group equipped with the Haar measure. We also note that residually finite groups are hyperlinear. Therefore, if we let $G$ be a residually finite group, $G_i$ be a decreasing sequence of normal subgroups of $G$ whose intersection is trivial, and $\mu$ be the Haar measure on the profinite completion $\varprojlim G/G_i$ of $G$, then $L^\infty(\varprojlim G/G_i, \mu) \subseteq L^\infty(\varprojlim G/G_i, \mu) \rtimes_{vN} G$ is RE/$\mathbb{C}$ where the action of $G$ on $L^\infty(\varprojlim G/G_i, \mu)$ is induced by the left multiplication action of $G$ on $\varprojlim G/G_i$. Moreover, we also have $\mathcal{M}(L^\infty(X, \mu) \bar{\otimes} N) \subseteq \mathcal{M}(L^\infty(X, \mu) \bar{\otimes} N) \rtimes_{vN} G$ is RE/$\mathcal{M}(L^\infty(X, \mu) \bar{\otimes} N)$ under the above assumptions or RE/$\mathbb{C}$ when we further assume $N$ is QWEP. We also have, similarly, that $L(H)^{\bar{\otimes} \varprojlim G/G_i} \subseteq L(H)^{\bar{\otimes} \varprojlim G/G_i} \rtimes_{vN} G$ is RE/$L(H)^{\bar{\otimes}\infty}$ for any discrete group $H$.

Free groups are residually finite, so the above applies in particular to the case where $G$ is a free group. (This is a special case of [Hall49, Theorem 5.1].) In this case, we can build upon the example obtained from Proposition 5.3 somewhat. Indeed, in [Jun21], the algebra associated with the noncommutative Poisson random variable is not just defined for finite von Neumann algebras but more generally for semifinite von Neumann algebra. The Poisson algebra thus defined is always finite even if the original algebra is not. See [Jun21, Theorem 2.11]. In particular, we may consider the action of $G$ on $M = L^\infty(\varprojlim G/G_i, \mu) \bar{\otimes} L(G) \bar{\otimes} \mathbb{B}(l_2(G))$, which is obtained from tensoring three actions: the action of $G$ on $L^\infty(\varprojlim G/G_i, \mu)$ induced by the left multiplication action of $G$ on $\varprojlim G/G_i$, the inner action of $G$ on $L(G)$ induced by the conjugation action of $G$ on itself, and the inner action of $G$ on $\mathbb{B}(l_2(G))$ given by the left regular representation $G \rightarrow U(\mathbb{B}(l_2(G)))$. The corresponding action of $G$ on $\mathcal{M}(M)$ is amenably approximable. Indeed, by [Jun21, Theorem 2.11], $\mathcal{M}(M)$ is a weak closure of the direct limit of $\mathcal{M}(e_n M e_n)$, where we may choose $e_n$ to be an increasing sequence of finite-rank projections in $\mathbb{B}(l_2(G))$ strongly converging to the identity. For a unitary $u$ in $\mathbb{B}(l_2(G))$, we may approximate it strongly by unitaries $u_n$ in $\mathbb{B}(l_2(G))$ of the form $v_n + (1 - e_n)$ where $v_n \in U(e_n \mathbb{B}(l_2(G)) e_n)$. Indeed, this can be done by taking $a = -i \ln(u)$, which is self-adjoint. We may then let $u_n = e^{i e_n a e_n}$. One easily verifies that $u_n$ has the correct form and, as the exponential function is strongly continuous, $u_n = e^{i e_n a e_n}$ approximates $e^{ia} = u$ strongly. We also observe that $u_n^* = e^{-i e_n a e_n}$ approximates $e^{-ia} = u^*$ strongly as well. In particular, for each free generator $x_m$ of $G$, we may approximate the corresponding unitaries $u_m \in U(\mathbb{B}(l_2(G)))$ by unitaries $u_{mn} \in U(\mathbb{B}(l_2(G)))$ of the form $v_{mn} + (1 - e_n)$ where $v_{mn} \in U(e_n \mathbb{B}(l_2(G)) e_n)$. Let $\widetilde{u_{mn}} = x_m \otimes u_{mn} \in U(L(G) \bar{\otimes} \mathbb{B}(l_2(G)))$. $\widetilde{u_{mn}}$ induces an inner automorphism $Ad_{\widetilde{u_{mn}}}$ of $L(G) \bar{\otimes} \mathbb{B}(l_2(G))$. Tensoring with the trivial action on $L^\infty(\varprojlim G/G_i, \mu)$, it in turn induces an automorphism of $\mathcal{M}(M)$. We claim this automorphism is inner. Indeed, let $a_{mn} = -i e_n \ln(x_m \otimes u_m) e_n$. Then as we have seen, $\widetilde{u_{mn}} = e^{i a_{mn}}$. Recall that,
\begin{equation*}
    \mathcal{M}(e_n M e_n) = \oplus_{k=0}^\infty \bar{\otimes}_s^k e_n M e_n = \oplus_{k=0}^\infty \bar{\otimes}_s^k (L^\infty(\varprojlim G/G_i, \mu) \bar{\otimes} L(G) \bar{\otimes} e_n \mathbb{B}(l_2(G)) e_n)
\end{equation*}

Then $a_{mn}$, regarded as a self-adjoint operator in $e_n M e_n$, gives a (possibly unbounded) self-adjoint operator $\lambda(a_{mn}) = ((a_{mn})^{\otimes k})$ associated with $\mathcal{M}(e_n M e_n)$. $\mathcal{M}(e_n M e_n) \subseteq \mathcal{M}(M)$, so we may consider $U_{mn} = e^{i\lambda(a_{mn})} \in U(\mathcal{M}(M))$. One may then verify that the inner automorphism given by this unitary is the same as the automorphism induced by $Ad_{\widetilde{u_{mn}}}$. (It is easy to check this on $\mathcal{M}(e_{n'} M e_{n'})$ when $n' \geq n$. Since $\mathcal{M}(M)$ is a weak closure of the direct limit of $\mathcal{M}(e_{n'} M e_{n'})$, this gives the result on $\mathcal{M}(M)$.) We also note that, 

Let $\beta_{1i}$ denote the action $G \curvearrowright G/G_i$, $\alpha'_i$ be the action of $\beta_{1i}(G)$ on $L^\infty(\varprojlim G/G_i, \mu)$ as in Proposition 5.2, and $\beta_{2n}: G \rightarrow U(\mathcal{M}(M))$ be the group homomorphism that sends $x_m$ to $U_{mn}$. We consider two parts of the action of $G$ on $M$. The first is induced by the action of $G$ on $L^\infty(\varprojlim G/G_i, \mu)$ tensoring with the trivial action on $L(G) \bar{\otimes} \mathbb{B}(l_2(G))$. If we let $\alpha_i: \beta_{1i}(G) \curvearrowright \mathcal{M}(M)$ be the action induced by $\alpha'_i: \beta_{1i}(G) \curvearrowright L^\infty(\varprojlim G/G_i, \mu)$, then this action is approximated by $\alpha_i \circ \beta_{1i}$. The second part of the action is induced by the action of $G$ on $L(G) \bar{\otimes} \mathbb{B}(l_2(G)))$ tensoring with the trivial action on $L^\infty(\varprojlim G/G_i, \mu)$. We note that this is an inner action. We have the following group homomorphism:
\begin{equation*}
    \beta_{in}: G \rightarrow \beta_{1i}(G) \oplus \beta_{2n}(G), \beta_{in}(g) = (\beta_{1i}(g), \beta_{2n}(g))
\end{equation*}

We also have an action of $\beta_{1i}(G) \oplus \beta_{2n}(G)$ on $\mathcal{M}(M)$. Indeed, on the first component, this is given by $\alpha_i: \beta_{1i}(G) \curvearrowright \mathcal{M}(M)$, while on the second component, since $\beta_{2n}(G) < U(\mathcal{M}(M))$, the action is simply given by the restriction of the conjugation action. As the first action is induced by an action on $L^\infty(\varprojlim G/G_i, \mu)$ and the second action is induced by an action on $L(G) \bar{\otimes} \mathbb{B}(l_2(G))$, they commute. Thus, the following action is well-defined:
\begin{equation*}
    \alpha_{in}: \beta_{1i}(G) \oplus \beta_{2n}(G) \rightarrow tpAut(\mathcal{M}(M)), \alpha_{in}(g_1, g_2) = \alpha_i(g_1) \circ Ad_{g_2}
\end{equation*}

Since $\beta_{1i}(G)$ is a finite group and $\beta_{2n}(G) < U(\mathcal{M}(M))$, $\beta_{1i}(G) \oplus \beta_{2n}(G)$ contains a co-amenable subgroup which can be identified with a subgroup of $U(\mathcal{M}(M))$ and whose action is given bu the restriction of the conjugation action. One may also easily verify that $\alpha_{in} \circ \beta_{in}(g)(m) \rightarrow \alpha(g)(m)$ for all $g \in G$ and $m \in \mathcal{M}(M)$ in $L^2$ as $i \rightarrow \infty$ and then $n \rightarrow \infty$. This shows that the action is indeed amenably approximable.

As $M = L^\infty(\varprojlim G/G_i, \mu) \bar{\otimes} L(G) \bar{\otimes} \mathbb{B}(l_2(G))$ is QWEP, $\mathcal{M}(M)$ is QWEP as well. Hence, by Corollary 5.1, $\mathcal{M}(M) \subseteq \mathcal{M}(M) \rtimes_{vN} G$ is RE/$\mathbb{C}$. 
\end{ex}

\begin{ex}
By Corollary 5.1 and Proposition 5.5, if a hyperlinear group $G$ acts on a countable discrete set $X$ and if there exists an increasing sequence $G_i$ of subgroups of $G$ s.t. $\cup_i G_i = G$ and all orbits of the restricted action $G_i \curvearrowright X$ are finite, then $N^{\bar{\otimes} X} \subseteq N^{\bar{\otimes} X} \rtimes_{vN} G$ is RE/$N^{\bar{\otimes} X}$. $N^{\bar{\otimes} X} = N^{\bar{\otimes}\infty}$, so we could write this as $N^{\bar{\otimes} X} \subseteq N^{\bar{\otimes} X} \rtimes_{vN} G$ is RE/$N^{\bar{\otimes}\infty}$. Similarly, using Proposition 5.6, $N^{\bar{\otimes} X} \subseteq N^{\bar{\otimes} X} \rtimes_{vN} G$ is always RE/$N^{\bar{\otimes}\infty}$ whenever $X$ is a countable discrete set, $G$ is a free group, and the action of $G$ on $N^{\bar{\otimes} X}$ is induced by an action $G \curvearrowright X$.
\end{ex}

It is reasonable to question at this point whether the approximate retraction assumption and the condition on the action $G \curvearrowright N$ introduced in this section are necessary to obtain the results here. The author is currently not aware whether a converse exists in general or whether we would actually have $N \rtimes_{vN} H \subseteq N \rtimes_{vN} G$ is RE/$N \rtimes_{vN} H$ whenever $H$ is $\sigma$-co-hyperlinear in $G$. However, there are indications of obstacles to obtaining such results without some conditions on the retraction and the action. We will first provide an example to illustrate some difficulties involved if one wishes to remove the retraction assumption. After that, we will prove that conditions on the action $G \curvearrowright N$ are necessary in case $N = R$.

\begin{ex}
At the beginning of Section 4, we mentioned that it is possible to apply the method in Section 3 of constructing explicit matrix models to the case where residually finite groups are involved. In particular, we consider this situation here: Suppose $H \triangleleft G$ and $G/H$ is residually finite, whether it is possible to show without any extra assumption that $L(H) \subseteq L(G)$ is RE/$L(H)$. We begin by noting that, since residually finite groups are hyperlinear, applying Theorem 4.2 immediately gives $L(H) \subseteq L(G)$ is RE/$L(G)$. Alternatively, one may consider the following approach. Since $G/H$ is residually finite, we have a decreasing sequence $G_i$ of normal subgroups of $G$ s.t. $\cap_i G_i = H$. Then we naturally have a map $L(G) \hookrightarrow \mathbb{M}_{|G: G_i|} \otimes L(G_i)$ as in Lemma 3.1. Taking the ultraproduct gives a map $L(G) \hookrightarrow \prod_\omega \mathbb{M}_{|G: G_i|} \otimes L(G_i)$. One then easily verifies that we have a commuting square

\begin{center}
\begin{tikzcd}[node distance = 1.8cm]
    {\prod_\omega \mathbb{M}_{|G: G_i|} \otimes L(G_i)} \arrow[hookleftarrow]{r}\arrow[hookleftarrow]{d}
        & {L(G)} \arrow[hookleftarrow]{d} \\
    {\prod_\omega l_\infty^{|G: G_i|} \otimes L(G_i)} \arrow[hookleftarrow]{r}
        & {L(H)}
\end{tikzcd}
\end{center}

However, this only gives RE/$L(G)$. But it is reasonable to consider whether the following is possible: We replace $L(G) \hookrightarrow \prod_\omega \mathbb{M}_{|G: G_i|} \otimes L(G_i)$ by the map $\mathbb{C}[G] \rightarrow \oplus_i \mathbb{M}_{|G: G_i|} \odot \mathbb{C}[G_i]$ and construct some maps $\pi_i: \mathbb{C}[G_i] \rightarrow L(H)$ s.t. after combining all maps together and taking the ultraproduct, we can obtain a tracial embedding $L(G) \hookrightarrow \prod_\omega \mathbb{M}_{|G: G_i|} \otimes L(H)$. Specifically, it might seem reasonable to conjecture that taking $\pi_i = E_{L(G_i), L(H)}$ would work. We now demonstrate that this is not the case.

We begin by defining, for a group $G$, the group $G^\infty = \{f: \mathbb{N} \rightarrow G\}$ with pointwise operations. Let $S_3$ be the symmetric group of degree 3 and $A_3$ be the alternating group of degree 3. Note that $A_3 \simeq \mathbb{Z}/3$ so we may write $A_3 = \{e, g, g^2\}$. (It does not matter for latter arguments which specific element we choose for $g$. Both $g = (1 2 3)$ and $g = (1 3 2)$ would work.) Then we may write $S_3 = \{e, g, g^2, a, ag, ag^2\}$. (Again, it does not matter which specific element we choose for $a$. It can either be $(1 2)$, $(1 3)$, or $(2 3)$.) Note here that $A_3$ is abelian, $a^{-1} = a$, and $aga = g^2$.

We then have $H = A_3^\infty$ is a normal subgroup of $G = S_3^\infty$. Furthermore, $G/H$ is residually finite. Indeed, one may choose,
\begin{equation*}
    G_i = \{f: \mathbb{N} \rightarrow S_3 | f(q) \in A_3, \forall q \leq i\}
\end{equation*}
which is a decreasing sequence of normal subgroups of $G$ containing $H$ s.t. $|G: G_i| = 2^i$ and $\cap_i G_i = H$. We fix here the representatives of the cosets in $G/G_i$ as the set,
\begin{equation*}
    \widetilde{G/G_i} = \{f: \mathbb{N} \rightarrow S_3 | f(q) = e, \forall q > i; f(q) = e\textrm{ or }a, \forall q \leq i\}    
\end{equation*}

We now have the following proposition,

\begin{prop}
Given $G$, $H$, $G_i$, $\widetilde{G/G_i}$ as defined above, we then have a map $\mathbb{C}[G] \rightarrow \oplus_i \mathbb{M}_{2^i} \odot \mathbb{C}[G_i]$. Suppose we also have a sequence of maps $\pi_i: \mathbb{C}[G_i] \rightarrow L(H)$. Combining the maps together and changing the direct sum into an ultraproduct gives a map $\pi: \mathbb{C}[G] \rightarrow \prod_\omega \mathbb{M}_{2^i} \otimes L(H)$. Assuming $\pi_i$ restricts to the natural inclusion $\mathbb{C}[H] \hookrightarrow L(H)$, then $\pi$ is not a tracial *-homomorphism.
\end{prop}

\begin{proof} Assume to the contrary that choosing some maps $\pi_i: \mathbb{C}[G_i] \rightarrow L(H)$ can result in $\pi$ being a tracial *-homomorphism. Consider the elements $\widetilde{a} = (a, a, \cdots) \in G$, $\widetilde{g} = (g, g, \cdots) \in G$, and $\widetilde{g^2} = (g^2, g^2, \cdots) \in G$. Now, $\widetilde{a}\widetilde{g}\widetilde{a} = \widetilde{g^2}$, so $\pi(\widetilde{a})\pi(\widetilde{g})\pi(\widetilde{a}) = \pi(\widetilde{g^2})$. Consider first the part $\pi(\widetilde{a})\pi(\widetilde{g})$. The composition $\phi_i$ of the natural map $\mathbb{C}[G] \rightarrow \mathbb{M}_{2^i} \odot \mathbb{C}[G_i]$ and $Id_{\mathbb{M}_{2^i}} \otimes \pi_i$ sends $\widetilde{g}$ to a diagonal matrix,
\begin{equation*}
    (g_{\alpha\beta})_{\alpha, \beta \in \widetilde{G/G_i}} = (\delta_{\alpha\beta} \pi_i(\beta^{-1}\widetilde{g}\beta))_{\alpha, \beta \in \widetilde{G/G_i}}
\end{equation*}

Similarly, $\phi_i(\widetilde{a})$ is of the form,
\begin{equation*}
    (a_{\alpha\beta})_{\alpha, \beta \in \widetilde{G/G_i}} = (1_{\widetilde{a}\beta \in \alpha G_i} \pi_i(\alpha^{-1}\widetilde{a}\beta))_{\alpha, \beta \in \widetilde{G/G_i}}
\end{equation*}

Noting that the range of $\pi_i$ is in $L(H)$, which is commutative, we have,
\begin{equation*}
\begin{split}
    (\phi_i(\widetilde{a})\phi_i(\widetilde{g}))_{\alpha\beta} &= 1_{\widetilde{a}\beta \in \alpha G_i} \pi_i(\alpha^{-1}\widetilde{a}\beta)\pi_i(\beta^{-1}\widetilde{g}\beta)\\
    &= 1_{\widetilde{a}\beta \in \alpha G_i}\pi_i(\beta^{-1}\widetilde{g}\beta)\pi_i(\alpha^{-1}\widetilde{a}\beta)\\
    &= 1_{\widetilde{a}\beta \in \alpha G_k}\pi_i(\alpha^{-1}(\beta\alpha^{-1})^{-1}\widetilde{g}(\beta\alpha^{-1})\alpha)\pi_i(\alpha^{-1}\widetilde{a}\beta)
\end{split}
\end{equation*}

Note here that $\beta \in \widetilde{G/G_i}$ is of the form $\beta(q) = e$ for all $q > i$ and $\beta(q) = e$ or $a$ for all $q \leq i$. Hence, $(\widetilde{a}\beta)(q) = a$ for $q > i$; and for $q \leq i$, if $\beta(q) = e$ then $(\widetilde{a}\beta)(q) = a$, and if $\beta(q) = a$ then $(\widetilde{a}\beta)(q) = e$. As $G_i = \{f: \mathbb{N} \rightarrow S_3 | f(q) \in A_3, \forall q \leq i\}$ and $\widetilde{a}\beta \in \alpha G_i$, we must have $\alpha(q) = e$ for $q > i$; and for $q \leq i$, if $\beta(q) = e$ then $\alpha(q) = a$, and if $\beta(q) = a$ then $\alpha(q) = e$. Hence, $(\beta\alpha^{-1})(q) = e$ if $q > i$ and $(\beta\alpha^{-1})(q) = a$ if $q \leq i$. For simplicity, we write $\widetilde{g_i}$ for the element of $H$ s.t. $\widetilde{g_i}(q) = g^2$ for $q \leq i$ and $\widetilde{g_i}(q) = g$ for $q > i$, then $(\beta\alpha^{-1})^{-1}\widetilde{g}(\beta\alpha^{-1}) = \widetilde{g_i}$ and,
\begin{equation*}
\begin{split}
    (\phi_i(\widetilde{a})\phi_i(\widetilde{g}))_{\alpha\beta} &= 1_{\widetilde{a}\beta \in \alpha G_i}\pi_i(\alpha^{-1}(\beta\alpha^{-1})^{-1}\widetilde{g}(\beta\alpha^{-1})\alpha)\pi_i(\alpha^{-1}\widetilde{a}\beta)\\
    &= 1_{\widetilde{a}\beta \in \alpha G_i}\pi_i(\alpha^{-1}\widetilde{g_i}\alpha)\pi_i(\alpha^{-1}\widetilde{a}\beta)\\
    &= (\phi_i(\widetilde{g_i})\phi_i(\widetilde{a}))_{\alpha\beta}
\end{split}
\end{equation*}

Write $\gamma = (\phi_i(\widetilde{g_i}))^\circ \in \prod_\omega \mathbb{M}_{2^i} \otimes L(H)$. We then have $\pi(\widetilde{g^2}) = \pi(\widetilde{a})\pi(\widetilde{g})\pi(\widetilde{a}) = \gamma\pi(\widetilde{a})\pi(\widetilde{a}) = \gamma$. Therefore, $\|\phi_i(\widetilde{g^2}) - \phi_i(\widetilde{g_i})\|_2 \rightarrow 0$ along the ultrafilter. But now, as $\pi_i$ restricts to the identity on $H$ and $\widetilde{g^2} \neq \widetilde{g_k}$,
\begin{equation*}
\begin{split}
    \|\phi_i(\widetilde{g^2}) - \phi_i(\widetilde{g_i})\|^2_2 &= \frac{1}{2^i} \sum_{\alpha \in \widetilde{G/G_i}} \|\pi_i(\alpha^{-1}\widetilde{g^2}\alpha) - \pi_i(\alpha^{-1}\widetilde{g_i}\alpha)\|^2_2\\
    &= \frac{1}{2^i} \sum_{\alpha \in \widetilde{G/G_i}} \|\widetilde{g^2} - \widetilde{g_i}\|^2_2\\
    &= 2
\end{split}
\end{equation*}

This gives the desired contradiction.
\end{proof}
\end{ex}

We now demonstrate that some assumptions on the action $G \curvearrowright N$ are necessary to obtain $N \subseteq N \rtimes_{vN} G$ is RE/$N$ when $N = R$. Clearly, this is equivalent to $R \rtimes_{vN} G$ being QWEP. We have the following:

\begin{thm}
Let $G$ be a hyperlinear group acting through trace-preserving automorphisms on $R$, then $R \rtimes_{vN} G$ is QWEP iff the action $\alpha: G \rightarrow tpAut(R)$ is matricially approximable.
\end{thm}

This provides the motivation for defining matricially approximable actions in the first place. To prove the theorem, we need the following lemma:

\begin{lemma}
Given a tracial embedding $\pi: R \hookrightarrow R^\omega$, there exists a tracial embedding $R^\omega \hookrightarrow \prod_{\omega'} \mathbb{M}_{n(i)} \otimes R$ s.t. the composite embedding $R \hookrightarrow R^\omega \hookrightarrow \prod_{\omega'} \mathbb{M}_{n(i)} \otimes R$ is of the form,
\begin{equation*}
    R \ni r \mapsto (1_{\mathbb{M}_{n(i)}} \otimes r)^\circ \in \prod_{\omega'} \mathbb{M}_{n(i)} \otimes R
\end{equation*}

Furthermore, the choice of $\omega'$ only depends on $\omega$.
\end{lemma}

\begin{proof} Let $R$ be the union of an increasing sequence of matrix subalgebras $\mathbb{M}_{2^i}$. Then the inclusion $\mathbb{M}_{2^i} \hookrightarrow R$ composed with $\pi$ gives an embedding $\mathbb{M}_{2^i} \hookrightarrow R^\omega$. This induces an isomorphism $R^\omega \simeq \hat{R} \otimes \mathbb{M}_{2^i}$ for some algebra $\hat{R}$. As $\hat{R} \simeq \hat{R} \otimes 1_{\mathbb{M}_{2^i}}$ tracially embeds into $R^\omega$, we obtain an embedding $R^\omega \hookrightarrow R^\omega \otimes \mathbb{M}_{2^i}$. $R^\omega \otimes \mathbb{M}_{2^i}$ is canonically isomorphic to $(R \otimes \mathbb{M}_{2^i})^\omega$. Therefore, we have,
\begin{equation*}
    \mathbb{M}_{2^i} \hookrightarrow R \hookrightarrow R^\omega \hookrightarrow (R \otimes \mathbb{M}_{2^i})^\omega
\end{equation*}
s.t. the composite embedding sends $r \in \mathbb{M}_{2^i}$ to $(1_R \otimes r)^\circ \in (R \otimes \mathbb{M}_{2^i})^\omega$. Since $\mathbb{M}_{2^i} \hookrightarrow R$, $(R \otimes \mathbb{M}_{2^i})^\omega$ embeds into $(R \bar{\otimes} R)^\omega$. We then take the ultraproduct of all composite embeddings
\begin{equation*}
    R \hookrightarrow R^\omega \hookrightarrow (R \otimes \mathbb{M}_{2^i})^\omega \hookrightarrow (R \bar{\otimes} R)^\omega
\end{equation*}
to obtain an embedding,
\begin{equation*}
    R \hookrightarrow R^\omega \hookrightarrow \prod_{i \rightarrow \omega} (R \bar{\otimes} R)^\omega
\end{equation*}

By the way this embedding is defined, we see that for any $i$, any $r \in \mathbb{M}_{2^i} \subseteq R$, $r$ is sent to $(1_R \otimes r) \in \prod_{i \rightarrow \omega} (R \bar{\otimes} R)^\omega$. As $\cup_i \mathbb{M}_{2^i}$ is weakly dense in $R$, the same must holds for all $r \in R$ as well. Finally, observe that $R$ naturally embeds into $\prod_\omega \mathbb{M}_{2^j}$. So we may embed $\prod_{i \rightarrow \omega} (R \bar{\otimes} R)^\omega$ into $\prod_{i \rightarrow \omega} (\prod_{j \rightarrow \omega} \mathbb{M}_{2^j} \otimes R)^\omega$. It is then clear that the composite embedding,
\begin{equation*}
    R \hookrightarrow R^\omega \hookrightarrow \prod_{i \rightarrow \omega} (R \bar{\otimes} R)^\omega \hookrightarrow \prod_{i \rightarrow \omega} (\prod_{j \rightarrow \omega} \mathbb{M}_{2^j} \otimes R)^\omega
\end{equation*}
satisfies the desired properties.
\end{proof}

\begin{proof}[Proof of Theorem 5.2] ($\Leftarrow$) This follows directly from Corollary 5.1.

($\Rightarrow$) Since $R \rtimes_{vN} G$ is QWEP, we have $R \hookrightarrow R \rtimes_{vN} G \hookrightarrow R^\omega$. Applying Lemma 5.4 gives,
\begin{equation*}
    R \hookrightarrow R \rtimes_{vN} G \hookrightarrow R^\omega \hookrightarrow \prod_{\omega'} \mathbb{M}_{n(i)} \otimes R
\end{equation*}
s.t. the composite embedding sends $r \in R$ to $(1_{\mathbb{M}_{n(i)}} \otimes r)^\circ \in \prod_{\omega'} \mathbb{M}_{n(i)} \otimes R$. Write the composite embedding $R \rtimes_{vN} G \hookrightarrow R^\omega \hookrightarrow \prod_{\omega'} \mathbb{M}_{n(i)} \otimes R$ as $\pi$. Then given any $g \in G$, $\pi(g)$ is a unitary in $\prod_{\omega'} \mathbb{M}_{n(i)} \otimes R$ and therefore can be lifted to a sequence of unitaries $g_i \in \mathbb{M}_{n(i)} \otimes R$. Let $\beta_i: G \rightarrow U(\mathbb{M}_{n(i)} \otimes R)$ be defined by $\beta_i(g) = g_i$. It is then easy to verify that this sequence of maps, together with the ultrafilter $\omega'$, satisfies the requirements of the definition of matricially approximable actions.
\end{proof}

\begin{ex}
As mentioned in Example 4.3, $R \rtimes_{vN} G$ is QWEP whenever $G$ is a free group. Thus, any action of a free group on $R$ is matricially approximable. This can also be proved more directly. We only consider here the case where $G$ is a finitely generated free group. The case where $G = \mathbb{F}_\infty$ can then be proved using an approximation argument. Now, as $G$ is finitely generated, we write its free generators as $\{x_1, \cdots, x_n\}$. We let $A_i = R \rtimes_{vN} \mathbb{Z}$ where the action of $\mathbb{Z}$ on $R$ is given by $\alpha(x_i) \in tpAut(R)$. $A_i$ is hyperfinite, so in particular QWEP. Fix some embedding $A_i \hookrightarrow R^\omega$. Since $R \hookrightarrow A_i \hookrightarrow R^\omega$, we may choose a tracial embedding $R^\omega \hookrightarrow \prod_{j \rightarrow \omega'} \mathbb{M}_{n(i, j)} \otimes R$ as in Lemma 5.4. Naturally, $\prod_{j \rightarrow \omega'} \mathbb{M}_{n(i, j)} \otimes R \hookrightarrow \prod_{j \rightarrow \omega'} (\otimes_{i = 1}^n \mathbb{M}_{n(i, j)}) \otimes R$. Now, we have,
\begin{equation*}
    R \hookrightarrow A_i \hookrightarrow R^\omega \hookrightarrow \prod_{\omega'} \mathbb{M}_{n(i, j)} \otimes R \hookrightarrow \prod_{j \rightarrow \omega'} (\otimes_{i = 1}^n \mathbb{M}_{n(i, j)}) \otimes R
\end{equation*}
s.t. the composite embedding is of the form,
\begin{equation*}
   R \ni r \mapsto (1_{\otimes_{i = 1}^n \mathbb{M}_{n(i, j)}} \otimes r)^\circ \in \hookrightarrow \prod_{j \rightarrow \omega'} (\otimes_{i = 1}^n \mathbb{M}_{n(i, j)}) \otimes R
\end{equation*}

Write $u_i$ for the unitary in $A_i = R \rtimes_{vN} \mathbb{Z}$ corresponding to $1 \in \mathbb{Z}$. Then $u_i$ can be regarded as a unitary in $\prod_{j \rightarrow \omega'} (\otimes_{i = 1}^n \mathbb{M}_{n(i, j)}) \otimes R$. Lift it to a sequence of unitaries $(u_{ij})$ in $(\otimes_{i = 1}^n \mathbb{M}_{n(i, j)}) \otimes R$ and define $\beta_j: G \rightarrow U((\otimes_{i = 1}^n \mathbb{M}_{n(i, j)}) \otimes R)$ by sending $x_i$ to $u_{ij}$. It is then easy to verify that this sequence of maps, together with the ultrafilter $\omega'$, satisfies the requirements of the definition of matricially approximable actions.
\end{ex}

It is possible to generalize this philosophy. Indeed, here, to show $R \rtimes_{vN} G$ is QWEP whenever $G$ is a free group and regardless of how $G$ acts on $R$, we have essentially fixed some finitely many $\alpha_1, \cdots, \alpha_n \in tpAut(R)$ (i.e., $\alpha(x_i)$'s) and construct an embedding $\pi: R \hookrightarrow R^\omega$ in which all $\alpha_i$ become inner. More precisely, there exists $u_i \in U(R^\omega)$ s.t. $u_i \pi(r) u_i^* = \pi(\alpha_i(r))$ for all $r \in R$. We do not need these $u_i$ to be freely independent. This observation leads to the following proposition:

\begin{prop}
Let $N$ be a QWEP finite von Neumann algebra. Then the following are equivalent:

\medskip

1. $N \rtimes G$ is QWEP whenever $G$ is a free group and regardless of how $G$ acts on $N$;

\medskip

2. $N \subseteq N \rtimes G$ is RE/$\mathbb{C}$ whenever $G$ is a free group and regardless of how $G$ acts on $N$;

\medskip

3. For any finitely many $\alpha_1, \cdots, \alpha_n \in tpAut(N)$, there exists an embedding $\pi: N \hookrightarrow R^\omega$ and $u_i \in U(R^\omega)$ s.t. $u_i \pi(r) u_i^* = \pi(\alpha_i(r))$ for all $i$ and $r \in R$.
\end{prop}

\begin{proof} The equivalence between 1 and 2 follows from Theorem 4.2 and Proposition 2.2. $1 \Rightarrow 3$ is obvious. For $3 \Rightarrow 1$, by Upward Limit Approximation Theorem, it suffices to consider the case where $G$ is finitely generated. Let $\{x_1, \cdots, x_n\}$ be the free generators of $G$ and let $\alpha_i = \alpha(x_i)$. Then by our assumption there exists an embedding $\pi: N \hookrightarrow R^\omega$ and $u_i \in U(R^\omega)$ s.t. $u_i \pi(r) u_i^* = \pi(\alpha_i(r))$ for all $i$ and $r \in R$. Define $\phi: G \rightarrow U(R^\omega)$ by sending $x_i$ to $u_i$. Then define $\widetilde{\pi}: N \rtimes_{vN} G \rightarrow R^\omega \bar{\otimes} L(G)$ by $\widetilde{\pi}(n) = \pi(n) \otimes 1_{L(G)}$ for all $n \in N$ and $\widetilde{\pi}(g) = \phi(g) \otimes g$ for all $g \in G$. It is then easy to verify that this is a tracial embedding. Since $L(G)$ is QWEP, $N \rtimes_{vN} G$ is thus QWEP.
\end{proof}

To apply this proposition, we need to understand $tpAut(N)$ well. However, this group is not well-understood for most interesting algebras. On the other hand, there are certain groups whose automorphism groups are well-understood, and if $G$ acts on the group instead of the group algebra, then we do not need to care about all automorphisms of the group algebra. Only those induced by automorphisms of the group are relevant. This observation yields the following proposition, whose proof is essentially the same as that of Proposition 5.8,

\begin{prop}
Let $H$ be a hyperlinear group. Then the following are equivalent:

\medskip

1. $H \rtimes G$ is hyperlinear whenever $G$ is a free group and regardless of how $G$ acts on $H$;

\medskip

2. $L(H) \subseteq L(H) \rtimes G$ is RE/$\mathbb{C}$ whenever $G$ is a free group and the action by $G$ on $L(H)$ is induced by an action of $G$ on $H$;

\medskip

3. For any finitely many $\alpha_1, \cdots, \alpha_n \in Aut(H)$, there exists an embedding $\pi: L(H) \hookrightarrow R^\omega$ and $u_i \in U(R^\omega)$ s.t. $u_i \pi(h) u_i^* = \pi(\alpha_i(h))$ for all $i$ and $h \in H$.
\end{prop}

\begin{rmk}
In both Proposition 5.8 and Proposition 5.9, it is clear that we only need to choose $\alpha_1, \cdots, \alpha_n$ among a generating set of either $tpAut(N)$ or $Aut(H)$ instead of the full group. In particular, if $tpAut(N)$ or $Aut(H)$ is finitely generated, then we only need to consider the case where $\{\alpha_1, \cdots, \alpha_n\}$ is some given generating set of $tpAut(N)$ or $Aut(H)$.
\end{rmk}

As an application of this proposition, we prove the following:

\begin{thm}
$\mathbb{F}_n \rtimes G$ is hyperlinear for all $2 \leq n < \infty$ whenever $G$ is a free group and regardless of how $G$ acts on $\mathbb{F}_n$.
\end{thm}

\begin{proof} By [MKS76, Section 3.5, Corollary N1], $Aut(\mathbb{F}_n)$ is generated by four elements $\alpha_1, \alpha_2, \alpha_3, \alpha_4$ defined by,
\begin{equation*}
\begin{split}
    &\alpha_1(x_1) = x_2, \alpha_1(x_2) = x_1, \alpha(x_i) = x_i \forall i \geq 3\\
    &\alpha_2(x_i) = x_{i+1} \forall 1 \leq i \leq n-1, \alpha_2(x_n) = x_1\\
    &\alpha_3(x_1) = x_1^{-1}, \alpha(x_i) = x_i \forall i \geq 2\\
    &\alpha_4(x_1) = x_1x_2, \alpha(x_i) = x_i \forall i \geq 2
\end{split}
\end{equation*}
where $\{x_1, \cdots, x_n\}$ are the free generators of $\mathbb{F}_n$. By Remark 5.1, it suffices to construct an embedding $L(\mathbb{F}_n) \hookrightarrow R^\omega$ in which $\alpha_1, \alpha_2, \alpha_3, \alpha_4$ become inner. To do so, we consider an embedding,
\begin{equation*}
    \pi: L(\mathbb{F}_n) \hookrightarrow \prod_{k \rightarrow \omega} \mathbb{M}_k \otimes ((A_1 \bar{\otimes} \cdots \bar{\otimes} A_n \bar{\otimes} B_1 \bar{\otimes} \cdots \bar{\otimes} B_n) \rtimes_{vN} \mathbb{F}_4)
\end{equation*}
where $A_i = B_i = L^\infty(U(k), \mu)$ and $\mu$ is the Haar measure on $U(k)$.

We shall define the action of $\mathbb{F}_4$ on $A_1 \bar{\otimes} \cdots \bar{\otimes} A_n \bar{\otimes} B_1 \bar{\otimes} \cdots \bar{\otimes} B_n$ later. For now, we simply observe that as $A_1 \bar{\otimes} \cdots \bar{\otimes} A_n \bar{\otimes} B_1 \bar{\otimes} \cdots \bar{\otimes} B_n$ is hyperfinite, regardless of how $\mathbb{F}_4$ acts, the RHS algebra tracially embeds into $R^{\omega'}$ for some ultrafilter $\omega'$. Thus, it suffices to define this embedding and show that all $\alpha_i$ become inner in the RHS algebra.

Fix $k$. For each $A_i$, let $u_i \in \mathbb{M}_k \otimes A_i$ be the Haar random unitary. Similarly, for each $B_i$, let $v_i \in \mathbb{M}_k \otimes B_i$ be the Haar random unitary. Permutation matrices in $U(k)$ form a subgroup isomorphic to $S_k$. Naturally, $\mathbb{Z}_k$ is a subgroup of $S_k$. Let $w = [1] \in \mathbb{Z}_k \subseteq S_k \subseteq U(k)$ as a unitary in $\mathbb{M}_k$. We may regard all these unitaries $u_i, v_i, w$ as unitaries in $\mathbb{M}_k \otimes ((A_1 \bar{\otimes} \cdots \bar{\otimes} A_n \bar{\otimes} B_1 \bar{\otimes} \cdots \bar{\otimes} B_n) \rtimes_{vN} \mathbb{F}_4)$. By [VDN92, Theorem 4.3.2], $(u_1wv_1^*)^\circ, \cdots, (u_nwv_n^*)^\circ \in \prod_{k \rightarrow \omega} \mathbb{M}_k \otimes ((A_1 \bar{\otimes} \cdots \bar{\otimes} A_n \bar{\otimes} B_1 \bar{\otimes} \cdots \bar{\otimes} B_n) \rtimes_{vN} \mathbb{F}_4)$ are freely independent unitaries, each with trace zero. So we may define the embedding $\pi$ by $\pi(x_i) = (u_iwv_i^*)^\circ$.

We now define the action of $\mathbb{F}_4$ on $A_1 \bar{\otimes} \cdots \bar{\otimes} A_n \bar{\otimes} B_1 \bar{\otimes} \cdots \bar{\otimes} B_n$. Let the free generators of $\mathbb{F}_4$ be $\{y_1, y_2, y_3, y_4\}$. Then,

I. $\alpha(y_1)(a_1 \otimes \cdots \otimes a_n \otimes b_1 \otimes \cdots \otimes b_n) = a_2 \otimes a_1 \otimes a_3 \otimes \cdots \otimes a_n \otimes b_2 \otimes b_1 \otimes b_3 \otimes \cdots \otimes b_n$;

II. $\alpha(y_2)(a_1 \otimes \cdots \otimes a_n \otimes b_1 \otimes \cdots \otimes b_n) = a_n \otimes a_1 \otimes a_2 \otimes \cdots \otimes a_{n-1} \otimes b_n \otimes b_1 \otimes b_2 \otimes \cdots \otimes b_{n-1}$;

III. $\alpha(y_3)$ is defined by an automorphism of $A_1 \bar{\otimes} B_1$ tensoring with the identity operators on all other $A_i$ and $B_i$. To define this automorphism on $A_1 \bar{\otimes} B_1$, we first note that $A_1 \bar{\otimes} B_1 = L^\infty(U_1 \times U_2, \mu \times \mu)$ where $U_1 = U_2 = U(k)$ and $U_1$ and $U_2$ correspond to $A_1$ and $B_1$, resp. We also note that there exists a permutation matrix $z \in \mathbb{M}_k$ s.t. $zwz^* = w^*$, as both $w$ and $w^*$ correspond to cycles of length $n$. We then define the $\alpha(y_3)$ as the automorphism on $A_1 \bar{\otimes} B_1 = L^\infty(U_1 \times U_2, \mu \times \mu)$ induced by the measure-preserving homeomorphism of $U_1 \times U_2$ given by $(u, v) \mapsto (vz, uz)$;

IV. $\alpha(y_3)$ is defined by an automorphism of $B_1 \bar{\otimes} A_2 \bar{\otimes} B_2$ tensoring with the identity operators on all other $A_i$ and $B_i$. Again, let $B_1 \bar{\otimes} A_2 \bar{\otimes} B_2 = L^\infty(U_1 \times U_2 \times U_3, \mu \times \mu \times \mu)$ where $U_1 = U_2 = U_3 = U(k)$ and $U_1$, $U_2$, and $U_3$ correspond to $B_1$, $A_2$, and $B_2$, resp. Then $\alpha(y_3)$ is induced by the measure-preserving homeomorphism of $U_1 \times U_2 \times U_3$ given by $(u, v, \eta) \mapsto (\eta w^*v^*u, v, \eta)$.

One then easily checks that $y_1, y_2, y_3, y_4$, as unitaries in $\prod_{k \rightarrow \omega} \mathbb{M}_k \otimes ((A_1 \bar{\otimes} \cdots \bar{\otimes} A_n \bar{\otimes} B_1 \bar{\otimes} \cdots \bar{\otimes} B_n) \rtimes_{vN} \mathbb{F}_4)$, implement $\alpha_1, \alpha_2, \alpha_3, \alpha_4$, resp. This concludes the proof.
\end{proof}

\medskip

\section{\sc Open Problems}

\medskip

We list some open problems related to the topic of this paper that might be interesting for further study in this section.

\begin{prob}
In Lemma 2.3, we have shown that $\ast_B A$ is QWEP whenever $A$ is a QWEP finite von Neumann algebra and $B$ is a hyperfinite subalgebra of $A$. Does the conclusion still holds if we only assume $B$ is QWEP? 
\end{prob}

By the results in Appendix, it is clear that this is equivalent to the assertion that $N \subseteq M$ is RE/$\mathbb{C}$ whenever $M$ is QWEP. Unfortunately, we could not resolve the problem one way or the other at this stage.

\begin{prob}
In Remark 3.2 we observed that while $H < G$ being co-amenable in $G$ is sufficient to show that $N \rtimes_{vN} G$ tracially embeds into $(R \bar{\otimes} N \rtimes_{vN} H)^\omega$, it is not clear whether the stronger result that $N \rtimes_{vN} H \subseteq N \rtimes_{vN} G$ is RE/$N \rtimes_{vN} H$ holds. As observed in Remark 4.3, the issue is that co-amenability does not imply co-hyperlinear, at least when assuming the hyperlinear conjecture is false. However, does there exist other approaches that might show $N \rtimes_{vN} H \subseteq N \rtimes_{vN} G$ is RE/$N \rtimes_{vN} H$?
\end{prob}

\begin{prob}
While Theorem 5.2 shows that $R \rtimes_{vN} G$ is QWEP iff $G$ is hyperlinear and the action $\alpha: G \rightarrow tpAut(R)$ is matricially approximable, it is not known whether there exists any action by a hyperlinear group on $R$ that is not matricially approximable. Equivalently, is $R \rtimes_{vN} G$ QWEP whenever $G$ is hyperlinear? Does there exists a hyperlinear group $G$ and an action by $G$ on $R$ s.t. $R \rtimes_{vN} G$ is not QWEP? Two specific cases might be of particular interest. The first is when $G$ acts on a countable discrete set $X$ so we may consider the induced action $G \curvearrowright R^{\bar{\otimes} X}$. We know this is always matricially approximable when $G$ is either amenable or a free group. Is this action matricially approximable in other scenarios? This is interesting because $R^{\bar{\otimes} X} \rtimes_{vN} G$ is a group von Neumann algebra, so a negative answer to this question will provide a counterexample to the hyperlinear conjecture. Another interesting case is similar to the situation discussed in Example 5.4. This time, we consider the inner action of a hyperlinear group $G$ on $\mathbb{B}(l_2(G))$ given by the left regular representation $G \rightarrow U(\mathbb{B}(l_2(G)))$. We may then consider the induced action $G \curvearrowright \mathcal{M}(\mathbb{B}(l_2(G)))$. It is known that $\mathcal{M}(\mathbb{B}(l_2(G))) \simeq R$. See [Jun21]. When $G$ is a free group, by similar methods as in Example 5.4, we can show that $\mathcal{M}(\mathbb{B}(l_2(G))) \rtimes_{vN} G$ is QWEP. The same, of course, works for the case where $G$ is amenable. Does this work for other hyperlinear groups $G$ as well? If it does not, then what characterizes the class of groups $G$ for which $\mathcal{M}(\mathbb{B}(l_2(G))) \rtimes_{vN} G$ is QWEP?
\end{prob}

More generally, we could consider similar questions for algebras other than $R$:

\begin{prob}
Given a finite von Neumann algebra $N$, what characterizes the actions $\alpha: G \rightarrow tpAut(N)$ for which $N \subseteq N \rtimes_{vN} G$ is RE/$N$? In case $G$ is hyperlinear, by Theorem 4.2 this reduces to characterizing the actions $\alpha: G \rightarrow tpAut(N)$ for which $N \rtimes_{vN} G$ embeds into $(R \bar{\otimes} N)^\omega$. If we consider the case where $N$ is QWEP, then we could ask whether for all QWEP $N$ and hyperlinear $G$, we always have $N \rtimes_{vN} G$ is QWEP. Or does there exists some QWEP $N$, some hyperlinear $G$ acting through trace-preserving automorphisms on $N$ s.t. $N \rtimes_{vN} G$ is not QWEP? It might be useful to first consider this problem when $N$ is some relatively well-understood factor, such as a free group factor.
\end{prob}

We could also consider what conditions does this impose on the group $G$. Specifically,

\begin{prob}
What characterizes the groups $G$ for which $R \rtimes_{vN} G$ is always QWEP regardless of how $G$ acts on $R$? We have seen that amenable groups and free groups satisfy this property. Furthermore, it is not hard to prove using Theorem 5.2 that given groups $\{G_i\}_{i \in I}$ satisfying this property, then $\ast^{i \in I} G_i$ does so as well. It is also clear that this implies $G$ is hyperlinear. However, it is still open whether this property is strictly stronger than $G$ being hyperlinear or actually equivalent to it.

More generally, what characterizes the groups for which $N \subseteq N \rtimes_{vN} G$ is always RE/$\mathbb{C}$ whenever $N$ is QWEP and regardless of how $G$ acts on $N$? This clearly implies $G$ is hyperlinear, so by Theorem 4.2 it is equivalent to ask what characterizes the groups for which $N \rtimes_{vN} G$ is always QWEP whenever $N$ is QWEP and regardless of how $G$ acts on $N$. At this point we only know that amenable groups satisfy this property and it is not known whether any other group does so. Amenable groups even satisfy the stronger property that $N \subseteq N \rtimes_{vN} G$ is always RE/$N$ regardless of whether $N$ is QWEP. (See Theorem 3.1.) One could also ask what characterizes the groups satisfying this stronger property. Is it equivalent to $G$ being amenable or strictly weaker than that?
\end{prob}

\medskip

\section{\sc Appendix}

\medskip

The aim of this appendix is to prove the following theorem:

\begin{thm}
Given $N \subseteq M$, a pair of finite von Neumann algebras, and $N_1$ a finite factor. Then the following are equivalent:

\medskip

1. $N \subseteq M$ is RE/$N_1$;

\medskip

2. $(M \bar{\otimes} L^\infty(S^1)) \ast_N M$ tracially embeds into $(R \bar{\otimes} N_1)^\omega$.

\medskip

If in addition we also have that there exists a unitary $u \in M$ s.t. $E_{M, N}(u^n) = 0$ for all $n \neq 0$, then the above is also equivalent to:

\medskip

3. $M \ast_N M$ tracially embeds into $(R \bar{\otimes} N_1)^\omega$.
\end{thm}

\begin{proof}[Proof of $1 \Rightarrow 2$ and $1 \Rightarrow 3$] $1 \Rightarrow 3$ is immediate from Corollary 2.1. For $1 \Rightarrow 2$, we note that $\mathbb{C} \subseteq L^\infty(S^1)$ is RE/$\mathbb{C}$, so part 2 of Proposition 2.3 gives $N \subseteq M \bar{\otimes} L^\infty(S^1)$ is RE/$N_1$. Applying Corollary 2.1 then gives the result.
\end{proof}

To prove $2 \Rightarrow 1$ and $3 \Rightarrow 1$, we need the following lemmas:

\begin{lemma}
Suppose $\hat{M}$ and $N_1$ are finite von Neumann algebras and $N_1$ is in addition a factor. Suppose further that we have a tracial inclusion $\hat{M} \subseteq (R \bar{\otimes} N_1)^\omega$ where $R$ is the hyperfinite $\textrm{II}_1$ factor and $\omega$ is an ultrafilter. If there exists a unitary $u \in \hat{M}$ of the form $u = \sum_{l=1}^m e^{2i\pi l/m} p_l$, where $p_l$ are orthogonal projections of trace $\frac{1}{m}$, then there exists a tracial inclusion $(R \bar{\otimes} N_1)^\omega \subseteq (\mathbb{M}_m \otimes R \bar{\otimes} N_1)^\omega$ s.t. that composite inclusion $\hat{M} \subseteq (R \bar{\otimes} N_1)^\omega \subseteq (\mathbb{M}_m \otimes R \bar{\otimes} N_1)^\omega$ sends $u$ to some $(u_n)^\circ$ and $\langle u \rangle' \cap \hat{M}$ into $\prod_{n \rightarrow \omega} \langle u_n \rangle'$.
\end{lemma}

\begin{proof} Since $N_1$ is a factor, so is $(R \bar{\otimes} N_1)^\omega$. Thus, we may construct partial isometries between $p_l$'s and obtain a matrix subalgebra $\mathbb{M}_m$ of $(R \bar{\otimes} N_1)^\omega$. We therefore obtain a natural isomorphism $(R \bar{\otimes} N_1)^\omega \simeq \mathbb{M}_m \otimes \hat{N}$ for some algebra $\hat{N}$ which sends $p_l$ to $e_{ll} \otimes 1_{\hat{N}}$. As $\hat{N} \simeq 1_{\mathbb{M}_m} \otimes \hat{N}$ tracially embeds into $(R \bar{\otimes} N_1)^\omega$, we obtain an embedding $(R \bar{\otimes} N_1)^\omega \hookrightarrow \mathbb{M}_m \otimes (R \bar{\otimes} N_1)^\omega$. $\mathbb{M}_m \otimes (R \bar{\otimes} N_1)^\omega$ is canonically isomorphic to $(\mathbb{M}_m \otimes R \bar{\otimes} N_1)^\omega$. Therefore, we have,
\begin{equation*}
    (R \bar{\otimes} N_1)^\omega \hookrightarrow (\mathbb{M}_m \otimes R \bar{\otimes} N_1)^\omega
\end{equation*}
where $p_l$ now becomes $(e_{ll} \otimes 1_R \otimes 1_{N_1})^\circ$ and therefore $u$ becomes,
\begin{equation*}
    u = (\sum_{l=1}^m e^{\frac{2i\pi l}{m}} e_{ll} \otimes 1_R \otimes 1_{N_1})^\circ
\end{equation*}

Let $\hat{N_1} = (\mathbb{M}_m \otimes R \bar{\otimes} N_1)^\omega$ and $u_n = \sum_{l=1}^m e^{2i\pi l/m} e_{ll} \otimes 1_R \otimes 1_{N_1} \in \mathbb{M}_m \otimes R \bar{\otimes} N_1$. Now, given any $x \in \langle u \rangle' \cap \hat{M}$, we let its image under the composite embedding $\hat{M} \hookrightarrow (R \bar{\otimes} N_1)^\omega \hookrightarrow \hat{N_1}$ be $(x_n)^\circ$. To show this is an element of $\prod_{n \rightarrow \omega} \langle u_n \rangle'$, it suffices to show $(E_{\mathbb{M}_m \otimes R \bar{\otimes} N_1, \langle u_n \rangle'}(x_n))^\circ = (x_n)^\circ$. We note that elements of $\mathbb{M}_m \otimes R \bar{\otimes} N_1$ can be written as $m \times m$-matrices with entries in $R \bar{\otimes} N_1$. Regarded as such, we see that $\langle u_n \rangle'$ is simply the diagonal subalgebra $l_\infty^m \otimes R \bar{\otimes} N_1$. For fixed $n$, we write $x_n$ in matrix form as $(x_{nij})_{1 \leq i,j \leq m}$. Thus,
\begin{equation*}
    x_n - E_{\mathbb{M}_m \otimes R \bar{\otimes} N_1, \langle u_n \rangle'}(x_n) = (\delta_{i \neq j}x_{nij})_{1 \leq i,j \leq m}
\end{equation*}

On the other hand,
\begin{equation*}
    [u_n, x_n] = ((e^{\frac{2i\pi i}{m}} - e^{\frac{2i\pi j}{m}})x_{nij})_{1 \leq i,j \leq m}
\end{equation*}

A calculation then yields,
\begin{equation*}
\begin{split}
    \|[u_n, x_n]\|_2^2 &= \frac{1}{m} \sum_{i, j} |e^{\frac{2i\pi i}{m}} - e^{\frac{2i\pi j}{m}}|^2 \|x_{nij}\|_2^2\\
    &= \frac{1}{m} \sum_{i \neq j} |e^{\frac{2i\pi (i - j)}{m}} - 1|^2 \|x_{nij}\|_2^2\\
    &\geq \frac{1}{m} \sum_{i \neq j} |e^{\frac{2i\pi}{m}} - 1|^2 \|x_{nij}\|_2^2\\
    &= |e^{\frac{2i\pi}{m}} - 1|^2 \frac{1}{m} \sum_{i \neq j} \|x_{nij}\|_2^2\\
    &= |e^{\frac{2i\pi}{m}} - 1|^2 \|x_n - E_{\mathbb{M}_m \otimes R \bar{\otimes} N_1, \langle u_n \rangle'}(x_n)\|_2^2
\end{split}
\end{equation*}

Therefore, $\|x_n - E_{\mathbb{M}_m \otimes R \bar{\otimes} N_1, \langle u_n \rangle'}(x_n)\|_2 \leq \frac{1}{|e^{\frac{2i\pi}{m}} - 1|} \|[u_n, x_n]\|_2$. As $x$ and $u$ commute, the RHS converges to 0 as $n \rightarrow \omega$, so the LHS tends to 0 as well. This shows that we indeed have $(E_{\mathbb{M}_m \otimes R \bar{\otimes} N_1, \langle u_n \rangle'}(x_n))^\circ = (x_n)^\circ$.
\end{proof}

\begin{lemma}
Suppose $N$, $M$, $\hat{M}$, and $N_1$ are finite von Neumann algebras and $N_1$ is in addition a factor. Suppose further that we have tracial inclusions $N \subseteq M \subseteq \hat{M} \subseteq (R \bar{\otimes} N_1)^\omega$ where $R$ is the hyperfinite $\textrm{II}_1$ factor and $\omega$ is an ultrafilter. If there exists a Haar unitary $u \in \hat{M}$ s.t. the following diagram is a commuting square,

\begin{center}
\begin{tikzcd}[node distance = 1.8cm]
    {\hat{M}} \arrow[hookleftarrow]{r}\arrow[hookleftarrow]{d}
        & M \arrow[hookleftarrow]{d} \\
    {\langle u \rangle' \cap \hat{M}} \arrow[hookleftarrow]{r}
        & N
\end{tikzcd}
\end{center}

Then we have $N \subseteq M$ is RE/$N_1$.
\end{lemma}

\begin{proof} Since $u$ is a Haar unitary, the subalgebra of $(R \bar{\otimes} N_1)^\omega$ generated by $u$ is isomorphic to $L^\infty(S^1)$ where $S^1$ is the unit circle and $u$ is sent to the unitary $(x \mapsto x) \in L^\infty(S^1)$. The inherited trace on $\langle u \rangle$ corresponds to the Haar measure on $S^1$. By partitioning $S^1$ into $m$ equal pieces, we may approximate $u$ by unitaries $u_m$ of the form $\sum_{l=1}^m e^{2i\pi l/m} p_{ml}$, where $p_{ml}$ are orthogonal projections with trace $\frac{1}{m}$, for each fixed $m$. $u_m$ converges to $u$ in $L^2$. By Lemma 7.1, for each fixed $m$, we have an embedding $(R \bar{\otimes} N_1)^\omega \hookrightarrow (\mathbb{M}_m \otimes R \bar{\otimes} N_1)^\omega$. Combining such embeddings for all $m$ together, we obtain,
\begin{equation*}
    (R \bar{\otimes} N_1)^\omega \hookrightarrow \prod_{m \rightarrow \omega} (\mathbb{M}_m \otimes R \bar{\otimes} N_1)^\omega
\end{equation*}

For simplicity, we shall write $\widetilde{N^m} = (\mathbb{M}_m \otimes R \bar{\otimes} N_1)^\omega$, so $\prod_{m \rightarrow \omega} (\mathbb{M}_m \otimes R \bar{\otimes} N_1)^\omega = \prod_{m \rightarrow \omega} \widetilde{N^m}$. Now, since $u_m$ approximates $u$ in $L^2$, we see that under the composite embedding $\hat{M} \hookrightarrow (R \bar{\otimes} N_1)^\omega \hookrightarrow \prod_{m \rightarrow \omega} \widetilde{N^m}$, $u$ is sent to $(u_m)^\circ$. By Lemma 7.1, for each fixed $m$, we write $u_m \in \widetilde{N^m}$ as $(u_{mn})^\circ$. Given any $x \in \langle u \rangle' \cap \hat{M}$, since $u_m \in \langle u \rangle$, we have $x$ and $u_m$ commutes, so Lemma 7.1 shows that $x$ is sent into $\prod_{n \rightarrow \omega} \langle u_{mn} \rangle' \subseteq \widetilde{N^m}$. Therefore, the composite embedding $\hat{M} \hookrightarrow (R \bar{\otimes} N_1)^\omega \hookrightarrow \prod_{m \rightarrow \omega} \widetilde{N^m} = \hat{N_1}$ sends $x$ into $\prod_{m \rightarrow \omega} \prod_{n \rightarrow \omega} \langle u_{mn} \rangle' \subseteq \hat{N_1}$. Hence, we have a commutative diagram,

\begin{center}
\begin{tikzcd}[node distance = 1.8cm]
    {\prod_{m \rightarrow \omega} \prod_{n \rightarrow \omega} \mathbb{M}_m \otimes R \bar{\otimes} N_1} \arrow[hookleftarrow]{r}\arrow[hookleftarrow]{d}
        & {\hat{M}} \arrow[hookleftarrow]{d} \\
    {\prod_{m \rightarrow \omega} \prod_{n \rightarrow \omega} \langle u_{mn} \rangle'} \arrow[hookleftarrow]{r}
        & {\langle u \rangle' \cap \hat{M}}
\end{tikzcd}
\end{center}

To show that this is a commuting square, we note that as $u = (u_{mn})^\circ$, $\prod_{m \rightarrow \omega} \prod_{n \rightarrow \omega} \langle u_{mn} \rangle'$ is contained in $\langle u \rangle' \cap \hat{N_1}$. By an easy application of the ergodic theorem (see [Lan76, Theorem 5.7]),
\begin{equation*}
    E_{\hat{N_1}, \langle u \rangle' \cap \hat{N_1}}(x) = \lim_{n \rightarrow \infty} \frac{1}{n} \sum_{i=0}^{n-1} u^i x {u^*}^i
\end{equation*}
and the same formula holds for $E_{\hat{M}, \langle u \rangle' \cap \hat{M}}$, where the convergences are both under the strong$^*$ topology. It is then clear that we have the following commuting square,

\begin{center}
\begin{tikzcd}[node distance = 1.8cm]
    {\hat{N_1}} \arrow[hookleftarrow]{r}\arrow[hookleftarrow]{d}
        & {\hat{M}} \arrow[hookleftarrow]{d} \\
    {\langle u \rangle' \cap \hat{N_1}} \arrow[hookleftarrow]{r}
        & {\langle u \rangle' \cap \hat{M}}
\end{tikzcd}
\end{center}

Thus,
\begin{equation*}
\begin{split}
    E_{\hat{N_1}, \prod_{m \rightarrow \omega} \prod_{n \rightarrow \omega} \langle u_{mn} \rangle'}E_{\hat{N_1}, \hat{M}} &= E_{\hat{N_1}, \prod_{m \rightarrow \omega} \prod_{n \rightarrow \omega} \langle u_{mn} \rangle'}E_{\hat{N_1}, \langle u \rangle' \cap \hat{N_1}}E_{\hat{N_1}, \hat{M}}\\
    &= E_{\hat{N_1}, \prod_{m \rightarrow \omega} \prod_{n \rightarrow \omega} \langle u_{mn} \rangle'}E_{\hat{N_1}, \langle u \rangle' \cap \hat{M}}\\
    &= E_{\hat{N_1}, \langle u \rangle' \cap \hat{M}}
\end{split}
\end{equation*}

This shows that the first commutative diagram above is indeed a commuting square. Composing it with the commuting square in the assumption gives the following commuting square,

\begin{center}
\begin{tikzcd}[node distance = 1.8cm]
    {\prod_{m \rightarrow \omega} \prod_{n \rightarrow \omega} \mathbb{M}_m \otimes R \bar{\otimes} N_1} \arrow[hookleftarrow]{r}\arrow[hookleftarrow]{d}
        & M \arrow[hookleftarrow]{d} \\
    {\prod_{m \rightarrow \omega} \prod_{n \rightarrow \omega} \langle u_{mn} \rangle'} \arrow[hookleftarrow]{r}
        & N
\end{tikzcd}
\end{center}

To conclude the proof, we simply observe that $\langle u_{mn} \rangle' = l_\infty^m \otimes R \bar{\otimes} N_1$.
\end{proof}

\begin{col}
In the above lemma, assuming $u$ is a unitary without necessarily being a Haar unitary is sufficient.
\end{col}

\begin{proof} Consider the algebra $\hat{\hat{M}} = \hat{M} \bar{\otimes} L^\infty(S^1)$. The subalgebra $\langle u \rangle \bar{\otimes} L^\infty(S^1)$ is commutative and non-atomic, and as such isomorphic to $L^\infty(S^1)$. It is therefore generated by a single Haar unitary $v \in \hat{\hat{M}}$. We claim we have a commuting square,

\begin{center}
\begin{tikzcd}[node distance = 1.8cm]
    {\hat{\hat{M}}} \arrow[hookleftarrow]{r}\arrow[hookleftarrow]{d}
        & {\hat{M}} \arrow[hookleftarrow]{d} \\
    {\langle v \rangle' \cap \hat{\hat{M}}} \arrow[hookleftarrow]{r}
        & {\langle u \rangle' \cap \hat{M}}
\end{tikzcd}
\end{center}

Since $\hat{\hat{M}} = \hat{M} \bar{\otimes} L^\infty(S^1)$ and $\langle v \rangle = \langle u \rangle \bar{\otimes} L^\infty(S^1)$, $\langle v \rangle' \cap \hat{\hat{M}}$ is simply $(\langle u \rangle' \cap \hat{M}) \bar{\otimes} L^\infty(S^1)$. Therefore, the above diagram is none other than,

\begin{center}
\begin{tikzcd}[node distance = 1.8cm]
    {\hat{M} \bar{\otimes} L^\infty(S^1)} \arrow[hookleftarrow]{r}\arrow[hookleftarrow]{d}
        & {\hat{M}} \arrow[hookleftarrow]{d} \\
    {(\langle u \rangle' \cap \hat{M}) \bar{\otimes} L^\infty(S^1)} \arrow[hookleftarrow]{r}
        & {\langle u \rangle' \cap \hat{M}}
\end{tikzcd}
\end{center}
which is clearly a commuting square. Composing it with the commuting square in the assumption gives the following commuting square,

\begin{center}
\begin{tikzcd}[node distance = 1.8cm]
    {\hat{\hat{M}}} \arrow[hookleftarrow]{r}\arrow[hookleftarrow]{d}
        & M \arrow[hookleftarrow]{d} \\
    {\langle v \rangle' \cap \hat{\hat{M}}} \arrow[hookleftarrow]{r}
        & N
\end{tikzcd}
\end{center}

We observe that as $\hat{M}$ embeds into $(R \bar{\otimes} N_1)^\omega$, $\hat{\hat{M}} = \hat{M} \bar{\otimes} L^\infty(S^1)$ naturally embeds into $(R \bar{\otimes} L^\infty(S^1) \bar{\otimes} N_1)^\omega$. Since $R \bar{\otimes} L^\infty(S^1)$ embeds into $R$, we have an embedding $\hat{\hat{M}} \hookrightarrow (R \bar{\otimes} N_1)^\omega$. Applying the lemma above concludes the proof.
\end{proof}

\begin{proof}[Proof of $2 \Rightarrow 1$ and $3 \Rightarrow 1$ of Theorem 7.1] In both cases, we will construct a commuting square as in Corollary 7.1. For $2 \Rightarrow 1$, we let $\hat{M}$ be $(M \bar{\otimes} L^\infty(S^1)) \ast_N M$ and for $3 \Rightarrow 1$ we let $\hat{M}$ be $M \ast_N M$. The inclusion $M \hookrightarrow \hat{M}$ is always the map that sends $M$ to the second component. In the case of $2 \Rightarrow 1$, let $u$ be the unitary $(x \mapsto x) \in L^\infty(S^1) \subseteq \hat{M}$. In the case of $3 \Rightarrow 1$, by assumption we have a unitary $u \in M$ s.t. $E_{M, N}(u^n) = 0$ for all $n \neq 0$. We shall simply choose this $u$, regarded as an element of the first component of $\hat{M} = M \ast_N M$.

Now, we need to verify the following commutative diagram is a commuting square,

\begin{center}
\begin{tikzcd}[node distance = 1.8cm]
    {\hat{M}} \arrow[hookleftarrow]{r}\arrow[hookleftarrow]{d}
        & M \arrow[hookleftarrow]{d} \\
    {\langle u \rangle' \cap \hat{M}} \arrow[hookleftarrow]{r}
        & N
\end{tikzcd}
\end{center}

Again by the ergodic theorem, $E_{\hat{M}, \langle u \rangle' \cap \hat{M}}(x) = \lim_{n \rightarrow \infty} \frac{1}{n} \sum_{i=0}^{n-1} u^i x {u^*}^i$ where the convergence is strong$^*$. Given any $x \in M$ with $E_{M, N}(x) = 0$, we need to show $E_{\hat{M}, \langle u \rangle' \cap \hat{M}}(x) = 0$. We have,
\begin{equation*}
\begin{split}
    \|E_{\hat{M}, \langle u \rangle' \cap \hat{M}}(x)\|_2^2 &= \tau(E_{\hat{M}, \langle u \rangle' \cap \hat{M}}(x)^*E_{\hat{M}, \langle u \rangle' \cap \hat{M}}(x))\\
    &= \lim_{n \rightarrow \infty} \frac{1}{n^2} \sum_{i,j=0}^{n-1} \tau(u^ix^*u^{j-i}xu^{-j})
\end{split}
\end{equation*}

Note here that $u$ and $x \in M$ belong to different components of $\hat{M} = (M \bar{\otimes} L^\infty(S^1)) \ast_N M$ or $M \ast_N M$. We have $E_{M, N}(x) = E_{M, N}(x^*) = 0$. Also, $E_{M \bar{\otimes} L^\infty(S^1), N}(u^n) = 0$ for all $n \neq 0$ in the $2 \Rightarrow 1$ case and $E_{M, N}(u^n) = 0$ for all $n \neq 0$ in the $3 \Rightarrow 1$ case. Hence, by free independence over $N$, $\tau(u^ix^*u^{j-i}xu^{-j}) = 0$ whenever $i \neq j$, so,
\begin{equation*}
\begin{split}
    \|E_{\hat{M}, \langle u \rangle' \cap \hat{M}}(x)\|_2^2 &= \lim_{n \rightarrow \infty} \frac{1}{n^2} \sum_{i=0}^{n-1} \tau(u^ix^*xu^{-i})\\
    &= \lim_{n \rightarrow \infty} \frac{1}{n^2} \sum_{i=0}^{n-1} \tau(x^*x)\\
    &= \lim_{n \rightarrow \infty} \frac{\tau(x^*x)}{n}\\
    &= 0
\end{split}
\end{equation*}

This proves the diagram before is indeed a commuting square. By our assumption $\hat{M} \subseteq (R \bar{\otimes} N_1)^\omega$, so the result follows from Corollary 7.1.
\end{proof}

\newpage

\begin{center}
    \textsc{References}
\end{center}

\medskip

\leftskip 0.25in
\parindent -0.25in
[BO08] Brown, N.P. and Ozawa, N. \textit{$C^*$-algebras and finite-dimensional approximations}. American Mathematical Society, 2008.

\smallskip

[Con76] Connes, A. \textit{Classification of injective factors: Cases} $\textrm{II}_1$, $\textrm{II}_\infty$, $\textrm{III}_\lambda$, $\lambda \neq 1$. Ann. Math. 104 (1976), pp. 73-115.

\smallskip

[Hall49] Hall, M. \textit{Coset representations in free groups}. Tran. Am. Math. Soc. 67 (1949), pp. 421-432.

\smallskip

[Hat05] Hatcher, A. \textit{Algebraic topology}. Cambridge University Press, 2005.

\smallskip

[Ioa11] Ioana, A. \textit{Cocycle superrigidity for profinite actions of property (T) groups}. Duke. Math. J. 157 (2011), pp. 337-367.

\smallskip

[JNV+20] Ji, Z., Natarajan, A., Vidick, T., Wright, J., and Yuen, H. \textit{$\textrm{MIP}^* = \textrm{RE}$}. arXiv preprint arXiv: 2001.04383, 2020.

\smallskip

[Jun05] Junge, M. \textit{Embedding of the operator space OH and the
logarithmic `little Grothendieck inequality'}. Invent. math. 161 (2005), pp. 225-286.

\smallskip

[Jun21] Junge, M. \textit{Noncommutative Poisson random measure}. Preprint.

\smallskip

[KT08] Kechris, A.S. and Tsankov, T. \textit{Amenable actions and almost invariant sets}. Proc. Am. Math. Soc. 136 (2008), pp. 687-697.

\smallskip

[Lan76] Lance, E.C. \textit{Ergodic theorems for convex sets and operator algebras}. Invent. math. 37 (1976), pp. 201-214.

\smallskip

[LR08] Long, D.D. and Reid, A.W. \textit{Subgroup separability and virtual retractions of groups}. Topology. 47 (2008), pp. 137-159.

\smallskip

[MKS76] Magnus, W., Karrass, A., and Solitar, D. \textit{Combinatorial group theory}. Dover Publications, 1976.

\smallskip

[MP03] Monod, N. and Popa, S. \textit{On co-amenability for groups and von Neumann algebras}. arXiv preprint arXiv: math/0301348v3, 2003.

\smallskip

[Tak03] Takesaki, M. \textit{Theory of operator algebras III}. Springer, 2003.

\smallskip

[VDN92] Voiculescu, D.V., Dykema, K.J., and Nica, A. \textit{Free random variables}. American Mathematical Society, 1992.

\bigskip\bigskip

\textsc{Department of Mathematics, University of Illinois, Urbana, IL 61801, USA}

\textit{E-mail address}, Weichen Gao: \texttt{wg6@illinois.edu}

\bigskip

\textsc{Department of Mathematics, University of Illinois, Urbana, IL 61801, USA}

\textit{E-mail address}, Marius Junge: \texttt{mjunge@illinois.edu}

\end{document}